\documentclass[runningheads]{llncs}

\usepackage[T1]{fontenc}
\usepackage[colorinlistoftodos,bordercolor=orange,backgroundcolor=orange!20,linecolor=orange,textsize=scriptsize,disable]{todonotes}
\usepackage[inline]{enumitem}

\usepackage{bbm}
\usepackage{amsfonts,amsmath,mathtools}
\usepackage{mathrsfs}
\usepackage{cleveref}

\usetikzlibrary{calc}
\usetikzlibrary{shapes.geometric}
\usetikzlibrary{automata, positioning}
\usetikzlibrary{arrows}
\usetikzlibrary{arrows.meta}
\usepackage{tikz-3dplot}

\usepackage{blkarray} %

\tikzset{place/.style={draw,circle,inner sep=2.8pt,semithick}}  
\tikzset{transition/.style={rectangle, thick,fill=black, minimum width=8.4pt,inner ysep=0.5pt}}
\tikzset{transitionV/.style={rectangle, thick,fill=black, minimum height=8.4pt,inner xsep=0.5pt}} 
\tikzset{jeton/.style={draw,circle,fill=black!80,inner sep=.35pt}}
\tikzset{pre/.style={=stealth'}}
\tikzset{post/.style={->,shorten >=1pt,>=stealth'}}
\tikzset{-|/.style={to path={-| (\tikztotarget)}}, |-/.style={to path={|- (\tikztotarget)}}}
\tikzset{bicolor/.style 2 args={dashed,dash pattern=on 1pt off 1pt,#1, postaction={draw,dashed,dash pattern=on 1pt off 1pt,#2,dash phase=1pt}}}   
\tikzset{arrowPetri/.style={>=latex,rounded corners=5pt,semithick}}

\newcommand{\R}{\mathbb{R}}
\newcommand{\Fbb}{\mathbb{F}}
\newcommand{\RS}{\R^{|S|}}
\newcommand{\N}{\mathbb{N}}
\newcommand{\Sg}{\mathscr{S}}
\newcommand{\Pcal}{\mathcal{P}}
\newcommand{\Vcal}{\mathcal{V}}
\newcommand{\tmin}{t_{\min}}
\newcommand{\tmax}{t_{\max}}
\newcommand{\SSigma}{\mathfrak{S}}
\newcommand{\unbm}{\mathbbm{1}}
\newcommand{\unbf}{\mathbf{1}}
\newcommand{\zerobf}{\mathbf{0}}
\newcommand{\Kbox}{\scalebox{.72}{$[K]$}}

\newcommand{\vsb}{v^{\sbullet}}
\newcommand{\tInterval}{[-\tmax, 0)}
\newcommand{\CFC}[1]{\mathcal{C}(#1)}

\newcommand{\ie}{i.e.}
\newcommand{\eg}{e.g.}
\newcommand{\SminusZ}{S\setminus\{0\}}
\newcommand{\supp}{\mathrm{supp}}

\newcommand\sbullet[1][.66]{\mathbin{\vcenter{\hbox{\scalebox{#1}{$\bullet$}}}}}

\newcommand{\textforall}{\text{for all}\;\,}
\newcommand{\textin}{\;\,\text{in}\;\,}

\newcommand{\udensdot}[1]{%
    \tikz[baseline=(todotted.base)]{
        \node[inner sep=1pt,outer sep=0pt] (todotted) {#1};
        \draw[densely dotted, opacity =.8] (todotted.south west) -- (todotted.south east);
    }%
}%

\definecolor{hreflinkcolor}{HTML}{3f87d3}
\definecolor{hrefcitecolor}{HTML}{ee7c27}
\definecolor{hrefurlcolor}{rgb}{1,0.5,0}

\newtheorem{assumption}{Assumption}

\begin{document}

\title{Computing Transience Bounds of Emergency Call Centers: a Hierarchical Timed Petri Net Approach}
\titlerunning{Computing Transience Bounds of Emergency Call Centers}

\author{Xavier Allamigeon \and Marin Boyet \and St\'ephane Gaubert}
\authorrunning{X. Allamigeon, M. Boyet and S. Gaubert}
\institute{INRIA and CMAP, \'Ecole polytechnique, IP Paris, CNRS\\
  \email{Firstname.Name@inria.fr}}
\maketitle

\begin{abstract}
  A fundamental issue in the analysis of emergency call centers is to estimate the time needed to return to a congestion-free regime after an unusual event with a massive arrival of calls. Call centers can generally be represented by timed Petri nets with a hierarchical structure, in which several layers describe the successive steps of treatments of calls. We study a continuous approximation of the Petri net dynamics (with infinitesimal tokens). Then, we show that a counter function, measuring the deviation to the stationary regime, coincides with the value function of a semi-Markov decision
  problem. Then, we establish a finite time convergence result, exploiting the hierarchical
  structure of the Petri net. We obtain an explicit bound for the transience time,
  as a function of the initial marking and sojourn times.
  This is based on methods from the theory of stochastic shortest paths and
  non-linear Perron--Frobenius theory.  
  We illustrate the bound on a case study of a medical emergency call center.
\end{abstract}
\renewcommand{\thefootnote}{}
\footnote{Version of January 28, 2022.}

\section{Introduction}

\paragraph{Context} The handling of emergency calls in dedicated call centers features various concurrency and synchronization patterns, breaking down in a stepwise process with tasks involving multiple agents, to be performed with as little delay as possible.
A fundamental question is to fix the staffing so that calls be swiftly handled, taking into account the customary demand, as well as scenarios allowing sudden bulk of calls, originating for instance from exceptional events. 
In particular, one needs to compute the time needed to absorb such a bulk of calls, depending on the center characteristics and on the number of agents of various types. The treatment of incoming calls may be delayed
during the transient phase, and so, the duration
of this phase appears to be a critical performance measure.

Timed Petri nets have been used
in~\cite{allamigeon2015performance,allamigeon2020piecewise}
to model emergency call centers.
Along the lines of~\cite{CGQ95b,allamigeon2020piecewise},
timed Petri nets can be modelled by counter variables
that give the number of firings of the transitions as a function of time.
The discrete dynamics is generally hard to analyze, and so,
a continuous approximation, with infinitesimal firings
and stationnary routings, has been studied there.
Then, the evolution of the counter
variables is determined by a dynamic programming
equation of semi-Markov type. In this way, first order performance measures like the long run throughput, were computed analytically as a function of resources~\cite{allamigeon2020piecewise}. However, finer key performance evaluation issues, like the understanding of the transient behavior --- in particular the estimation of the ``catch-up time'', \ie, the time needed to return to a stationary regime --- have not been addressed so far in this setting. 

\paragraph{Contribution}
We consider a class of continuous timed Petri nets with a single input.
In a reference scenario, this input is an affine function of time; this represents a regular arrival of calls in our application. Then, we study the behavior of the dynamical system under a deviation from this scenario, induced for instance by a bulk of arrivals.

Our main result shows that when the system is sufficiently staffed,
meaning that the intrinsic throughput of the Petri net exceeds
the input flow, the Petri net trajectory ultimately catches up a stationary regime driven by the input, see~\Cref{thm:SSSP}.
This result is obtained by showing that the deviation to the stationary regime coincides with the value function of a stochastic shortest path problem (SSP),
see~\Cref{thm:normalization}.
In this way, we are reduced to quantifying the convergence time of SSP problems. In \Cref{thm:finite_time_theoretical}, we characterize the SSP configurations for which the convergence occurs in finite time. This characterization involves the existence of a partial order over states, such that all optimal policies make moves that decrease the order. 
Then, we consider a hierarchical class of timed Petri nets, for which this partial order is known \textit{a priori} --- in the application to call centers, this corresponds to the natural ordering of tasks in the chain of treatment. In~\Cref{thm:bound_SSP}, we obtain an explicit upper bound on the catch-up times (or transience times) to recover from a perturbation. We illustrate our results on the emergency call center application throughout the paper. \Cref{sec:intro_aux_SMDP} provides the needed background and tools on semi-Markov decision processes; a novelty here, of a somewhat technical nature, is to handle the case of instantaneous transitions, which arises in our applications. 

\paragraph{Related work}
The question of convergence of the earliest behavior of timed
event graphs to a periodic or stationary regine has received
much attention in the discrete event systems literature,
see~\cite{baccelli1992sync} (especially Th.~3.109), and also
~\cite[Ch.~8 and 9]{HOW:05}, with an application to the Dutch
railway network.
For timed event graphs, the duration of the transient behavior,
also sometimes called ``coupling time'',
has been extensively studied, in particular by techniques of max-plus spectral theory, see~\cite{BG01,sergei,merlet2014weak,merlet2021}. The same problem
has arisen in the setting of deterministic
dynamic programming~\cite{hartmann},
and in the analysis of distributed algorithms~\cite{bernadette},

Here, we extend the bounds given in the above references,
passing from timed event graphs to continuous Petri nets with
stationary routings, or, equivalently, passing
from deterministic (semi-)Markov decision processes
to stochastic semi-Markov decision processes.
By comparison with the ``deterministic'' case,
our proofs require new tools, coming from non-linear
Perron-Frobenius theory and from the theory of stochastic shortest
paths. In particular, \Cref{thm:SSSP} on the asymptotic convergence time builds on the work of
~\cite{bertsekas1991analysis} on stochastic shortest path problems, and extends some of these results to the semi-Markov case. \Cref{thm:finite_time_theoretical} exploits techniques of non-linear Perron--Frobenius theory~\cite{akian2011collatz,akian2019solving} in order to characterize the property of convergence in finite time. To our knowledge, no characterization of the finite time convergence was known, even in the setting of Markov decisions processes.

Estimating the speed of convergence to the stationary regime
for Markov decision processes is indeed a difficult and classical
issue. General asymptotic convergence results, like the ones of
~\cite{Schweitzer1978a,SchweitzerFedergruenEquation,schweitzer79}, show that a convergence does occur with a ultimately geometric rate. However,  they lead to bounds and speed estimates that are nonconstructive. An explicit bound of the time needed to enter in the geometric convergence regime was given
in~\cite{bonet2007speed}, for shortest path stochastic configurations, supposing that all costs are positive. In contrast, we consider here the property of {\em finite time} convergence, leading to different bounds.

We rely on a correspondence between continuous timed Petri nets and
semi-Markov decisions processes, 
developed in a series of works~\cite{CGQ95a,CGQ95b,allamigeon2020piecewise}.
This allows one to obtain asymptotic theorems and develop computational methods
to solve performance evaluation issues by means of Markov decision
techniques. Related analytical results were obtained in~\cite{gaujal2004optimal} with a different approach.
Alternative approaches to the question of ``absorption time''
may rely on stochastic models, or network calculus. We leave
the treatment of such aspects for further work, noting
that the continuous Petri net behaviors studied arise
as scaling limits of discrete deterministic or stochastic models.
Hence, capturing probabilistic or network calculus aspects is expected
to add one layer of difficulty -- see e.g.~\cite{boeufrobert} for a probabilistic treatment of emergency call centers.

Finally, we refer the reader to~\cite{yushkevich1982semi,feinberg1994constrained,ross1970average} for background
on semi-Markov decision processes, a.k.a, renewal programs~\cite{Denardo-Fox68}. See also~\cite{huang2011finite} for a recent reference.

The proofs of our main results are presented in appendix; the essential tools and ideas, however, are in the body of the paper. Additional explanatory or illustrative materials can be found in appendix. 

\section{Petri net model of a medical emergency call center}
A specific motivation in computing transience bounds for timed Petri nets originates from a real-life case study of the medical emergency call centers of the Paris area (SAMU 75, 92, 93 and 94), that is discussed in details in~\cite{allamigeon2020piecewise}.%

We consider a medical emergency call center with three types of agents: \emph{medical regulation assistants} (MRAs) who pick up calls and orient patients through the system, \emph{emergency physicians} who handle the calls deemed to be the most serious by the MRAs, after which it can be decided to send an ambulance or a mobile intensive care unit, and \emph{general practicioners}.
Once a MRA has detected that a patient should talk with some emergency physician, the former stays on line with the patient and waits for one of the latter to be available; at this point a very short conversation happens between the MRA and the doctor, in which the MRA summarizes
the case. Calls not passed to emergency physicians by the MRAs are either transfered to general practitioners if they concern less serious medical matters, and in this case the patient is put on hold in a virtual waiting room, or these calls are hang up if it was not a health-related distress call. In the last two situations, the MRA is able to pick up immediately a new incoming call.
We  focus on the next fundamental tasks of the emergency chain
of treatment.

\setlist[enumerate,1]{start=0}
    \begin{enumerate}[label=\textsc{Task} \arabic*:, leftmargin = 1.1cm]
    \item an emergency inbound call arrives
    \item an inbound call is picked up by a medical regulation assistant, who will decide if the call should be passed on to an emergency physician or not
    \item the instruction by the MRA of a call not requiring to talk with the emergency physician is completed
    \item the instruction by the MRA of a call requiring to talk with the emergency physician is completed \textbf{and} communication with the latter is initiated
    \item the short briefing between the MRA and the emergency physician is over, phone consultation between the patient and the physician starts 
    \item the consultation of the patient with the physician ends
    \end{enumerate}
  \setlist[enumerate,1]{start=1}

  Depending on the need to talk with an emergency physician or not, calls will follow \textsc{Tasks} $0$, $1$ and $2$ or \textsc{Tasks} $0$, $1$, $3$, $4$ and $5$. For $t\geq 0$, we denote by $z_i(t)$ the number of \textsc{Tasks} $i$ completed up to time $t$, starting with no completed tasks at the instant $t=0$.

  We assume that calls arrive at a rate $\lambda$, and that a fraction $\pi$ of them will require a discussion with an emergency physician. There are a total of $N_A$ MRAs and $N_P$ emergency doctors. The conversation between the patient and the MRA takes a time $t_1$, the synchronization step between an MRA and a physician takes a time $t_2$, and the consultation between the physician and the patient consumes a time $t_3$ (all these durations are assumed constant for sake of simplicity).

  \begin{figure}
    \begin{center}
      \vspace{-.76cm}
      \begin{tabular}{ccc}\def\tkzscl{.27}
\definecolor{colorARM}{rgb}{0,0,0}
\definecolor{colorARMres}{rgb}{0,0,0}
\definecolor{colorAMU}{rgb}{0,0,0}
\definecolor{colorexit}{rgb}{0,0,0}

\tikzset{place/.style={draw,circle,inner sep=2.5pt,semithick}}  
\tikzset{transition/.style={rectangle, thick,fill=black, minimum width=2mm,inner ysep=0.5pt, minimum width=2mm}}
\tikzset{jeton/.style={draw,circle,fill=black!80,inner sep=.35pt}}
\tikzset{pre/.style={=stealth'}}
\tikzset{post/.style={->,shorten >=1pt,>=stealth'}}
\tikzset{-|/.style={to path={-| (\tikztotarget)}}, |-/.style={to path={|- (\tikztotarget)}}}
\tikzset{bicolor/.style 2 args={dashed,dash pattern=on 1pt off 1pt,#1, postaction={draw,dashed,dash pattern=on 1pt off 1pt,#2,dash phase=1pt}}}   
\tikzset{arrowPetri/.style={>=latex,rounded corners=5pt}}

\begin{tikzpicture}[scale=\tkzscl,font=\tiny]

\def\p{2.2}

\begin{scope}[shift={(0,4)}]
\node[place] (pool_arm) at ($(-2*\p,-1.5*\p)$) {};
\node (txt_Na) at ($(pool_arm)+(-1.3,0)$) {$N_A$};    

\node[transition]    (q_arrivals)                          at    (0, 0) {};
\node[place]         (p_inc_calls)                         at    ($(q_arrivals) + (0, \p)$)  {};
\node[transition]    (q_inc_calls)                         at    ($(p_inc_calls) + (0, \p)$)  {};
\node[place]         (p_arrivals)        at    ($(q_arrivals) + (0,-\p)$) {};
\node[place]         (p_arrivals2)       at    ($(q_arrivals) + (2*\p,-\p)$) {};
\node[transition]    (q_debut_NFU)       at    ($(p_arrivals) + (0,-\p)$) {};
\node[transition]    (q_debut_AMU)         at    ($(p_arrivals) + (2*\p,-\p)$) {};

\draw[->,arrowPetri,colorARM] (q_arrivals)  |- ($(q_arrivals)+(.5*\p,-.3*\p)$) -| (p_arrivals2);
\draw[->,arrowPetri,colorARM] (q_debut_NFU) |- ($(q_debut_NFU)+(0,-.5*\p)$) -| (pool_arm);
\draw[->,arrowPetri,colorARM] (p_arrivals)  -- (q_debut_NFU);    
\draw[->,arrowPetri,colorARM] (q_arrivals) -- (p_arrivals);
\draw[->,arrowPetri,colorARM] (p_arrivals2) -- (q_debut_AMU);

\node[place]         (p_synchro)           at    ($(q_debut_AMU)  + (0,-\p)$) {};
\node[transition]    (q_unsynchro)         at    ($(p_synchro)  + (0,-\p)$) {};

\draw[arrowPetri]    (q_inc_calls) |- ($(q_inc_calls)+(-.35*\p,.25*\p)$);
\draw[dashed, arrowPetri]    (q_inc_calls) |- ($(q_inc_calls)+(-1.5*\p,.25*\p)$);
\draw[->,arrowPetri] (q_inc_calls) -- (p_inc_calls);
\draw[->,arrowPetri]                                    (p_inc_calls) -- (q_arrivals);
\draw[->,arrowPetri,colorARM]                           (q_unsynchro) -- ($(q_unsynchro)+(0,-.5*\p)$)  -|  (pool_arm);

\node[place]    (p_consult_AMU)    at    ($(q_unsynchro)+(0,-\p)$) {};

\draw[->,arrowPetri,bicolor={colorARM}{colorAMU}]    (q_debut_AMU) -- (p_synchro);
\draw[->,arrowPetri,bicolor={colorARM}{colorAMU}]    (p_synchro) -- (q_unsynchro);

\draw[->,arrowPetri,colorARM]                           (pool_arm)    |- ($(q_arrivals)+(-.5*\p,.5*\p)$) -- (q_arrivals);

\node (txt_taus) at ($(p_synchro)+(.5*\p,0)$) {$\tau_2$};
\node (txt_tau1) at ($(p_arrivals.center)+(.5*\p,0)$) {${\tau}_1$};
\node (txt_tau1) at ($(p_arrivals2.center)+(.5*\p,0)$) {${\tau}_1$};
\node (txt_tau3) at ($(p_consult_AMU.center)+(.5*\p,0)$) {${\tau}_3$};
\node (txt_pi)  at ($(p_arrivals.center)+(-1.4,1.65)$) {$1\!-\!\pi$};
\node (txt_pi3) at ($(p_arrivals.center)+(4.6, 1.65)$) {$\pi$};

\node[place]      (pool_amu_old)    at    ($(q_unsynchro)+(2*\p,.5*\p)$) {};
\node (txt_Nm) at ($(pool_amu_old)+(1.3,0)$) {$N_P$};    

\node[transition] (q_end_consult_AMU)                  at    ($(p_consult_AMU)+(0,-\p)$) {};      
\draw[->,arrowPetri,colorAMU]    (q_unsynchro)       -- (p_consult_AMU);
\draw[->,arrowPetri,colorAMU]    (p_consult_AMU)     -- (q_end_consult_AMU);
\draw[->,arrowPetri,colorAMU]    (q_end_consult_AMU) -- ($(q_end_consult_AMU)+(0,-1)$) -| (pool_amu_old);
\draw[->,arrowPetri,colorAMU]    (pool_amu_old)      |- ($(q_debut_AMU)+(.5*\p,.5*\p)$)      -- (q_debut_AMU);

\node (txt_z0) at ($(q_inc_calls)+(1*\p,0)$) {$z_0 = \lambda t$};
\node (txt_z1) at ($(q_arrivals)+(.5*\p,0)$) {$z_1$};
\node (txt_z2) at ($(q_debut_NFU)+(.5*\p,0)$) {$z_2$};
\node (txt_z3) at ($(q_debut_AMU)+(.5*\p,0)$) {$z_3$};
\node (txt_z4) at ($(q_unsynchro)+(.5*\p,0)$) {$z_4$};
\node (txt_z5) at ($(q_end_consult_AMU)+(.5*\p,0)$) {$z_5$};

\end{scope}

\end{tikzpicture}	
        & \hspace{.1cm} & \def\tkzscl{.55}\definecolor{orangeDIY}{rgb}{.92,.5,.11}
\begin{tikzpicture}[auto,node distance=8mm,>=latex,scale=\tkzscl]
  \tikzset{arrowPetri/.style={>=latex,rounded corners=5pt}}

    \tikzstyle{round}=[thick,draw=black,circle,inner sep=2pt]

    \def\p{2.2}

    \node[round] (s0) at (0,0)              {\scriptsize$0$};
    \node[round] (s1) at ($(s0)+(0,-\p)$) {\scriptsize$1$};
    \node[round] (s2) at ($(s1)+(0,-\p)$)   {\scriptsize$2$};
    \node[round] (s3) at ($(s1)+(\p,-\p)$)  {\scriptsize$3$};
    \node[round] (s4) at ($(s3)+(0,-\p)$)   {\scriptsize$4$};
    \node[round] (s5) at ($(s4)+(0,-\p)$)   {\scriptsize$5$};

   \coordinate (a0)  at ($(s0)+(0,.5*\p)$) {};  
   \coordinate (a11) at ($(s1)+(0,.5*\p)$) {};  
   \coordinate (a12) at ($(s1)+(-.5*\p,0)$) {};  
   \coordinate (a2)  at ($(s2)+(0,.5*\p)$) {};  
   \coordinate (a31) at ($(s3)+(0,.5*\p)$) {};  
   \coordinate (a32) at ($(s3)+(.5*\p,0)$) {};  
   \coordinate (a4)  at ($(s4)+(0,.5*\p)$) {};  
   \coordinate (a5)  at ($(s5)+(0,.5*\p)$) {};  

    \node (lbl_a0)  at ($(a0) +(.75,0)$) {\tiny$(\lambda,1)$};
    \node (lbl_a11) at ($(a11)+(.75,0)$) {\tiny$(0,0)$};
    \node (lbl_a12) at ($(a12)+(-.25,.4)$) {\tiny$(N_A,0)$};
    \node (lbl_a2)  at ($(a2) +(.85,0.2)$) {\tiny$(0,\tau_1)$};
    \node (lbl_a31) at ($(a31)+(.85,0.2)$) {\tiny$(0,\tau_1)$};
    \node (lbl_a32) at ($(a32)+(0.3,0.45)$) {\tiny$(N_P/\pi,0)$};
    \node (lbl_a4)  at ($(a4) +(-.9,0)$) {\tiny$(0,\tau_2)$};
    \node (lbl_a5)  at ($(a5) +(-.9,0)$) {\tiny$(0,\tau_3)$};
    \node (lbl_pi1) at ($(s2) +(-1.1,-1*\p+.3)$) {\tiny$\pi$};
    \node (lbl_pi1) at ($(s2) +(-1,.3)$) {\tiny$1\!-\!\pi$};

     \draw[-{Latex[scale=1]}, rounded corners=5pt] (a0) |- ($(a0)+(-.5*\p-.5,.2)$) |- (s0);
     \draw[-{Latex[scale=1]}, rounded corners=5pt] (a11) -- (s0);
     \draw[-{Latex[scale=1]}, rounded corners=5pt] (a12) -| ($(a12)+(-.5,-1)$) |- (s2);
     \draw[-{Latex[scale=1]}, rounded corners=5pt] (a12) -| ($(a12)+(-.5,-1)$) |- (s4);
     \draw[-{Latex[scale=1]}, rounded corners=5pt] (a31) |- (s1);
     \draw[-{Latex[scale=1]}, rounded corners=5pt] (a2) -- (s1);
     \draw[-{Latex[scale=1]}, rounded corners=5pt] (a4) -- (s3);
     \draw[-{Latex[scale=1]}, rounded corners=5pt] (a5) -- (s4);
     \draw[-{Latex[scale=1]}, rounded corners=5pt] (a32) -| ($(s5)+(.66*\p+.4,0)$) -- (s5);

   \draw[-{Square[scale=1.5, open, fill=white]}, semithick] (s0) -- (a0);
   \draw[-{Square[scale=1.5, open, fill=white]}, semithick] (s1) -- (a11);
   \draw[-{Square[scale=1.5, open, fill=white]}, semithick] (s1) -- (a12);
   \draw[-{Square[scale=1.5, open, fill=white]}, semithick] (s2) -- (a2);
   \draw[-{Square[scale=1.5, open, fill=white]}, semithick] (s3) -- (a31);
   \draw[-{Square[scale=1.5, open, fill=white]}, semithick] (s3) -- (a32);
   \draw[-{Square[scale=1.5, open, fill=white]}, semithick] (s4) -- (a4);
   \draw[-{Square[scale=1.5, open, fill=white]}, semithick] (s5) -- (a5);
\end{tikzpicture} \end{tabular}
    \vspace{-.3cm}
    \end{center}
    \caption{Petri net representing an emergency call center~\cite{allamigeon2020piecewise} (left).%
      The corresponding undiscounted SMDP (right), see \S\ref{sec:intro_aux_SMDP} for more information:
     states (resp.\ actions) are depicted by circles (resp.\ squares). A pair of the form (cost, sojourn time) is attached to each action. Transition probabilities from actions to states are given along the arcs if not equal to one.}
    \label{fig:SMDP}
      \label{fig:EMS_A_SMDP}
    \label{fig-emsA}
  \end{figure}
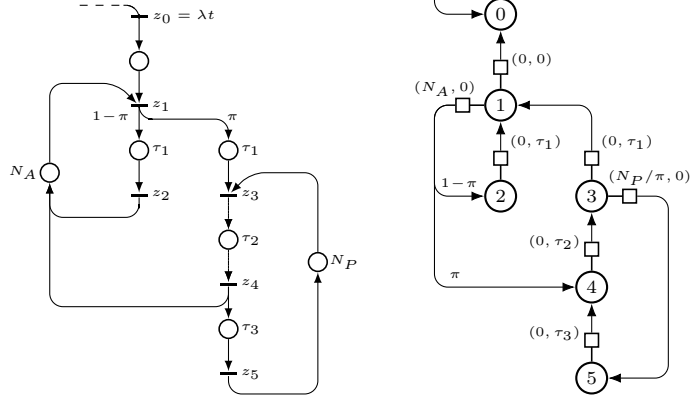
  In~\cite{allamigeon2020piecewise}, we modeled the previous organization by a continuous timed Petri net~\cite{CGQ95b}, shown on~\Cref{fig:EMS_A_SMDP}. So $z_i(t)$
  represents the number of firings of the transition labeled $z_i$, up
  to time $t$. We use the approximation in~\cite{allamigeon2020piecewise},
  allowing ``infinitesimal firings'', so that $z_i(t)$ is a real nonnegative
  number, rather than an integer. 
  In this setting, the functions~$z_i$ are governed by the following dynamics:
  \begin{equation}\hspace{-.2cm}\left\{
    \begin{array}{r@{\!\;\;}c@{\!\;\;}r@{\!\;\;}c@{\!\;\;}l}
    z_1(t) & = & z_0(t)\;\; & \wedge & \big(N_{A} + z_2(t) + z_4(t)\big)\\[1ex]
    z_2(t) & = & (1-\pi)z_1(t-t_1) & &\\[1ex]
    z_3(t) & = & \pi  z_1(t-t_1) &\wedge& \big(N_{P} + z_5(t)\big)\\[1ex]
    z_4(t) & = & z_3(t-t_2) & &\\[1ex] 
    z_5(t) & = & z_4(t-t_3)& &
    \end{array}\right.\tag{EMS}
    \label{SAMU1}
  \end{equation}
  where $\wedge$ stands for the minimum operation. A minimum of several terms arises if a task requires the synchronization of multiple resources (e.g., an inbound call and the availability of a MRA, or availability of both a MRA and a doctor). Delayed terms (of the form $z_i(t - \tau)$) are induced by the completion time of tasks. Weights (\eg, $\pi$ and $1-\pi$) occur when calls are split in several categories.

\section{The Equivalence between Semi-Markov Decision Processes and Continuous Time Petri nets}
\label{sec:intro_aux_SMDP}
In this section, we recall the correspondence between the dynamics of continuous
timed Petri nets with stationary routings and semi-Markov decision processes,
developed in~\cite{allamigeon2020piecewise}.
\subsection{The SMDP model}
\label{ssec:well-posedness}
\label{ssec:computing_value}
\todo[inline]{SG: expliquer continuous}
In {\em Semi-Markov Decision Processes} (SMDPs)~\cite{puterman2014markov,yushkevich1982semi}, sometimes referred to as Markov renewal programs, the decision epochs are allowed to occur at any real-valued time. Indeed, between two successive moves, a \emph{sojourn time} attached to states and actions must elapse. 

Here, the finite set of states is denoted by $S$, and for all $i\in S$ the finite set of playable actions from state $i$ is denoted by $A_i$. We denote $A \coloneqq \biguplus_{i\in S}A_i$ the disjoint union of the $(A_i)_{i\in S}$.
 As a result of playing action $a$ from state $i$, the player immediately incurs a deterministic cost $c^a$, 
is held in the state $i$ for a nonnegative and deterministic (so possibly null) time $t^a$, and finally
goes to state $j\in S$ with probability $p^a_{j}$. Actions are pulled as soon as a state is reached, and we assume that
$\sum_{j\in S}p^a_{j}=1$ for all $i\in S$.
We depict in~\Cref{fig:EMS_A_SMDP} an SMDP with six states and eight actions. We will see that this is precisely the SMDP corresponding to our running example of a medical emergency call center, in the sense that the counter equations
of the Petri nets are equivalent ot the dynamic programming equations
of the SMDP~\cite{allamigeon2020piecewise}.

Recall that
a \emph{strategy}~\cite{puterman2014markov} is a mapping that associate choices of actions to play based on \emph{histories}, i.e. possible realizations of already visited states, pulled actions, incurred costs and sojourn times etc. We denote by $\Fbb$ the set of strategies.
A strategy $f$ in $\Fbb$ and an initial state $i\in S$ induce a probability measure $\mathbb{P}^f_i$ on the set of histories of the game. %
The strategy $f$ and the initial state $i$ also give rise to a random process $(\widehat{i}_k,\widehat{a}_k)_{k\in\N}$ (with $\widehat{i}_0=i$) of visited states and chosen actions. Denoting by $h=(i_k,a_k)_{k\in\N}$ a general trajectory realized by this process, we denote by $\widehat{c}_k$ (resp.\ $\widehat{t}_k$) the random variable such that $\widehat{c}_k(h)=c^{a_k}$ (resp.\ $\widehat{t}_k(h)=t^{a_k}$) for all $k$ in $\N$.

The {\em value function} $v^*:S\times \R \to \R$ of the game in finite horizon is then defined as follows, so that for $i$ in $S$ and $t$ in $\R$, $v^*(i,t)$ denotes the minimum (over all strategies) expected cost incurred by the player up to time $t$ by starting in state $i$ at the instant $0$:%
\begin{equation}\label{eq:def_value}
   v^*(i,t) \coloneqq \inf_{f\in\Fbb} \;\mathbb{E}^{f}_{i}\bigg(\sum_{k=0}^{\widehat{N_t}}\widehat{c}_k\bigg)
\end{equation}
where $\mathbb{E}^{f}_{i}$ denotes the expectation operator relatively to $\mathbb{P}^f_{i}$ and $\widehat{N_t}$ is the random variable with values in $\N$ such that 
\(\widehat{N_t}(h)=\sup\,\big\{n\in\N\;\big|\; \textstyle\sum_{k=0}^{n-1} \widehat{t}_k(h) \leq t\big\}\).
indeed observe that the cost $\widehat{c}_n$ is incurred at time $\sum_{k=0}^{n-1} \widehat{t}_k(h)$. We shall mainly focus on cases where $t\geq 0$, since the above definition implies $v^*(i,t) = 0$ for all $i$ in $S$ and $t < 0$.

Allowing zero sojourn times and yet immediate incurred costs raises the question of the well-posedness of the value function in~\eqref{eq:def_value}. To prevent \emph{Zeno behaviors}, i.e. accumulating infinitely many costs over a finite time period, a restriction is in order. We characterize \emph{non-Zeno} SMDPs using the standard notion of \emph{policies}: %
a policy $\sigma$ is a map from $S$ to $A$ such that $\sigma(i)$ lies in $A_i$ for every state $i$ of $S$ (some authors refer to this object as a {\em decision rule}). %
We denote by $\SSigma\coloneqq \prod_{i\in S}A_i$ the finite set of all the policies.
If $\sigma$ is a policy, $P^{\sigma}$ denotes the $|S|\times|S|$ matrix with entries $(p^{\sigma(i)}_{j})_{i,j\in S}$, while $c^{\sigma}$ (resp.\ $t^{\sigma}$) is the vector with entries $(c^{\sigma(i)})_{i\in S}$ (resp.\ $(t^{\sigma(i)})_{i\in S}$).

\begin{assumption}\label{ass:nonzeno}
For all policies $\sigma$ in $\SSigma$, for all finite classes $F\subset S$ of the Markov chain induced by $\sigma$, there is at least one state $i$ in $F$ with $t^{\sigma(i)}>0$, i.e. with positive sojourn time attached to the chosen action $\sigma(i)$ from $i$.
\end{assumption}

This assumption, milder than imposing an almost surely positive sojourn time for all actions of the game, is the same as the one that Schweitzer and Federgruen considered in their study~\cite{Schweitzer1978a} of an average-cost criterion for SMDPs. Under this assumption, one can check that the value function $v^*$ is well defined. In the rest of the paper, we shall denote by $\tmax$ the maximum positive sojourn time $t^a$ attached to an action $a$ in $A$.

In accordance with the dynamic programming principle, the value function
satsifies the following recursive optimality equation: %
\begin{equation}
  \tag{DP}
  \forall t\geq 0, \;\; \forall i\in S,\quad v(i,t) = \min_{a\in A_i} \bigg\{ c^{a} + \sum_{j\in S} p^{a}_{j}\,v(j, t-t^{a}) \bigg\} \, .
\label{eq:SMDP_finite_horizon}
\end{equation}
The {\em correspondence theorem}, established in~\cite[Th.~6.1]{allamigeon2020piecewise}, shows
that the counter functions of priority-free continuous timed Petri nets with preselection
routings are always of this form. When $t^a\equiv 1$, corresponding
to holding times equal to $1$ in all places, we recover
the standard dynamic programming equation of Markov decision
processes. However, allowing {\em real} soujourn times is natural
in the Petri-net applications, leading to a semi-Markov framework.
When $t^a\equiv 1$,
an initial condition $v(-1)$ fully determines the value function $v(t)$ for all $t \in \N$. In the case of SMDPs,
the knowledge of the value function over the whole interval $[t-\tmax, t)$ is needed to compute $v(t)$, where $\tmax$ is the maximum positive sojourn time $t^a$ attached to an action $a \in A$, and we get an infinite dimensional
  dynamics. This motivates the introduction of the space $\Vcal$ of bounded $|S|$-dimensional vector functions over $[-\tmax,0)$ in order to characterize the initial conditions of~\eqref{eq:SMDP_finite_horizon}. The following proposition shows that, under such an initial condition, there is a unique trajectory satisfying the dynamic programming equations~\eqref{eq:SMDP_finite_horizon}. The main difficulty is to handle null sojourn times, then terms of $v(\cdot, t)$ can appear in both sides of~\eqref{eq:SMDP_finite_horizon}.
We recall that a function of a real variable is {\em c\`adl\`ag} if it is right continuous and admits a limit at the left of every point.
\begin{proposition}\label{prop:uniquely_determined}
    Let $v^0 \in \Vcal$. Under Assumption~\ref{ass:nonzeno}, there is a unique function $v$ defined on $[-\tmax, \infty)$ satisfying  the dynamic programming equation~\eqref{eq:SMDP_finite_horizon} and for all $t\in[-\tmax,0)$, $v(t)=v^0(t)$. In addition, if $v^0$ is càdlàg\todo{SG: defined càdlàg above} (resp.\ piecewise-constant, piecewise-affine), then $v$ is càdlàg (resp.\ piecewise-constant, piecewise-affine).
\end{proposition}

Building in~\Cref{prop:uniquely_determined}, we establish that $v^*$ satisfies~\eqref{eq:SMDP_finite_horizon}. %

\begin{theorem}\label{thm:value_is_S0}
The value function $v^*$ defined by~\eqref{eq:def_value} is the unique function null on $[-\tmax,0)$ that verifies the dynamic programming equations~\eqref{eq:SMDP_finite_horizon}.
\end{theorem}

The last theorem and \Cref{prop:uniquely_determined} show in particular that $v^*$ is piecewise-constant. %

\paragraph{\textnormal{\textbf{Example}}} In the example of a medical emergency call center, with equations governing the number of tasks complete over the time given by~\eqref{SAMU1}. On the other hand, we write the dynamic programming equations~\eqref{eq:SMDP_finite_horizon} verified by the value function of the SMDP depicted in~\Cref{fig:EMS_A_SMDP}, for all $t\geq 0$:
\begin{subequations}
\begin{eqnarray}%
  v^*(0,t) & = & v^*(0,t-t_{0})  + \lambda t_{0} \\
  v^*(1,t) & = & v^*(0,t)  \wedge  \big(N_{A} + (1-\pi)v^*(2,t) + \pi v^*(4,t)\big)\\
  v^*(2,t) & = & v^*(1,t-t_1) \\
  v^*(3,t) & = & v^*(1,t-t_1) \wedge \big(N_{P}/\pi + v^*(5,t)\big)\\
  v^*(4,t) & = & v^*(3,t-t_2) \\
  v^*(5,t) & = & v^*(4,t-t_3)
\end{eqnarray} %
  \label{SAMU-SMDP}
\end{subequations}

We denote $(e_0 \coloneqq 1, e_1 \coloneqq 1, e_2 \coloneqq 1-\pi, e_3 \coloneqq \pi, e_4 \coloneqq \pi, e_5 \coloneqq \pi)$. As shown in~\cite{allamigeon2020piecewise}, under steady call arrivals with rate $\lambda$ so that $z_0(t)=\lambda t$, it can be seen that $z_i(t)/e_i = v^*(i,t)$, since initial conditions (nullity before $t=0$) and dynamics coincide. This is a consequence of the main result of~\cite{allamigeon2020piecewise} which states that priority free timed Petri nets with stoichiometric invariant are in correspondence with SMDPs and share common dynamics equations (actually the correspondence $v^*(0,t)=\lambda t$ holds for $t_{0}\to 0$ but further computations shall not require to take this limit, see~\Cref{ssec:SSP}).

\subsection{Properties of the dynamics governing the value function}

To portray the effect of applying equations~\eqref{eq:SMDP_finite_horizon} to a function of $\Vcal$, it shall come in handy to introduce the following evolution operator.

\begin{definition}\label{def:semigroup}
  For $t\geq 0$, we define the {\em evolution operator} $\Sg_t$ as the self-map
  of the set of functions $\Vcal$, propagating an initial condition $v^0$ by $t$ time units, i.e.,
\begin{equation}\label{eq:def_semigroup}
  \Sg_t:
  v^{0}  \longmapsto  \left\{\begin{array}{rcl} [-\tmax,0) & \to & \RS \\ s & \mapsto & v(t+s)\enspace, \end{array}\right.
\end{equation} 
where $v$ is the function uniquely determined by the initial condition $v^0$ and the equations~\eqref{eq:SMDP_finite_horizon}, in accordance with~\Cref{prop:uniquely_determined}.
\end{definition}

The operator $\Sg_t$ for SMDPs with $t\geq 0$ plays the same role as the $t$-fold iterate of the Bellman operator for MDPs (with in that case $t\in \N$). 
It follows from~\Cref{thm:value_is_S0} and~\eqref{eq:def_semigroup} that for all $t\geq 0$ and sufficiently small $\varepsilon$, we have $v^*(t)=\Sg_{t+\varepsilon}[\tilde{\zerobf}](-\varepsilon)$, where $\tilde{\zerobf}$ denotes the null function of $\Vcal$. 
It readily follows that the family $(\Sg_t)_{t\geq 0}$ constitutes a one-parameter semi-group, \ie, $\Sg_0$ is the identity map of $\Vcal$, and for all $t,t'\geq 0$, we have $\Sg_t \circ \Sg_{t'} = \Sg_{t+t'}$.

We list below some important properties of the operators $(\Sg_t)_{t\geq 0}$.
The space $\Vcal$ is equipped with the usual pointwise partial ordering $\leq$, and with the sup-norm $\big\Vert v^0\big\Vert_{\infty}\coloneqq \sup_{s\in[-\tmax,0)}\big\Vert v^0(s)\big\Vert_{\infty}$. We denote by $\tilde{\unbf}$ the constant function of $\Vcal$ with all coordinates equal to $1$.
\begin{proposition}\label{prop:properties_semigroup}
 Let $t\geq 0$. The operator $\Sg_t$ is
(i)  additively homogeneous: \(\forall \alpha\in\R,\quad \Sg_t[v^0+\alpha\tilde{\unbf}]=\Sg_t[v^0]+\alpha\tilde{\unbf}\);%
 (ii)%
 monotone (or order-preserving): \(v^0 \leq v'{}^0 \implies \Sg_t[v^0] \leq \Sg_t[v'{}^0]\);%
  (iii) nonexpansive: \( \big\Vert \Sg_t[v'{}^0] - \Sg_t[v^0]\big\Vert_{\infty} \leq \big\Vert v'{}^0 - v^0\big\Vert_{\infty} \);%
 (iv) continuous in both the uniform and pointwise convergence topologies.
\end{proposition}

The following proposition states the existence of an affine stationary regime, \ie, a particular initial condition function of $\Vcal$ for which evolution under~\eqref{eq:SMDP_finite_horizon} amounts to a translation in time.

\begin{proposition}[\cite{allamigeon2020piecewise}]\label{prop:affine1}
Suppose that Assumption~\ref{ass:nonzeno} holds. Then,
\begin{enumerate}[label=(\roman*)]
  \item\label{prop:affine1:kohlberg} there exists two vectors $\chi(\Sg)$ and $h$ in $\R^{|S|}$ such that the affine function $v^{\textrm{aff}}$ of $\Vcal$ defined by $v^{\textrm{aff}}:s\mapsto \chi(\Sg) s + h$ satisfies
    \( \Sg_t[ v^{\textrm{aff}}](s) = v^{\textrm{aff}}(t+s)\);
  \item\label{prop:affine1:universality} for all $v^0$ in $\Vcal$, we have for all $s$ in $[-\tmax, 0)$,
  \( \Sg_t[v^0](s) \underset{t\to\infty}{=} \chi(\Sg) t + O(1)\).
In particular, we have $v^*(t) \underset{t\to\infty}{=} \chi(\Sg) t + O(1)$.
\end{enumerate}
\end{proposition}

As it appears in item~\ref{prop:affine1:universality} above, the \emph{growth rate} $\lim_{t\to\infty}\Sg_t[v^0]/t$ is independent of the choice of $v^0$ in $\Vcal$ and relates only to the topology and the parameters of the SMDP (or its evolution semi-group $\Sg$), hence the notation $\chi(\Sg)$. This vector is none other than the solution of the minimal time-average cost problem associated with our SMDP, and the following proposition provides an explicit formula to compute it  in terms of the policies and the associated Markov chains, see also~\cite{SchweitzerFedergruenEquation}: 
\begin{proposition}\label{prop:affine2}Suppose that Assumption~\ref{ass:nonzeno} holds. Then, the vector $\chi(\Sg)$ featured in~\Cref{prop:affine1} is unique and we have for all $i$ in $S$
  \begin{equation}\label{eq:def_g}
    \chi(\Sg)_i = \min_{\sigma\in\SSigma}\sum_{F\in\mathcal{F}(\sigma)}\phi^F_i\,\frac{\langle \mu^{\sigma}_F, c^{\sigma}\rangle}{\langle \mu^{\sigma}_F, t^{\sigma}\rangle}\,,
  \end{equation}
  where for a policy $\sigma$ in $\SSigma$, $\mathcal{F}(\sigma)$ denotes the set of final classes of the Markov chain induced by $\sigma$, and for a final class $F$ in $\mathcal{F}(\sigma)$, $\mu^{\sigma}_F$ (resp.\ $\phi^F_i$) denotes the invariant nonnegative measure of class $F$ (resp.\ the probability to reach $F$ starting from $i$ under $\sigma$). 
\end{proposition}

Observe that the conditions of the non-Zeno Assumption~\ref{ass:nonzeno} are precisely those which make $\chi(\Sg)$ well-defined in~\Cref{prop:affine2}.

\begin{remark}\label{remark:reduced_semigroup}
  It shall turn convenient in~\Cref{sec:hierarchical} to consider an SMDP induced by a subset of actions. This may also require to restrict ourselves to an adequate subset of states. To that purpose,
  we say that a pair $(S',A'\coloneqq \biguplus_{i\in S'}A_i')$ with $S'\subset S$ and $\emptyset\neq A_i'\subset A_i$ for all $i$ in $S$ is a \emph{consistent subset of states and actions} if for all $a\in A'$, we have $\sum_{j\in S'}p^a_j = 1$.
  
  If $(S',A')$ is such a pair, the equations~\eqref{eq:SMDP_finite_horizon} restricted to actions of $A'$ and to states of $S'$ induce an evolution semi-group that we may denote by $\big(\Sg_t\big|_{S',\,A'}\big)_{t\geq 0}$ on the set of ``initial condition'' functions with suitable dimension, as set out in~\Cref{prop:uniquely_determined} and~\Cref{def:semigroup}. Any restricted semi-group of this form also enjoys the properties listed in~\Cref{prop:properties_semigroup}.  
\end{remark}

\section{Deviation of the value from the congestion-free regime}

\subsection{Reducing to a stochastic shortest path problem}

We deal with systems that have to meet a known inbound demand. In the framework of SMDP, such an input is conveniently modeled by distinguishing a sink state denoted by $0$, for which the value function $t\mapsto v^*(0,t)$ is prescribed for all $t\geq 0$.
We will assume that $v^*(0,t) = \lambda t$ for all $t\geq 0$, where $\lambda > 0$ is constant rate of tasks to perform per unit of time.
This input profile may be realized by our SMDP framework by first supposing that a single action $a_0$ is playable from state $0$ such that $p^{a_0}_0 = 1$, $t^{a_0} > 0$ and $c^{a_0}=\lambda t^{a_0}$, and also that the dynamics is initialized with a function of $\Vcal_{\lambda}\coloneqq \{v\in\Vcal \;|\; v(0,s) = \lambda s \;\;\text{for all}\;\;s\in[-\tmax,0)\}$. In what follows, we shall refer to this setting as \emph{$\lambda$-sink SMDPs}. 

By an abuse of language and notation, we shall still refer to the \emph{value function} and denote by $v^*$ the solution of the dynamics of a $\lambda$-sink SMDP uniquely determined by the initial condition $v^0$ in $\Vcal$ such that $v^0\big|_{\SminusZ} = \tilde{\zerobf}$. Indeed, this initial condition is not null on state $0$ on $[-\tmax,0)$, as required by~\Cref{thm:value_is_S0}, but the function it generates coincides on $[0,\infty)$ with the value function obtained by prescribing $v^*(0,t)=0$ for $t\in[-\tmax,0)$ and $v^*(0,t)=\lambda t$ for $t\geq 0$, provided that the actions giving access to state $0$ are equipped with null sojourn time, which involves no loss of generality up to considering intermediary states.

      \Cref{prop:affine1}
      and~\Cref{prop:affine2} show that the growth rate of the solutions of~\eqref{eq:SMDP_finite_horizon} is given by the one of an \emph{average-cost optimal} policy; the final classes of this policy indicate which parts of the system, by staying within them, enable the player to minimize his/her costs in the long run. In our emergency call center application, these final classes correspond to the slowest part of the treatment chain and therefore indicate which resources are bottleneck \textit{via} the computation of the growth rate. %
Observe that in a $\lambda$-sink SMDP, the singleton $\{0\}$ is a final class of all policies of $\SSigma$ and that $0$ must always have a growth-rate of $\lambda$.
We want to focus on ``congestion-free regimes'' of these SMDPs, \ie, for which the evolution under~\eqref{eq:SMDP_finite_horizon} is driven by the input, resulting in $\chi(\Sg)=\lambda \unbf$ (where $\unbf$ denotes the vector with all components set to $1$).
According to~\Cref{prop:affine2} and the previous remark, it is relevant to introduce the minimum possible growth-rate in the SMDP except $\lambda$, denoted by $\underline{\chi}$ and defined as
\begin{equation}\label{eq:chi_lower}\underline{\chi}\coloneqq \min_{\sigma\in\SSigma}\min_{\substack{F\in \mathcal{F}(\sigma) \\ F \neq \{0\}}}\left\{\frac{\langle \mu^{\sigma}_F, c^{\sigma}\rangle}{\langle \mu^{\sigma}_F, t^{\sigma}\rangle}\right\}\,.\end{equation}

It is therefore seen from~\Cref{prop:affine2} that $\chi(\Sg)=\lambda\unbf$ can occur only if $\underline{\chi} \geq \lambda$, \ie, the least growth-rate of the cost is achieved by accessing the sink state $0$; and imposing $\underline{\chi}>\lambda$ ensures that the less costly final class of all policies is always $\{0\}$. Conversely, we note that the value function of a state $i$ can be driven by the input only if this state has access to $0$, captured by the term $\phi^{\{0\}}_i$ in~\eqref{eq:def_g} (recall that for two states $i$ and $j$ of the SMDP, we say that $i$ \emph{has access to} $j$ iff there are states $i_0,i_1,\dots,i_k,i_{k+1}$ of $S$ with $i_0 = i$ and $i_{k+1}=j$, and actions $a_0,a_1,\dots,a_k$ of $A_{i_0}\times A_{i_1}\times\dots\times A_{i_k}$ such that $p^{a_{\ell}}_{i_{\ell+1}}>0$ for all $0\leq \ell \leq k$). This motivates the statement of the next assumption to characterize the congestion-free regimes of $\lambda$-sink SMDPs.

\begin{assumption}\label{ass:fluid_regime} Consider a $\lambda$-sink SMDP.
  \begin{enumerate}[label=(\arabic*)]
    \item\label{ass:fluid_regime:item1} All the states of $S$ have access to $0$,
    \item\label{ass:fluid_regime:item2} $\underline{\chi}>\lambda$.
  \end{enumerate}
  \end{assumption}

To investigate if a better result than the mere $v^*(t) \underset{t\to\infty}{=} \lambda t\unbf + O(1)$ can be obtained under Assumption~\ref{ass:fluid_regime}, we focus on the deviation $\Delta v^*(t)\coloneqq v^*(t)-\lambda t\unbf$. 
We remark that our setting bears much resemblance with the subclass of problems associated with MDPs known as \emph{Stochastic Shortest Path} (SSP) problems. An \emph{SSP configuration}
refers to a MDP in which there is a distinguished sink state denoted by $0$ such that any playable action from state $0$ has null cost and forces to stay in $0$ (for all $a\in A_0$, $c^a=0$ and $p^a_{0}=1$). It is therefore seen that as soon as one reaches state $0$, the accumulated cost no longer evolves and the game virtually stops; in this case it is licit to study the limit of the value function in~\eqref{eq:def_value} when $t$ tends to $\infty$.
We point out that the notion of SSP configuration carries over SMDPs. However, we are not aware of any study of the stochastic shortest path problem in the semi-Markov setting.

We now show that studying the deviation $\Delta v^*$ reduces to a SSP problem.
\todo{SG: terminology reduced cost introduced}
  \begin{theorem}\label{thm:normalization}
    Let $\Pcal$ be a $\lambda$-sink SMDP. The SMDP $\Pcal'$ obtained from $\Pcal$ by changing the costs to the {\em reduced costs} $c^a - \lambda t^a$ for all $a \in A$ is in SSP configuration. Moreover, $\Delta v^*$ satisfies the dynamic programming equations of $\Pcal'$, \ie, for all $t\geq 0$ and $i \in S$:
    \begin{equation}\label{eq:delta_SSP}
      \Delta v^*(i,t) =  \min_{a\in A_i} \bigg\{ \big(c^{a}-\lambda t^a\big) + \sum_{j\in S} p^{a}_{j}\, \Delta v^*(j, t-t^{a}) \bigg\} \enspace .
    \end{equation}
    \end{theorem}

    The key elements that make the SSP configuration arise under the cost reduction transformation featured in~\Cref{thm:normalization} are the fact that $0$ remains an absorbing sink state but in addition becomes cost-free.

\label{ssec:SSP}

The mere SSP configuration in terms of topology is not sufficient to obtain convergence of $\Delta v^*$. 
Instead, to ensure that the SSP problem is well-posed and that ultimate reachability of state $0$ is guaranteed, the notion of \emph{proper} policy is often introduced.

\begin{definition}
A policy $\sigma$ in $\SSigma$ is said \emph{proper} if for all $i$ in $S$, $\lim_{n\to\infty}\big[\big(P^{\sigma}\big)^n\big]_{i0} = 1$.
A non-proper policy is called \emph{improper}.\end{definition}

The next proposition links our setting of congestion-free regime for $\lambda$-sink SMDPs to the most standard assumptions made in the SSP literature, expressed in terms of proper and improper policies.
\begin{proposition}
\label{prop:ass_SSP_is_fluid_regime}
Let $\Pcal$ be a $\lambda$-sink SMDP. 
The two assumptions of~\ref{ass:fluid_regime} are equivalent to the following two conditions on the SMDP $\Pcal'$ in SSP configuration constructed in~\Cref{thm:normalization}:

\begin{enumerate}[label=(\arabic*')]
\item\label{prop:ass_SSP_is_fluid_regime:item1} There exists a proper policy.
\item\label{prop:ass_SSP_is_fluid_regime:item2} For every improper policy, the expectation of the accumulated reduced
  cost incurred up to time $t$ converges to $+\infty$ as $t\to\infty$,
  for at least one initial state.\todo{SG: I explained what it means, since the notation of infinite total reduced cost was undefined}
\end{enumerate}
\end{proposition}
\todo{SG: I explain better the link with BT}

This entails that the SMDP with reduced costs satisfies precisely
the condition \cite[Assumption 1]{bertsekas1991analysis}.

\begin{theorem}[Corollary of~\cite{bertsekas1991analysis}]\label{thm:bertsekas}
  Suppose that conditions~\ref{prop:ass_SSP_is_fluid_regime:item1} and~\ref{prop:ass_SSP_is_fluid_regime:item2} of~\Cref{prop:ass_SSP_is_fluid_regime} hold. Then, the equations 
  \begin{eqnarray}\label{eq:fixed_point_of_T}
      \forall i\in S\setminus\{0\},\quad u(i) &=& \displaystyle\min_{a\in A_i}\bigg\{ \big(c^a-\lambda t^a\big) + \sum_{j\in S} p^{a}_{j}\,u(j) \bigg\},\quad 
     \quad\;\; u(0)  =  0
\end{eqnarray}
      admit a unique solution in $\R^{|S|}$.%
\end{theorem}
In what follows, if Assumption~\ref{ass:fluid_regime} is verified, we denote by $u^*$ the unique solution of equations~\eqref{eq:fixed_point_of_T}. Building on~\Cref{thm:bertsekas}, Bertsekas and Tsitsiklis show that the total cost of the SSP problem specified to the case of MDPs (when all the delays $(t^a)_{a\in A}$ are equal to $1$) coincides with the solution of~\eqref{eq:fixed_point_of_T}, and that it arises as the limit of iterates of the associated Bellman operator applied to any starting vector. %
We establish in the next theorem that these results carry over the semi-Markov framework, by considering the evolution semigroup $(\Sg_t^{\Delta})_{t\geq 0}$ associated with the reduced dynamic programming equations~\eqref{eq:delta_SSP}. The latter naturally acts on the set of initial conditions $\Vcal_0 \coloneqq \{v\in\Vcal \;|\; v_0(s) = 0 \;\;\text{for all}\;\;s\in[-\tmax,0)\}$. 

\begin{theorem}\label{thm:SSSP}
Suppose that Assumptions~\ref{ass:nonzeno} and~\ref{ass:fluid_regime} hold. Then,  %
for all $v^0$ in $\Vcal_0$, for all $s$ in $[-\tmax,0)$, we have 
\(\lim_{t\to\infty}\Sg_t^{\Delta}[v^0](s) = u^*\).%
\end{theorem}
I
n other words, \Cref{thm:SSSP} states that the minimum ultimate total expected cost of an SMDP in SSP configuration (such as the decision process $\Pcal'$ featured in~\Cref{thm:normalization}) is the same than in a corresponding MDP with unit transition times.%
It is a consequence of the fact we study the infinite-horizon limit of a total reduced cost, which is a
time indifferent quantity. In particular, the optimality equality characterizing
the limit cost $u^*$ has precisely
the same form in the MDP and in the SMDP case,
the only change being that the delays $(t^a)_{a\in A}$ can take non unit values in the SMDP case.

We obtain a direct corollary of~\Cref{thm:SSSP} for the study of the value function, which indeed improves the result of~\Cref{prop:affine1} in the congestion-free regime. 
We insist on the importance of the first result for applications: whatever the initial condition, the trajectory of a non-Zeno $\lambda$-sink SMDP in congestion-free regime always ultimately catches up the input $t\mapsto \lambda t$, up to a constant delay $u^*$.

\begin{corollary}\label{coro:value_cv}
  For a $\lambda$-sink SMDP, if Assumptions~\ref{ass:nonzeno} and~\ref{ass:fluid_regime} hold, we have for all $v^0$ in $\Vcal_{\lambda}$ and $s$ in $[-\tmax,0)$:
  \( \Sg_t[v^0](s)\underset{t\to\infty}{=} \lambda (t+s)\unbf + u^* + o(1)\). %
   In particular,
   \( v^*(t) \underset{t\to\infty}{=} \lambda t\unbf + u^* + o(1)\).%
\end{corollary}

\paragraph{\textnormal{\textbf{Example}}} We come back to our running example of a medical emergency call center, and we first acknowledge that the SMDP depicted in~\Cref{fig:EMS_A_SMDP} is indeed a $\lambda$-sink SMDP. It is easy to verify in this figure that all the states have access to $0$ and thus Assumption~\ref{ass:fluid_regime}-\ref{ass:fluid_regime:item1} is satisfied. The computation of the scalar $\underline{\chi}$, using closed forms of the probabilistic invariants associated with the final classes of the four policies of the SMDP, provides $\underline{\chi} = \lambda_A \wedge \lambda_P$, where 
\[ \lambda_A\coloneqq \frac{N_A}{t_1+\pi t_2} \qquad \text{and}\qquad \lambda_P \coloneqq \frac{N_P}{\pi(t_2+t_3)}\,.\]
These two quantities are fundamental in the analysis of the SMDP behavior and closely tied to real organization. The first one, $\lambda_A$, can be interpreted by the calls handling speed of MRAs, indeed there are $N_A$ of them and a fraction $1-\pi$ (resp. $\pi$) of their work is to perform tasks 1 and 2 (resp. 1, 2, 3 and 4) which consumes a time $t_1$ (resp. $t_1+t_2$), hence accounting for an average cycle time of $t_1+\pi t_2$. The quantity $\lambda_P$ can similarly be interpreted as the handling speed of the $N_P$ emergency physicians, having a cycle time of $t_2+t_3$ (to achieve tasks 3, 4 and 5) for a fraction $\pi$ of all the calls. Depending on the choice of $N_A$ and $N_P$ (but also of the other parameters), there may or may not be enough agents (whether MRAs or doctors) to perform without delay all the tasks arriving with flow $\lambda$. 

The medical emergency call center that we model by means of this SMDP thus behaves in its congestion-free regime if and only if Assumption~\ref{ass:fluid_regime}-\ref{ass:fluid_regime:item2} is met, i.e., if $\lambda_A>\lambda$ and $\lambda_P>\lambda$. In this case, the call center is correctly staffed and there are enough agents to ultimately pick up all the calls with no delay, \ie, all the functions $t\mapsto z_i(t)$ in~\eqref{SAMU1} admit their maximum throughput $\lambda e_i$. The policy that selects $(\sigma(1),\sigma(3))=(a_1^-,a_3^-)$ is the one that achieves minimal throughput (only the calls arrival itself is bottleneck). It is associated with the unique final class $\{0\}$ and an affine stationary regime $t\mapsto \lambda t \unbf + u^*$, where $u^* = -\lambda(0, 0, t_1, t_1, t_1+t_2, t_1+t_2+t_3)$. As shown in~\Cref{coro:value_cv}, thanks to the sufficient staffing, this congestion-free regime where all calls are handled with no delays is always ultimately reached.

If $\lambda_A < \lambda$ or $\lambda_P<\lambda$, this means that there are either not enough MRAs or not enough emergency physicians to perform all the required tasks, and due to the synchronization step (\textsc{Task} 3), the whole chain of treatment is slowed down. Limit cases ($\lambda_A=\lambda$ or $\lambda_P=\lambda$) would require a more detailed analysis.

\subsection{Transience time needed to catch-up the input}

In the congestion-free regime, the function $w:t\mapsto \lambda t \unbf + u^*$ of $\Vcal_{\lambda}$ featured in~\Cref{coro:value_cv} gives rise to an affine stationary regime in the sense that for all $t\geq 0$ and $s$ in $[-\tmax,0)$, we have $\Sg_t[w](s) = w(t+s)$ with no error term, like introduced in~\Cref{prop:affine1}-\ref{prop:affine1:kohlberg}; any such regime actually differs from $w$ by a multiple of $\tilde{\unbf}$. 

We leverage on the previous framework to study the effect of a perturbation on a steady-flow input in a $\lambda$-sink SMDP.
For sake of simplicity, suppose that the system is initialized in its stationary affine congestion-free regime $w$ (this entails little loss of generality since this behavior is always ultimately reached according to~\Cref{coro:value_cv}). We suppose that at $t=\underline{t}$, the input incurs a step of cost, so that $v(0,t) = \lambda t  + M \tilde{H}(t-\underline{t})$ for all $t\geq 0$, where $M \geq 0$ and $\tilde{H}$ denotes the Heaviside function. In our emergency call center application, this can be used to simulate the sudden arrival at time $\underline{t}$ of $M$ new calls to handle on top of the usual demand with rate $\lambda$. The next result formalizes the fact that these requirements on $v(0,\cdot)$ actually define a unique trajectory of the $\lambda$-sink SMDP.%
\begin{proposition}
  \label{prop:uniquely_determined_heaviside}
  Under Assumption~\ref{ass:nonzeno}, there is a unique function $v \colon [-\tmax,\infty)\to\R^{|S|}$ such that \begin{enumerate}[label=--]
    \item for all $t$ in $[-\tmax,\infty)$, $v(0,t) = \lambda t + M\tilde{H}(t-\underline{t})$, 
    \item for all $t$ in $[-\tmax,0)$, $v\big|_{\SminusZ}=\tilde{\zerobf}$, 
    \item $v$ satisfies equations~\eqref{eq:SMDP_finite_horizon} for states in $\SminusZ$. 
  \end{enumerate}

  Moreover, if Assumption~\ref{ass:fluid_regime} holds, we have \(v(t) \underset{t\to\infty}{=} \lambda t \unbf + u^* + M\unbf + o(1)\).%
\end{proposition}

Building on the last statement of~\Cref{prop:uniquely_determined_heaviside}, which guarantees that in spite of the perturbation, the trajectory of the system still catches up the input (including the step of magnitude $M$), we want to study the transience time before this final regime is reached. To this purpose, we define for the state-by-state catch-up times (or transience times) $(\theta_i)_{i\in S}$ by:
\[ \forall i \in S,\quad \theta_i\coloneqq \inf\big\{ t\geq \underline{t} \;\big|\; v(i,t) = \lambda t + u^*(i) + M\big\}\,.\]

In~\Cref{sec:convergence}, we shall study cases where the catching-up is exact and thus occurs in finite time; but we may as well have defined the $(\theta_i)_{i\in S}$ by the time needed for $v(i,t)$ to approach $\lambda t + u^*(i) + M$ up to a chosen precision $\varepsilon$.

The next theorem shows that the problem of characterizing the $(\theta_i)_{i\in S}$ reduces to the study of catch-up times in an SMDP in SSP configuration starting from a particular initial condition.
\begin{theorem}\label{thm:reduction_SSP_2}
  Suppose that Assumptions~\ref{ass:nonzeno} and~\ref{ass:fluid_regime} hold. Let $v'$ be the function uniquely determined by the initial condition $s\mapsto u^* - M\unbm_{\SminusZ}$ in $\Vcal_0$ and the dynamics of the reduced-costs SMDP in SSP configuration featured in~\Cref{thm:normalization}.
    Then, for all $i$ in $S$, we have 
    $ \theta_i = \underline{t} + \theta_i',\quad \textrm{where} \quad \theta_i'\coloneqq \inf\big\{t\geq 0 \;\big|\; v'(i,t) = u^*(i)\big\}$.
\end{theorem}

 \section{Finite time convergence of the Semi-Markov SSP-value}
 \label{sec:convergence}

 In this section, we tackle the purified version of our problem brought to light in~\Cref{thm:reduction_SSP_2}, letting aside $\lambda$-sink SMDP to solely focus on SMDPs in SSP configurations.
 In particular, the evolution semigroups considered hereafter such as $(\Sg_t)_{t\geq 0}$ are associated with dynamics of type~\eqref{eq:delta_SSP} with null costs on state $0$. Furthermore, observing that the definition of $\underline{\chi}$ in~\eqref{eq:chi_lower} for $\lambda$-sink SMDPs carries over to SSP configurations (that are essentially $0$-sink SMDPs). In consequence, we rephrase the two conditions of Assumption~\ref{ass:fluid_regime} in the special case of SMDPs in SSP configuration.

 \begin{assumption}\label{ass:fluid_SSP_configuration} Consider a SMDP in SSP configuration.
  \begin{enumerate}[label=(\arabic*)]
    \item\label{ass:fluid_SSP_configuration:item1} All the states of $S$ have access to $0$,
    \item\label{ass:fluid_SSP_configuration:item2} $\underline{\chi}>0$.
  \end{enumerate}
\end{assumption}

The following lemma ensures that Assumption~\ref{ass:fluid_SSP_configuration} corresponds to Assumption~\ref{ass:fluid_regime} in the sense of the reduction stated in \Cref{thm:normalization}.  

\begin{lemma}
Let $\Pcal$ be a $\lambda$-sink SMDP, and $\Pcal'$ be the SMDP in SSP configuration built in \Cref{thm:normalization}. Then $\Pcal$ satisfies Assumption~\ref{ass:fluid_regime} if and only if $\Pcal'$ satisfies Assumption~\ref{ass:fluid_SSP_configuration}.
\end{lemma}
We point out that this result follows from the application of \Cref{prop:ass_SSP_is_fluid_regime} to $\Pcal$ and $\Pcal'$, and the fact that $\Pcal'$ is precisely the reduced cost SMDP associated to itself.

The purpose of this section is to provide quantitative and constructive upper bounds on the $(\theta_i')_{i\in S}$ introduced in~\Cref{thm:reduction_SSP_2}. This, in turn, yields upper bounds on the catch-up times of the original $\lambda$-sink SMDP, thanks to~\Cref{thm:reduction_SSP_2}.

\subsection{The different convergence phases}
  \label{sec:regimes}

  Recall as stated in~\Cref{thm:SSSP} that a function determined by an initial condition $v^0$ in $\Vcal_0$ and satisfying the dynamic programming equations of a SMDP in SSP configuration converges towards $u^*$ as considered horizons approach infinity. We first provide a qualitative insight on the speed of convergence associated with this result.

  For all $i$ in $\SminusZ$, we denote by $A_i^*$ the nonempty subset of $A_i$ composed of \emph{optimal actions}, i.e. those achieving minimality in~\eqref{eq:fixed_point_of_T}. Similarly, we introduce $\SSigma^*\coloneqq A_0\times\prod_{i\in \SminusZ}{A_i^*}$ the set of \emph{optimal policies}, included in $\SSigma$. In accordance with~\Cref{remark:reduced_semigroup}, optimal actions induce an evolution semi-group $\Sg^*$, which corresponds to applying only optimal policies. The following proposition shows that there is always an instant $t^*$ such that: 
\begin{enumerate*}
\item before~$t^*$, either optimal or non-optimal actions can be picked, and thus the evolution of the dynamics amounts to applying the semigroup $\Sg$;
\item after~$t^*$, only optimal actions are chosen, so that the finer dynamics associated with the semigroup $\Sg^*$ are actually used.
\end{enumerate*}
Moreover, the second phase induces a geometric speed of convergence controlled by the spectral radii of the probability matrices of optimal policies.
  \todo{SG: warning, previous statement was false, contraction is only in a WEIGHTED sup-norm and an epsilon was missing}
The geometric convergence result is deduced from~\cite[Th.~1 and Th~.2]{akian2019solving},
using tools from non-linear Perron--Frobenius theory~\cite{akian2011collatz}.
In particular, we refer to~\cite{akian2019solving} for background on weighted sup-norms.
\todo{SG: The notion of weighted sup norm is not so trivial in the semigroup case, we will have to discuss this after submission}

\begin{proposition}\label{prop:phases}
Suppose Assumptions~\ref{ass:nonzeno} and \ref{ass:fluid_SSP_configuration} hold.
Let $v^0 \in \Vcal_0$ and denote by $v$ the solution of~\eqref{eq:SMDP_finite_horizon} it determines. Then, there exists $t^*\geq 0$ such that for all $t\geq t^*$ and all $s$ in $[-\tmax, 0)$, we have 
\( v(t+s) = \Sg^*_{t-t^*}\big[\Sg_{t^*}[v^0]\big](s)\).
Moreover, if we define 
\(
\nu:=\max_{\sigma\in \SSigma^*}\rho\big(P^{\sigma}\big|_{(\SminusZ)\times(\SminusZ)}\big)<1\), where $\rho(\cdot)$ denotes the spectral radius of a matrix, then 
for all $\varepsilon>0$ small enough, there is a
weighted sup norm in which $\Sg^*_{\tmax}$ is a contraction of rate $\nu+\varepsilon$.
\end{proposition}

We point out that this proposition is reminiscent of the different phases of convergence established by Schweitzer and Federgruen in their work~\cite{schweitzer79} of the deviation $v(t)-\chi(\Sg)t$ in the case of MDPs when this quantity admits a limit.

We are interested in cases where convergence occurs in finite time,
i.e., {\em without geometric residual},
which is desirable in our call center application.
The next theorem shows that such a finite time convergence cannot be expected unless all probability matrices associated with optimal proper policies and restricted to states of $\SminusZ$ are nilpotent, hence making the rate of geometric convergence $\nu$ featured in~\Cref{prop:phases} null.
The latter is also equivalent to some restrictions on the SMDP topology, that we may interpret as requiring a form of \emph{hierarchy} within the set of states that is compatible with the moves made by the optimal policies; these should always ``descend'' in the hierarchy until finally reaching $0$, the minimal element.

\begin{theorem}
  \label{thm:finite_time_theoretical}
  Suppose Assumptions~\ref{ass:nonzeno} and \ref{ass:fluid_SSP_configuration} hold. %
  Then, the following are equivalent:
  \begin{enumerate}[label=(\roman*)]
    \item\label{thm:finite_time_theoretical:item1} for all $v^0$ in $\Vcal_0$ and associated solution $v$ of~\eqref{eq:SMDP_finite_horizon}, there exists $t^*$ in $\R$ such that for all $t\geq t^*$, $v(t)=u^*$%
    \item\label{thm:finite_time_theoretical:item2} for all proper optimal policies $\sigma$ in $\SSigma^*$, we have %
      \(\rho\big(P^{\sigma}\big|_{(\SminusZ)\times(\SminusZ)}\big) = 0\),%
    \item\label{thm:finite_time_theoretical:item3} there exists a partial ordering $(\leq)$ on $S$ such that for all $\sigma$ in $\SSigma^*$ and for all $i$ in $\SminusZ$, $\supp(\sigma(i))\subset\{j\in S\,,\;j < i\}$.
  \end{enumerate}
\end{theorem}

\subsection{Hierarchical SSP configurations}
\label{sec:hierarchical}

The set $\SSigma^*$ of optimal policies that control the ultimate rate of convergence according to the previous theorem is in general not known, and depends on the costs and sojourn times attached to the actions of the SMDP. It is desirable to identify conditions of a more topological nature under which finite time convergence occurs regardless of optimal character of some policies.
Building on the statement of~\Cref{thm:finite_time_theoretical}-\ref{thm:finite_time_theoretical:item3}, we choose to enforce the existence of a partial ordering $(\leq)$ on $S$ such that applying some actions necessarily make the order strictly decrease. This is formalized in the next assumption.

\begin{assumption}\label{ass:hierarchy1}
  There is a partial ordering $(\leq)$ on $S$, such that for all state $i \in S$, there is a partition of the set of playable actions from $i$ in the form $A_i = A_i^- \uplus A_i^+$, with the condition that
  if $a\in A_i^-$, then ($ p^a_j > 0 \implies j < i$) and
  if $a\in A_i^+$, then ($p^a_j > 0 \implies i \leq j)$,
in which $j < i$ means that $j \leq i$ and $j \neq i$. 
In addition, $A_i^-\neq\emptyset$ for all $i\in \SminusZ$.
\end{assumption}

In other words, Assumption~\ref{ass:hierarchy1} requires that actions either strictly ``descend'' or weakly ``ascend'' relatively to the states hierarchy, the first case being always possible. In what follows, we shall refer to $a\in A_i^-$ (resp.\ $a\in A_i^+$) as a ``descending action'' (resp.\ an ``ascending action'').
Remark that imposing in Assumption~\ref{ass:hierarchy1} that $A_i^-\neq\emptyset$ for all $i\in\SminusZ$ implies the condition \ref{prop:ass_SSP_is_fluid_regime:item1} of \Cref{prop:ass_SSP_is_fluid_regime}, since any policy which makes use of only descending actions is proper. This implication turns into an equivalence (\ie, any proper policy must make use of only descending actions) if the next condition is also met.

\begin{assumption}\label{ass:hierarchy2}
  For all $\sigma\in\SSigma$, if there exists $i\in S$ such that $\sigma(i)\in A_i^+$, then $\sigma$ is improper.
\end{assumption}

Under Assumptions~\ref{ass:hierarchy1} and~\ref{ass:hierarchy2} (stronger than the condition~\ref{thm:finite_time_theoretical:item3} of~\Cref{thm:finite_time_theoretical}), we obtain that $A_i^* \subset A_i^-$ for all $i \in S$ and according to~\Cref{thm:finite_time_theoretical}, convergence towards $u^*$ arises in finite time.

\paragraph{\textnormal{\textbf{Example}}} For the medical emergency call center depicted on~\Cref{fig:EMS_A_SMDP}, these two hierarchical assumptions are satisfied, as highlighted on the first picture of~\Cref{fig:three_hierarchies}. The states hierarchy is the one naturally given by the (partial) order in which the different tasks are performed in the call center.

\medskip 

Our goal is to leverage on the hierarchical structure brought by Assumptions~\ref{ass:hierarchy1} and~\ref{ass:hierarchy2} to bound the catch-up times $(\theta_i')_{i\in S}$ defined in~\Cref{thm:reduction_SSP_2}. To this purpose, we remark that the state $0$ is the bottom element of the order $(\leq)$, and that the states that lie in low layers of the hierarchy shall have ``quicker'' access to state $0$ than states located higher in the hierarchy (since the latter may need to pass through the former). From this perspective, if $i$ is a state of $S$, it is natural to study the catch-up time $\theta_i'$ after the states located lower than $i$ in the hierarchy already caught up the input. It amounts to determining the $(\theta_i')_{i\in S}$ by following an inductive scheme given by the partial ordering $(\leq)$.

This reasoning would result in straightforward bounds in the absence of ascending actions, \ie, which make the player move in states located higher in the hierarchy and thus slow down the catching-up of the input at the lowermost state $0$. We know under Assumption~\ref{ass:fluid_SSP_configuration} and~\ref{ass:hierarchy2} that these actions are ultimately non optimal because they are necessarily associated with improper policies and it is less costly to end up playing a proper policy. However for short decision horizons, some of these ascending actions may actually be optimal.

We tackle this difficulty by considering the collection of state-accessibility graphs of our SMDP in which particular states cannot use descending actions. More precisely, given $i\in \SminusZ$ such that $A_i^+\neq\emptyset$, we introduce the graph $\mathcal{G}^{(i)}_{S}$ with nodes set $S$ and with edge set
  \( \big\{(k,\ell) \in S^2\;\;\big|\;\; \exists a \in A_k \setminus A_k^-\quad\textrm{s.t.}\quad p^a_{\ell} > 0\big\}\).%
We denote by $\CFC{i}$ the strongly connected component of state $i$ in the graph $\mathcal{G}^{(i)}_S$. The~\Cref{fig:three_hierarchies} represents the two such subgraphs of our running example SMDP introduced in~\Cref{fig:EMS_A_SMDP}. It can be seen on this example that for $i$ in $S$ such that $A_i^+\neq \emptyset$, not selecting actions in $A_i^-$ amounts to staying within states of $\CFC{i}$ with no way of accessing $0$. The next lemma formalizes this fact.

\begin{figure}[htb]
\begin{center}
\resizebox{.48\textwidth}{!}{
  \def\tkzscl{.5}
\begin{tabular}{ccc}
  $\mathcal{G}_S$ & 
  $\mathcal{G}_S^{(1)}$ & 
  $\mathcal{G}_S^{(3)}$ \\
  \def\cutAt{0}\definecolor{orangeDIY}{rgb}{.92,.5,.11}
\definecolor{bluePY}{HTML}{1F77B4} 

\begin{tikzpicture}[auto,node distance=8mm,>=latex,scale=\tkzscl]
  \tikzset{arrowPetri/.style={>=latex,rounded corners=5pt}}

    \tikzstyle{round}=[thick,draw=black,circle,inner sep=2pt]

    \def\p{2.2}
    \def\colorBlockA{black}
    \def\colorBlockB{black}
    \def\colorBlockC{black}
    \def\opacityLinkB{1}
    \def\opacityLinkC{1}
    \ifnum\cutAt>0
      \def\colorBlockB{bluePY}
      \def\colorBlockC{bluePY}
        \def\opacityLinkB{.22}
    \fi
    \ifnum\cutAt>1
    \def\colorBlockB{black}
    \def\colorBlockC{bluePY}
      \def\opacityLinkB{1}
        \def\opacityLinkC{.22}
    \fi

    \node[round, color = \colorBlockA] (s0) at (0,0)              {\scriptsize$0$};
    \node[round, color = \colorBlockB] (s1) at ($(s0)+(0,-\p)$) {\scriptsize$1$};
    \node[round, color = \colorBlockB] (s2) at ($(s1)+(0,-\p)$)   {\scriptsize$2$};
    \node[round, color = \colorBlockC] (s3) at ($(s1)+(\p,-\p)$)  {\scriptsize$3$};
    \node[round, color = \colorBlockC] (s4) at ($(s3)+(0,-\p)$)   {\scriptsize$4$};
    \node[round, color = \colorBlockC] (s5) at ($(s4)+(0,-\p)$)   {\scriptsize$5$};

   \coordinate (a0)  at ($(s0)+(0,.5*\p)$) {};  
   \coordinate (a11) at ($(s1)+(0,.5*\p)$) {};  
   \coordinate (a12) at ($(s1)+(-.5*\p,0)$) {};  
   \coordinate (a2)  at ($(s2)+(0,.5*\p)$) {};  
   \coordinate (a31) at ($(s3)+(0,.5*\p)$) {};  
   \coordinate (a32) at ($(s3)+(.5*\p,0)$) {};  
   \coordinate (a4)  at ($(s4)+(0,.5*\p)$) {};  
   \coordinate (a5)  at ($(s5)+(0,.5*\p)$) {};  


     \draw[-{Latex[scale=1]}, rounded corners=5pt,  color = \colorBlockA] (a0) |- ($(a0)+(-.5*\p-.5,.2)$) |- (s0);
     \draw[-{Latex[scale=1]}, rounded corners=5pt,  color = \colorBlockA, opacity = \opacityLinkB] (a11) -- (s0);
     \draw[-{Latex[scale=1]}, rounded corners=5pt,  color = \colorBlockB] (a12) -| ($(a12)+(-.5,-1)$) |- (s2);
     \draw[-{Latex[scale=1]}, rounded corners=5pt,  color = \colorBlockB] (a12) -| ($(a12)+(-.5,-1)$) |- (s4);
     \draw[-{Latex[scale=1]}, rounded corners=5pt,  color = \colorBlockB, opacity = \opacityLinkC] (a31) |- (s1);
     \draw[-{Latex[scale=1]}, rounded corners=5pt,  color = \colorBlockB] (a2) -- (s1);
     \draw[-{Latex[scale=1]}, rounded corners=5pt,  color = \colorBlockC] (a4) -- (s3);
     \draw[-{Latex[scale=1]}, rounded corners=5pt,  color = \colorBlockC] (a5) -- (s4);
     \draw[-{Latex[scale=1]}, rounded corners=5pt,  color = \colorBlockC] (a32) -| ($(s5)+(.66*\p+.4,0)$) -- (s5);

   \draw[-{Square[scale=1.5, open, fill=white]}, semithick, color = \colorBlockA] (s0) -- (a0);
   \draw[-{Square[scale=1.5, open, fill=white]}, semithick, color = \colorBlockA, opacity = \opacityLinkB] (s1) -- (a11);
   \draw[-{Square[scale=1.5, open, fill=white]}, semithick, color = \colorBlockB] (s1) -- (a12);
   \draw[-{Square[scale=1.5, open, fill=white]}, semithick, color = \colorBlockB] (s2) -- (a2);
   \draw[-{Square[scale=1.5, open, fill=white]}, semithick, color = \colorBlockB, opacity = \opacityLinkC] (s3) -- (a31);
   \draw[-{Square[scale=1.5, open, fill=white]}, semithick, color = \colorBlockC] (s3) -- (a32);
   \draw[-{Square[scale=1.5, open, fill=white]}, semithick, color = \colorBlockC] (s4) -- (a4);
   \draw[-{Square[scale=1.5, open, fill=white]}, semithick, color = \colorBlockC] (s5) -- (a5);

   \node[ color = \colorBlockA] (lbl_a0)  at ($(a0) +(0,-.2)$) {\tiny$+$};
   \node[ color = \colorBlockB] (lbl_a0)  at ($(a12) +(.2,0)$) {\tiny$+$};
   \node[ color = \colorBlockC] (lbl_a0)  at ($(a32) +(-.2,0)$) {\tiny$+$};
   \node[ color = \colorBlockA, opacity = \opacityLinkB] (lbl_a0)  at ($(a11) +(0,-.245)$) {\tiny$-$};
   \node[ color = \colorBlockB, opacity = \opacityLinkC] (lbl_a0)  at ($(a31) +(0,-.245)$) {\tiny$-$};
   \node[ color = \colorBlockB] (lbl_a0)  at ($(a2) +(0,-.245)$) {\tiny$-$};
   \node[ color = \colorBlockC] (lbl_a0)  at ($(a4) +(0,-.245)$) {\tiny$-$};
   \node[ color = \colorBlockC] (lbl_a0)  at ($(a5) +(0,-.245)$) {\tiny$-$};

\end{tikzpicture} & 
  \def\cutAt{1}\definecolor{orangeDIY}{rgb}{.92,.5,.11}
\definecolor{bluePY}{HTML}{1F77B4} 

\begin{tikzpicture}[auto,node distance=8mm,>=latex,scale=\tkzscl]
  \tikzset{arrowPetri/.style={>=latex,rounded corners=5pt}}

    \tikzstyle{round}=[thick,draw=black,circle,inner sep=2pt]

    \def\p{2.2}
    \def\colorBlockA{black}
    \def\colorBlockB{black}
    \def\colorBlockC{black}
    \def\opacityLinkB{1}
    \def\opacityLinkC{1}
    \ifnum\cutAt>0
      \def\colorBlockB{bluePY}
      \def\colorBlockC{bluePY}
        \def\opacityLinkB{.22}
    \fi
    \ifnum\cutAt>1
    \def\colorBlockB{black}
    \def\colorBlockC{bluePY}
      \def\opacityLinkB{1}
        \def\opacityLinkC{.22}
    \fi

    \node[round, color = \colorBlockA] (s0) at (0,0)              {\scriptsize$0$};
    \node[round, color = \colorBlockB] (s1) at ($(s0)+(0,-\p)$) {\scriptsize$1$};
    \node[round, color = \colorBlockB] (s2) at ($(s1)+(0,-\p)$)   {\scriptsize$2$};
    \node[round, color = \colorBlockC] (s3) at ($(s1)+(\p,-\p)$)  {\scriptsize$3$};
    \node[round, color = \colorBlockC] (s4) at ($(s3)+(0,-\p)$)   {\scriptsize$4$};
    \node[round, color = \colorBlockC] (s5) at ($(s4)+(0,-\p)$)   {\scriptsize$5$};

   \coordinate (a0)  at ($(s0)+(0,.5*\p)$) {};  
   \coordinate (a11) at ($(s1)+(0,.5*\p)$) {};  
   \coordinate (a12) at ($(s1)+(-.5*\p,0)$) {};  
   \coordinate (a2)  at ($(s2)+(0,.5*\p)$) {};  
   \coordinate (a31) at ($(s3)+(0,.5*\p)$) {};  
   \coordinate (a32) at ($(s3)+(.5*\p,0)$) {};  
   \coordinate (a4)  at ($(s4)+(0,.5*\p)$) {};  
   \coordinate (a5)  at ($(s5)+(0,.5*\p)$) {};  


     \draw[-{Latex[scale=1]}, rounded corners=5pt,  color = \colorBlockA] (a0) |- ($(a0)+(-.5*\p-.5,.2)$) |- (s0);
     \draw[-{Latex[scale=1]}, rounded corners=5pt,  color = \colorBlockA, opacity = \opacityLinkB] (a11) -- (s0);
     \draw[-{Latex[scale=1]}, rounded corners=5pt,  color = \colorBlockB] (a12) -| ($(a12)+(-.5,-1)$) |- (s2);
     \draw[-{Latex[scale=1]}, rounded corners=5pt,  color = \colorBlockB] (a12) -| ($(a12)+(-.5,-1)$) |- (s4);
     \draw[-{Latex[scale=1]}, rounded corners=5pt,  color = \colorBlockB, opacity = \opacityLinkC] (a31) |- (s1);
     \draw[-{Latex[scale=1]}, rounded corners=5pt,  color = \colorBlockB] (a2) -- (s1);
     \draw[-{Latex[scale=1]}, rounded corners=5pt,  color = \colorBlockC] (a4) -- (s3);
     \draw[-{Latex[scale=1]}, rounded corners=5pt,  color = \colorBlockC] (a5) -- (s4);
     \draw[-{Latex[scale=1]}, rounded corners=5pt,  color = \colorBlockC] (a32) -| ($(s5)+(.66*\p+.4,0)$) -- (s5);

   \draw[-{Square[scale=1.5, open, fill=white]}, semithick, color = \colorBlockA] (s0) -- (a0);
   \draw[-{Square[scale=1.5, open, fill=white]}, semithick, color = \colorBlockA, opacity = \opacityLinkB] (s1) -- (a11);
   \draw[-{Square[scale=1.5, open, fill=white]}, semithick, color = \colorBlockB] (s1) -- (a12);
   \draw[-{Square[scale=1.5, open, fill=white]}, semithick, color = \colorBlockB] (s2) -- (a2);
   \draw[-{Square[scale=1.5, open, fill=white]}, semithick, color = \colorBlockB, opacity = \opacityLinkC] (s3) -- (a31);
   \draw[-{Square[scale=1.5, open, fill=white]}, semithick, color = \colorBlockC] (s3) -- (a32);
   \draw[-{Square[scale=1.5, open, fill=white]}, semithick, color = \colorBlockC] (s4) -- (a4);
   \draw[-{Square[scale=1.5, open, fill=white]}, semithick, color = \colorBlockC] (s5) -- (a5);

   \node[ color = \colorBlockA] (lbl_a0)  at ($(a0) +(0,-.2)$) {\tiny$+$};
   \node[ color = \colorBlockB] (lbl_a0)  at ($(a12) +(.2,0)$) {\tiny$+$};
   \node[ color = \colorBlockC] (lbl_a0)  at ($(a32) +(-.2,0)$) {\tiny$+$};
   \node[ color = \colorBlockA, opacity = \opacityLinkB] (lbl_a0)  at ($(a11) +(0,-.245)$) {\tiny$-$};
   \node[ color = \colorBlockB, opacity = \opacityLinkC] (lbl_a0)  at ($(a31) +(0,-.245)$) {\tiny$-$};
   \node[ color = \colorBlockB] (lbl_a0)  at ($(a2) +(0,-.245)$) {\tiny$-$};
   \node[ color = \colorBlockC] (lbl_a0)  at ($(a4) +(0,-.245)$) {\tiny$-$};
   \node[ color = \colorBlockC] (lbl_a0)  at ($(a5) +(0,-.245)$) {\tiny$-$};

\end{tikzpicture} & 
  \def\cutAt{3}\definecolor{orangeDIY}{rgb}{.92,.5,.11}
\definecolor{bluePY}{HTML}{1F77B4} 

\begin{tikzpicture}[auto,node distance=8mm,>=latex,scale=\tkzscl]
  \tikzset{arrowPetri/.style={>=latex,rounded corners=5pt}}

    \tikzstyle{round}=[thick,draw=black,circle,inner sep=2pt]

    \def\p{2.2}
    \def\colorBlockA{black}
    \def\colorBlockB{black}
    \def\colorBlockC{black}
    \def\opacityLinkB{1}
    \def\opacityLinkC{1}
    \ifnum\cutAt>0
      \def\colorBlockB{bluePY}
      \def\colorBlockC{bluePY}
        \def\opacityLinkB{.22}
    \fi
    \ifnum\cutAt>1
    \def\colorBlockB{black}
    \def\colorBlockC{bluePY}
      \def\opacityLinkB{1}
        \def\opacityLinkC{.22}
    \fi

    \node[round, color = \colorBlockA] (s0) at (0,0)              {\scriptsize$0$};
    \node[round, color = \colorBlockB] (s1) at ($(s0)+(0,-\p)$) {\scriptsize$1$};
    \node[round, color = \colorBlockB] (s2) at ($(s1)+(0,-\p)$)   {\scriptsize$2$};
    \node[round, color = \colorBlockC] (s3) at ($(s1)+(\p,-\p)$)  {\scriptsize$3$};
    \node[round, color = \colorBlockC] (s4) at ($(s3)+(0,-\p)$)   {\scriptsize$4$};
    \node[round, color = \colorBlockC] (s5) at ($(s4)+(0,-\p)$)   {\scriptsize$5$};

   \coordinate (a0)  at ($(s0)+(0,.5*\p)$) {};  
   \coordinate (a11) at ($(s1)+(0,.5*\p)$) {};  
   \coordinate (a12) at ($(s1)+(-.5*\p,0)$) {};  
   \coordinate (a2)  at ($(s2)+(0,.5*\p)$) {};  
   \coordinate (a31) at ($(s3)+(0,.5*\p)$) {};  
   \coordinate (a32) at ($(s3)+(.5*\p,0)$) {};  
   \coordinate (a4)  at ($(s4)+(0,.5*\p)$) {};  
   \coordinate (a5)  at ($(s5)+(0,.5*\p)$) {};  


     \draw[-{Latex[scale=1]}, rounded corners=5pt,  color = \colorBlockA] (a0) |- ($(a0)+(-.5*\p-.5,.2)$) |- (s0);
     \draw[-{Latex[scale=1]}, rounded corners=5pt,  color = \colorBlockA, opacity = \opacityLinkB] (a11) -- (s0);
     \draw[-{Latex[scale=1]}, rounded corners=5pt,  color = \colorBlockB] (a12) -| ($(a12)+(-.5,-1)$) |- (s2);
     \draw[-{Latex[scale=1]}, rounded corners=5pt,  color = \colorBlockB] (a12) -| ($(a12)+(-.5,-1)$) |- (s4);
     \draw[-{Latex[scale=1]}, rounded corners=5pt,  color = \colorBlockB, opacity = \opacityLinkC] (a31) |- (s1);
     \draw[-{Latex[scale=1]}, rounded corners=5pt,  color = \colorBlockB] (a2) -- (s1);
     \draw[-{Latex[scale=1]}, rounded corners=5pt,  color = \colorBlockC] (a4) -- (s3);
     \draw[-{Latex[scale=1]}, rounded corners=5pt,  color = \colorBlockC] (a5) -- (s4);
     \draw[-{Latex[scale=1]}, rounded corners=5pt,  color = \colorBlockC] (a32) -| ($(s5)+(.66*\p+.4,0)$) -- (s5);

   \draw[-{Square[scale=1.5, open, fill=white]}, semithick, color = \colorBlockA] (s0) -- (a0);
   \draw[-{Square[scale=1.5, open, fill=white]}, semithick, color = \colorBlockA, opacity = \opacityLinkB] (s1) -- (a11);
   \draw[-{Square[scale=1.5, open, fill=white]}, semithick, color = \colorBlockB] (s1) -- (a12);
   \draw[-{Square[scale=1.5, open, fill=white]}, semithick, color = \colorBlockB] (s2) -- (a2);
   \draw[-{Square[scale=1.5, open, fill=white]}, semithick, color = \colorBlockB, opacity = \opacityLinkC] (s3) -- (a31);
   \draw[-{Square[scale=1.5, open, fill=white]}, semithick, color = \colorBlockC] (s3) -- (a32);
   \draw[-{Square[scale=1.5, open, fill=white]}, semithick, color = \colorBlockC] (s4) -- (a4);
   \draw[-{Square[scale=1.5, open, fill=white]}, semithick, color = \colorBlockC] (s5) -- (a5);

   \node[ color = \colorBlockA] (lbl_a0)  at ($(a0) +(0,-.2)$) {\tiny$+$};
   \node[ color = \colorBlockB] (lbl_a0)  at ($(a12) +(.2,0)$) {\tiny$+$};
   \node[ color = \colorBlockC] (lbl_a0)  at ($(a32) +(-.2,0)$) {\tiny$+$};
   \node[ color = \colorBlockA, opacity = \opacityLinkB] (lbl_a0)  at ($(a11) +(0,-.245)$) {\tiny$-$};
   \node[ color = \colorBlockB, opacity = \opacityLinkC] (lbl_a0)  at ($(a31) +(0,-.245)$) {\tiny$-$};
   \node[ color = \colorBlockB] (lbl_a0)  at ($(a2) +(0,-.245)$) {\tiny$-$};
   \node[ color = \colorBlockC] (lbl_a0)  at ($(a4) +(0,-.245)$) {\tiny$-$};
   \node[ color = \colorBlockC] (lbl_a0)  at ($(a5) +(0,-.245)$) {\tiny$-$};

\end{tikzpicture}
\end{tabular}
}\end{center}
\caption{\normalfont{Left: a hierarchical Stochastic Shortest Path configuration (satisfying Assumptions~\ref{ass:hierarchy1} and~\ref{ass:hierarchy2}) with six states. We have $0\leq 1 \leq 2$ and $0\leq 1 \leq 3 \leq 4 \leq 5$, while other pairs need not to be comparable. Middle (resp.\ right): the remaining sub-graph when ruling out actions of $A_1^-$ (resp.\ $A_3^-$), with strongly connected component $\CFC{1}$ (resp.\ $\CFC{3}$) outlined in blue.}} 
\label{fig:three_hierarchies}
\end{figure}
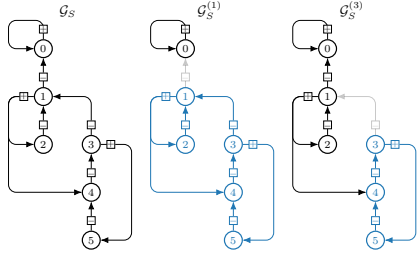

\begin{lemma}\label{lemma:accessible_from_i}
  Suppose that Assumptions~\ref{ass:fluid_SSP_configuration}, \ref{ass:hierarchy1} and~\ref{ass:hierarchy2} hold. For all $i \in \SminusZ$ such that $A_i^+\neq\emptyset$, the following properties hold:

  \begin{enumerate}[label=(\roman*)]
      \item \label{lemma:accessible_from_i:i} $0$ is not accessible from $i$ in $\mathcal{G}^{(i)}_S$,
      \item \label{lemma:accessible_from_i:ii} every state accessible from $i$ in $\mathcal{G}^{(i)}_S$ also has access to $i$ via descending actions.%
  \end{enumerate}

\end{lemma}

Let $i \in \SminusZ$ such that $A_i^+\neq \emptyset$. To prove that it cannot be more interesting in the long run to stay in the class $\CFC{i}$ than playing descending actions to reach state $0$, we focus on the particular evolution associated with~\eqref{eq:SMDP_finite_horizon} in $\CFC{i}$. Observe that the pair
  $(\CFC{i}, A_{\CFC{i}})$, with 
  \( A_{\CFC{i}}\coloneqq A_i^+\;\uplus\biguplus_{\substack{j\in\CFC{i} \\ j\neq i}} A_j\),
  is a consistent subset of states and actions in the sense of~\Cref{remark:reduced_semigroup}. 
  We thus introduce the associated evolution semigroup $\big(\Sg_t\big|_{\CFC{i},\,A_{\CFC{i}}}\big)_{t\geq 0}$, denoted for short by
  $(\widetilde{\Sg}^{(i)}_t)_{t\geq 0}$. %

\begin{lemma}\label{lemma:reduced_game}
  Suppose Assumptions~\ref{ass:hierarchy1} and~\ref{ass:hierarchy2} hold. Let $i$ in $S$ such that $A_i^+\neq \emptyset$. 
Then, %
  the growth rate of $\widetilde{\Sg}^{(i)}$ is uniform on $\CFC{i}$, \ie, of the form $\chi^{(i)}\unbf$, where $\chi^{(i)}\geq\underline{\chi}$.
\end{lemma}

\newcommand{\minMi}{M}

We are now ready to formulate the bounds on the catch-up times $(\theta_i')_{i\in S}$. Recall that these are defined in~\Cref{thm:reduction_SSP_2} under Assumption~\ref{ass:fluid_regime} for $\lambda$-sink SMDPs, or equivalently under Assumption~\ref{ass:fluid_SSP_configuration} for SMDPs in SSP configuration. In particular we have $\underline{\chi}>0$.

\begin{theorem}\label{thm:bound_SSP}
  Suppose that hierarchical Assumptions~\ref{ass:hierarchy1} and~\ref{ass:hierarchy2} hold.
  Then, we have $\theta'_0 = 0$ and for all $i\in \SminusZ$, $\theta'_i$ consistently verifies by induction:%
  \[ \theta'_i \leq  \max_{a\in A_i^-}\big\{ t^a+\max_{j \in \supp(a)} \theta'_j\big\} + \frac{\minMi}{\chi^{(i)}}\unbm_{A_i^+\neq\emptyset} ,\]
    \noindent where for all $i$ in $S$ such that $A_i^+\neq\emptyset$, %
    $\chi^{(i)}\unbf$ denotes the growth rate vector of the evolution semigroup $\widetilde{\Sg}^{(i)}$ introduced in~\Cref{lemma:reduced_game}. %
  \end{theorem}

  The proof of~\Cref{thm:bound_SSP} makes an extensive use of the fact that the function $v'$ introduced in~\Cref{thm:reduction_SSP_2} is nondecreasing. As indicated in the statement, a different behavior occurs if $A_i^+$ is empty or not. The case $A_i^+=\emptyset$ is rather easy to address: it consists in applying the inductive scheme described earlier and waiting for states lower in the hierarchy to catch up the input. In contrast, the case $A_i^+\neq\emptyset$, in which an ``adverse'' behavior may delay the catch-up by using actions of $A_i^+$, is harder to analyze. The key ingredient of the proof is to bound below the evolution of $v'$ on the class $\CFC{i}$ by an affine stationary regime of the evolution semigroup $\widetilde{\Sg}^{(i)}$. Since the latter has positive growth rate, and $v'$ eventually reaches a constant value on states located lower in the hierarchy, actions of $A_i^+$ cannot be selected for a too long duration.
  
We set $\theta^* \coloneqq \max_{i\in S}\theta'_i$ the maximum of the catch-up times, so that after time $\theta^*$, all states have caught up the input.
  
  \paragraph{\textnormal{\textbf{Example}}} We suppose that in addition to the usual calls arrival rate $\lambda$, our medical emergency call center undergoes the sudden arrival of $M$ extra calls (for instance, corresponding to an event with many casualties) at time $\underline{t}=0$. As proved in~\Cref{coro:value_cv}, if the call center is well-staffed, this peak of calls will ultimately be absorbed and the system shall return to a ``cruise regime'', where the policy $(\sigma(1),\sigma(3))=(a_1^-,a_3^-)$ is applied (recall that the latter is the only proper policy of the system). As already observed and commented in~\Cref{fig:three_hierarchies}, our SMDP modeling the call center satisfies hierarchical assumptions~\ref{ass:hierarchy1} and~\ref{ass:hierarchy2}. As a result, the previous peak is overcome in a finite time $\theta^*$ and we can apply~\Cref{thm:bound_SSP} to bound this duration. We obtain
  \begin{equation} \theta^* = \frac{M}{(\lambda_A\wedge\lambda_P) - \lambda} + t_1 + \frac{M}{\lambda_P-\lambda} + t_2 + t_3\,.\label{eq:finite_time_convergence_example}\end{equation}
  In this transience bound, we identify three terms coming from the mere communication delays between states (the times $t_1$, $t_2$ and $t_3$), and two terms originating from the states that could play ascending actions. The term $M/((\lambda_A\wedge\lambda_P) - \lambda)$ corresponds to a maximum time needed for state $1$ to choose action $a^-_1$. We check that it is proportional to the amount of extra calls to pick up and is governed by the minimum residual handling speed of agents after they performed all the usual tasks, i.e. the throughput on class $\CFC{1}$ of semigroup $(\widetilde{\Sg}_t^{(1)})_{t\geq 0}$, given by either $\lambda_A-\lambda$ or $\lambda_P-\lambda$ depending on which policy realizes the minimal average-cost between $(\sigma(1),\sigma(3))=(a_1^+,a_3^-)$ and $(\sigma(1),\sigma(3))=(a_1^+,a_3^+)$. Similarly, the term $M/(\lambda_P-\lambda)$ bounds the time needed for state $3$ to choose action $a_3^-$ after $\theta_1'+t_1$; the denominator corresponds to the throughput on class $\CFC{3}=\{3,4,5\}$ of semigroup $(\widetilde{\Sg}_t^{(3)})_{t\geq 0}$ (applying policy $\sigma(1),\sigma(3))=(a_1^+,a_3^+)$).
  
  \medskip 

  Building on~\Cref{thm:bound_SSP}, we may give a coarser but simple bound on the global catch-up time $\theta^*$ that emphasizes the tree structure of hierarchical SMDPs. We denote by $d$ the maximal length of a descending path in $S$ relatively to the ordering $(\leq)$, and by $d^+$ the maximal number of states with non-empty set of ascending playable actions along a descending path.

    \begin{corollary}\label{coro:convergence_from_below}Under conditions of~\Cref{thm:bound_SSP}, convergence towards $u^*$ occurs in a time $\theta^*$ such that
    $ \theta^* \leq d \scalebox{.87}{$\;\times\;$} \tmax \,+\, d^+ \scalebox{.87}{$\;\times\;$} \big(M/\underline{\chi}\big)$.%
    \end{corollary}
 
    This upper bound on the total time of convergence is governed by the ratio $M/\underline{\chi}$, \ie, it is proportional to the magnitude of the bulk that affected the input, and inversely proportional to the growth-rate of the inner part of the system (no taking into account the input). Rephrasing this result in terms of the $\lambda$-sink SMDPs, such as done in the example above, we get that $\underline{\chi}=\lambda'-\lambda$, where $\lambda'$ is an intrinsic throughput of the system and $\lambda$ is the input throughput. In the emergency call center example, $\lambda'$~represents the maximal admissible
    input rate of calls. E.g., in~\eqref{eq:finite_time_convergence_example}, $\lambda'= \lambda_A\wedge\lambda_P$ is the maximal input flow that does not exceed the capacities of treatment of the MRA and physicians.
    Hence, the difference $\lambda'-\lambda$ has an intuitive interpretation as a {\em security margin}, which increases with the staffing.

    Although the results of~\Cref{thm:bound_SSP} and~\Cref{coro:convergence_from_below} were brought to address the question  raised in~\Cref{thm:reduction_SSP_2} of the transience time to catch-up a specific Heaviside-type perturbation in $\lambda$-sink SMDPs, the derived bounds bear some generality. Indeed, any perturbation taking the form of an extra positive cost on the input --- or equivalently an extra negative cost on states of $\SminusZ$ --- possibly not instantaneous but spread over a time window of length smaller than $\tmax$ can be bounded by such a template. This addresses the realistic case where the inner part of the system suffers a delay with respect to the input due to the perturbation. Besides, note that our techniques also allow us to bound the transience time of going back to an input-driven regime starting ``ahead of it''. This would correspond to a ``negative bulk'' of arrivals (\eg, a reduced rate of arrivals, or an absence of arrivals, over a short time period), \ie, the opposite situation to the one considered in our motivating application.
    As an examination of our proofs shows, this case is easier to handle because it is always non-optimal to pick ascending actions, thus resulting in bounds without the $M/\underline{\chi}$ terms.\todo{SG: I explained better the last sentence, but speaking of ``negative bulk'', I think it is clearer.}

      \section*{Concluding remarks}
      We provided an explicit upper bound for the time needed for an emergency call center to absorb a bulk of calls, relying on a continuous time Petri net model. This is based on an analysis of the conditions for convergence in finite time for semi-Markov decision processes. We showed that, under a ``hierarchical'' assumption on the topology of Petri nets, satisfied by emergency call centers, there is a finite absorption time, bounded by an expression whose essential term is of the form $d^+ M/(\lambda' -\lambda)$, where $\lambda$ is the input rate, $\lambda'$ is an intrinsic throughput of the system (with an explicit monotone dependence in the staffing), assuming that $\lambda'>\lambda$, $M$ is the bulk size, and $d^+$ is a constant depending on the topology of the system but not on the staffing.
      Whereas the order $M/(\lambda'-\lambda)$ of our bound
      is optimal, we believe there is still room for improvement for the multiplicative constant $d^+$ that we obtained. We plan to address this issue in further work.

\bibliographystyle{alpha}
\bibliography{transience_bound}
\appendix
\section{Illustration of the SMDP dynamics}
\begin{figure}[htb]
  \centering
  \includegraphics[trim={5 5 0 7},clip]{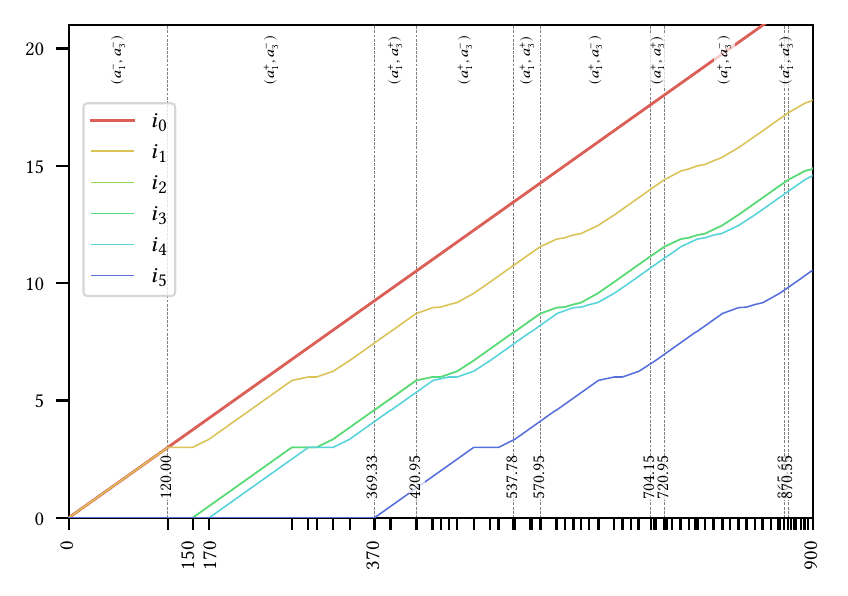}
  \caption{\textnormal{The six components of the value function $v^*$ associated with the SMDP depicted in~\Cref{fig:EMS_A_SMDP}, as functions of the time. The values of the parameters have been set to $\lambda=0.025$, $\pi = 0.3$, $t_{0}$ is arbitrary small, $t_1=150$, $t_2=20$, $t_3=200$, $N_A=3$ and $N_P = 1.375$.}}
  \label{fig:real_value_function}
  \end{figure}

We depict in~\Cref{fig:real_value_function} for a specific choice of parameters the value function $v^*$ associated to the SMDP of~\Cref{fig:EMS_A_SMDP} and previous dynamics~\eqref{SAMU-SMDP}, thus corresponding to our idealized emergency call center. In this SMDP, a total of four policies can be played at any instant (two different actions can indeed be pulled in both states $1$ and $3$), and we mention on the~\Cref{fig:real_value_function} a strategy to be played at every instant to realize $v^*$ in terms of these policies. %
For instance, $v^*(400)$ can be achieved by applying policy $(\sigma(1),\sigma(3))=(a_1^-,a_3^-)$ for 120 seconds, then $(a_1^+,a_3^-)$ until $t\simeq369.33\,\mathrm{s}$ and $(a_1^+,a_3^+)$ until $t=400\,\mathrm{s}$.

\section{Illustration of the evolution semigroup}
The action of the semigroup $\Sg$ is illustrated in \Cref{fig:sin_value_function}.
\begin{figure}[htb]
  \centering
  \includegraphics[trim={5 7 0 7},clip]{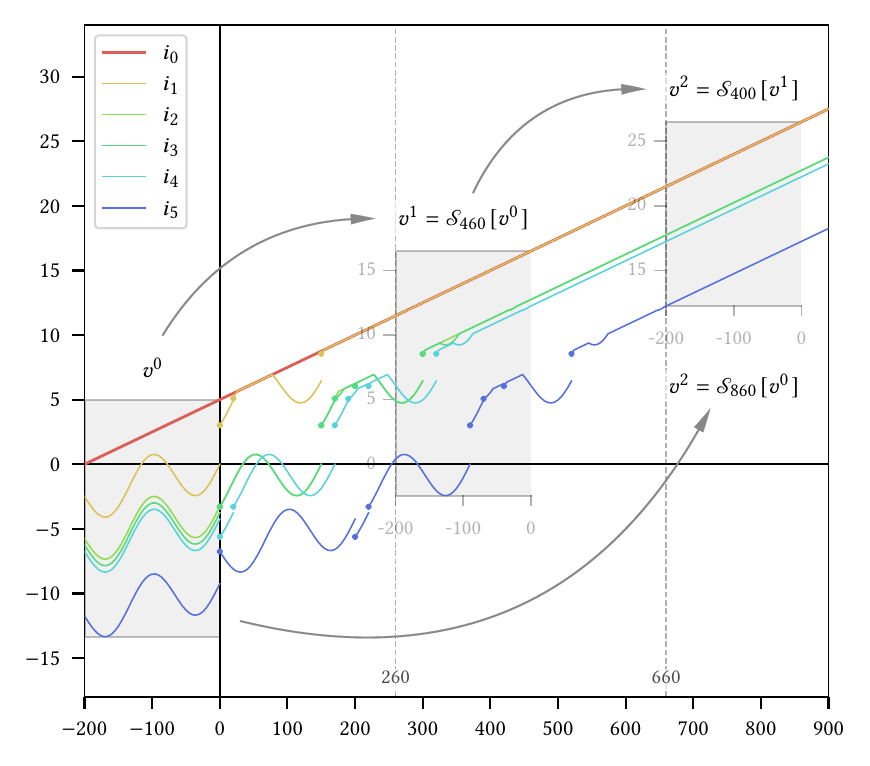}
  \caption{\textnormal{The dynamics~\eqref{eq:SMDP_finite_horizon} of the SMDP of~\Cref{fig:EMS_A_SMDP} applied on an arbitrary initial condition $v^0$ of $\Vcal$ (given by a linear function on coordinate $0$ and by the sum of linear functions with a sine wave for other states), and the action of the semigroup $\Sg$ introduced in~\Cref{def:semigroup} illustrated for evolution periods of 450, 430 and 880 units of time. Parameter values are identical to~\Cref{fig:real_value_function} except $N_A=7$ and $N_P=3.5$. The finite time convergence to a stationary regime is visible on the figure.}}
  \label{fig:sin_value_function}
  \end{figure}

\section{Proof of the results of Section~\ref{ssec:well-posedness}}

The aim of this appendix is to provide the somewhat technical proofs of the main results of Section~\ref{sec:intro_aux_SMDP}. The main difficulty of our SMDP model is that one may perform several moves and incur the corresponding costs in zero time. 

We first show that the finite horizon SMDP problem is well posed. 
\begin{theorem}\label{thm:value_well_defined}
  Under Assumption~\ref{ass:nonzeno}, the value-function $v^*$ is well-defined.
\end{theorem}
We can actually show that Assumption~\ref{ass:nonzeno} cannot be weakened.

\begin{lemma}\label{lemma:contains_stochastic_matrix}
  Let $P\in\R^{n\times n}$ be a stochastic matrix (thought as a Markov chain on $\{1,\dots,n\}$). We denote by $\mathcal{F}$ the set of its final classes. Let $J$ be a subset of $\{1,\dots,n\}$ and $P\big|_{J\times J}$ be the square submatrix of $P$ restricted to $J$. Then, this submatrix has spectral radius $1$ if and only if it contains a stochastic submatrix, or alternatively
  \[\rho\big(P\big|_{J\times J}\big) = 1 \;\Longleftrightarrow\; \exists F \in \mathcal{F},\quad F \subset J \]
  \end{lemma}
  \begin{proof}
    We take advantage of the decomposition of $P$ adapted to the partition of $\{1,\dots,n\}$ into the transient states $T$ of $P$ and the final classes $F_1$, $F_2$, \dots, $F_k$ of $P$:
    \[ P= \begin{pmatrix}
      \;Q_0\; & \;Q_1\;   & \;Q_2\;    &  \cdots & \;Q_k\; \\
        0 & P_1   &  0     &  \cdots & 0 \\
        0 &   0   & P_2    &  \cdots & 0 \\[-.5ex]
    \vdots & \vdots& \vdots &  \ddots & 0  \\
    0     & 0     & 0      & \cdots  & P_k 
  \end{pmatrix}\,,\]
  where $P_1$,\dots,$P_k$ are stochastic matrices associated with irreducible Markov chains, and $Q_0$ is a substochastic matrix. It is clear that if we have $F_{\ell}\subset J$ for some $1\leq \ell\leq k$, we can find a right eigenvector associated with eigenvalue $1$ for $P\big|_{J\times J}$. Conversely, since any strict submatrix of $P_1$,\dots, $P_k$ is substochastic and then has spectral radius strictly less than $1$, $P\big|_{J\times J}$ cannot admit an eigenvector associated with $1$ if $J$ doesn't contain the whole support of some $P_{\ell}$ for $1\leq \ell\leq k$, i.e. a final class of $P$.
  \end{proof}

\subsection{Well-posedness of the value function (\Cref{thm:value_well_defined})}\label{appendix:Nt}

To show that $\sup_{f\in\Fbb}\mathbb{E}_i^f(\widehat{N}_t)<\infty$, we shall actually study an auxiliary Markov Decision Process (in which the time is therefore not taken into account) with larger state space and equipped with a \emph{layered} structure.

Let $K$ be a positive integer, and define for all $k$ in $\{1,\dots,K\}$ the states of the $k$-th layer:
\[ S^{(k)} \coloneqq \Big\{ i_1^{(k)},i_2^{(k)},\dots i_{|S|}^{(k)}\Big\}\,.\]
We may think of all these layers as exact copies of the initial state space $S$ of our SMDP and so we have natural projections $(\pi_k:S^{(k)}\to S)_{1\leq k\leq K}$ mapping a state of any layer to the state of $S$ from which it is derived (we shall drop the index $k$ and only denote these projections by $\pi$ when nonambiguous). We also introduce a \emph{cemetery} state denoted by $\bot$, so that the complete state space $\widetilde{S}$ of our new MDP writes
\[ \widetilde{S} \coloneqq \Big( \bigcup_{k=1}^K S^{(k)}\Big)\cup \{\bot\}\,.\]
The set $S$ of states of the original SMDP can be partitioned into two sets $S_Z$ and $S_{NZ}$, respectively the \emph{possibly Zeno states}, so that $S_Z\coloneqq \{i\in S\;|\; \exists a \in A_i,\;\; t^a = 0\}$ and the \emph{necessarily non-Zeno states} $S_{NZ} \coloneqq \{i\in S\;|\;\forall a\in A_i,\;\;t^a > 0\}$. This decomposition naturally carries over to the layers $(S^{(k)})_{1\leq k\leq K}$ and we shall later use the partition $S^{(k)}=S^{(k)}_Z \cup S^{(k)}_{NZ}$ for all $k\in\{1,\dots, K\}$.

As for actions, we convene that from the cemetery state, only one action $a_{\bot}$ can be pulled (and then $A_{\bot} = \{a_{\bot}\}$). From any other state $i^{(k)}$ within the $k$-th layer, the player can pick exactly as many actions as in $A_{\pi(i^{(k)})}$:
\[ A_{i^{(k)}} \coloneqq \Big\{ a_1^{(k)},a_2^{(k)},\dots a_{|A_{\pi(i^{(k)})}|}^{(k)}\Big\}\,.\]
Once again there are natural projections from these newly created sets of playable actions towards those of the original SMDP, and for the sake of simplicity we still denote these maps by $\pi$.%

For all $i^{(k)}$ in $S^{(k)}$, the transition probabilities after playing action $a$ in $A_{i^{(k)}}$ are such that 
\[ \text{if}\;\,t^{\pi(a)} > 0,\;\,\text{then}\quad p^a_j\coloneqq \left\{\begin{array}{ll} p^{\pi(a)}_{\pi(j)} & \text{if}\;\, k < K\;\,\text{and}\;\,j\in S^{(k+1)}\\ 1 & \text{if}\;\,k=K\;\,\text{and}\,\; j=\bot \\ 0 & \text{otherwise}\end{array}\right.\]
\[ \text{if}\;\,t^{\pi(a)} = 0,\;\,\text{then}\quad p^a_j\coloneqq \left\{\begin{array}{ll} p^{\pi(a)}_{\pi(j)} & \text{if}\;\, j\in S^{(k)} \\ 0 & \text{otherwise}\end{array}\right.\]
and finally $p^{a_{\bot}}_{\bot} = 1$ (the cemetery state is absorbing).

In other words, all actions in the new game derived from Zeno actions in the original SMDP make the player stay in the current layer, while picking some new action derived from a non-Zeno initial action forces to dive into the next layer (and to join the cemetery state if already in the $K$-th layer). 

Unit rewards $r^a \coloneqq 1$ are attached to all the actions $a\in\widetilde{A}\setminus\{a_{\bot}\}$ of this new MDP except in the cemetery state $\bot$ where no more rewards can be earned, i.e. $r^{a_{\bot}}\coloneqq 0$.

Furthermore, the admissible trajectories $\widetilde{H}_{\infty}$ associated with the new game correspond to those of the initial SMDP $H_{\infty}$, up to potentially hitting $\bot$ after which a single action can be chosen, so the strategies of the original game precisely induce all the strategies of the new one. Since any $f$ in $\Fbb$ and initial state $i$ in $\widetilde{S}$ define a random process $(i,\widehat{a}_0,\widehat{r}_0,\widehat{i}_1,\widehat{a}_1,\widehat{r}_1,\dots)$, we can introduce the vector $N^{\Kbox}$ of $\R^{|\widetilde{S}|}_{\geq 0}$ where 
\begin{equation}\label{eq:number_of_jumps} \text{for all}\;\,i\in\widetilde{S},\quad N^{\Kbox}(i) \coloneqq \sup_{f\in\Fbb} \;\mathbb{E}^f_i\bigg(\sum_{k=0}^{\infty}\widehat{r}_k\bigg)\,.\end{equation}
 
\begin{lemma}\label{lemma:sandwich_N}
Let $t\geq\tmin$ and denote $K\coloneqq \lceil t/\tmin\rceil$. We have for all $i$ in $S$:
\[ N^{\scalebox{.72}{$[1]$}}(i^{(1)}) \;\leq\; \sup_{f\in\Fbb}\mathbb{E}^f_i(\widehat{N}_t) \;\leq\; N^{\Kbox}(i^{(1)})\,,\]
where the second inequality actually holds for all $t\geq 0$.
\end{lemma}
\begin{proof}
  From what preceeds, we understand that for all $i$ in $S$ and $k$ in $\{1,\dots,K\}$ for the $K$-layered MDP, the quantity $N^{\Kbox}(i^{(k)})$ gives the maximum expected number of moves starting from $i$ and before performing $K+1-k$ non-Zeno moves in the original SMDP. 

In particular for any admissible trajectory of the process, the number $\widehat{N_t}$ of moves performed in this original SMDP before exceeding a total physical elapsed time $t$ must be less than the number of moves before pulling $K$ non-Zeno actions (costing a total time of at least $K\tmin\geq t$ units of time), hence the second inequality. On the other hand if $t\geq \tmin$, it requires to pull at least one non-Zeno action before possibly exceeding a total physical elapsed time of $t$, hence the first inequality. 
\end{proof}

We use hereafter the notation $\unbf$ to denote the vector with all coordinates equal to $1$.

\begin{lemma}\label{lemma:N_is_a_fixed_point_of_B}
Denote $\widetilde{\SSigma}\coloneqq \prod_{i\in\widetilde{S}}\widetilde{A}_i$ the set of policies (or decision rules) in our new MDP. 
Similarly to the original SMDP, the matrix with components $(p^{\widetilde{\sigma}(i)}_j)_{i,j\in \widetilde{S}}$ is denoted by $P^{\widetilde{\sigma}}$. We introduce the self-map $B$ of \(\R^{|\widetilde{S}\setminus \{\bot\}|}\) 

\[ B(x) = \unbf + \max_{\widetilde{\sigma}\in\widetilde{\SSigma}} \big\{P^{\widetilde{\sigma}}\big|_{(\widetilde{S}\setminus \{\bot\})\times (\widetilde{S}\setminus \{\bot\})} x\big\}\]

Then, the vector $N^{\Kbox}$ with its $\bot$-coordinate dropped is a fixed-point of $B$ :
\[ N^{\Kbox}\big|_{\widetilde{S}\setminus\{\bot\}} = B\big(N^{\Kbox}\big|_{\widetilde{S}\setminus\{\bot\}}\big)\]
\end{lemma}
\begin{proof}
  We recognize in~\eqref{eq:number_of_jumps} an expected total reward as defined in~\cite[Section 7.1.1]{puterman2014markov}. All the rewards are nonnegative so~\cite[Theorem 7.1.3]{puterman2014markov} applies and we know that $N^{\Kbox}$ verifies an optimality equation of the form 
  \begin{equation}\label{eq:puterman_total} N^{\Kbox} = \max_{\widetilde{\sigma}\in\widetilde{\SSigma}}\big\{r^{\widetilde{\sigma}}+P^{\widetilde{\sigma}} N^{\Kbox} \big\}\,,\end{equation}
  Remark that the vector $r^{\widetilde{\sigma}}$ does not actually depend on $\widetilde{\sigma}\in\widetilde{\SSigma}$ (it has only unit components except on the last coordinate where it is null), and above all, any $P^{\widetilde{\sigma}}$ has the form 

  \hspace{-.2cm}\resizebox{.48\textwidth}{!}{$
  \begin{blockarray}{ccccccccc}
      \begin{block}{(cccccccc)c}
          & & & & & & & & &\\[-2ex]
          P_{Z,Z}^{(1)} & P_{Z,NZ}^{(1)} & Q_{Z,Z}^{(1)} & Q_{Z,NZ}^{(1)} & \cdots & 0 & 0 & 0 & \scalebox{.8}{$S_Z^{(1)}   $} \\[2ex]
          0 & 0 & P_{NZ,Z}^{(1)} & P_{NZ,NZ}^{(1)} & \cdots & 0 & 0 & 0 & \scalebox{.8}{$S_{NZ}^{(1)}$} \\[2ex]
          0 & 0 & P_{Z,Z}^{(2)} & P_{Z,NZ}^{(2)} & \cdots & 0 & 0 & 0 & \scalebox{.8}{$S_Z^{(2)}   $} \\[2ex]
          0 & 0 & 0 & 0 & \cdots & 0 & 0 & 0 & \scalebox{.8}{$S_{NZ}^{(2)}$} \\[2ex]
          \vdots & \vdots & \vdots & \vdots & \ddots & \vdots &\vdots & \vdots & \vdots \\[2ex] 
          0 & 0 & 0 & 0 & \cdots & P_{Z,Z}^{(K)} & P_{Z,NZ}^{(K)} & Q_{Z,\bot}^{(K)} & \scalebox{.8}{$S_Z^{(K)}   $} \\[2ex]
          0 & 0 & 0 & 0 & \cdots & 0 & 0 & \unbf & \scalebox{.8}{$S_{NZ}^{(K)}$} \\[2ex]
          0 & 0 & 0 & 0 & \cdots & 0 & 0 & 1 & \scalebox{.8}{$\bot$} \\[2ex]
      \end{block}
      \scalebox{.8}{$S_Z^{(1)}$} & \scalebox{.8}{$S_{NZ}^{(1)}$} & \scalebox{.8}{$S_Z^{(2)}$} & \scalebox{.8}{$S_{NZ}^{(2)}$} & \cdots & \scalebox{.8}{$S_Z^{(K)}$} & \scalebox{.8}{$S_{NZ}^{(K)}$} & \scalebox{.8}{$\bot$} & \\
  \end{blockarray}$
   }
  where all the blocks depend on $\widetilde{\sigma}$. Indeed, we already justified that playing actions from necessarily non-Zeno states forces to join the next layer, while playing actions from possibly Zeno states can either result in staying on the current layer (if a Zeno action is chosen) or going to the next one (otherwise). 

  We know that $N^{\Kbox}$ satisfies~\eqref{eq:puterman_total} and that $N^{\Kbox}(\bot)=0$ by the definition~\eqref{eq:number_of_jumps}. 
  This last fact actually makes the above two blocks $Q^{(K)}_{Z,\bot}$ and $\unbf$ in the last column not active and we may just write $N^{\Kbox}\big|_{S^{(K)}_Z} = \unbf + P^{(K)}_{Z,Z} N^{\Kbox}\big|_{S^{(K)}_Z} + P^{(K)}_{Z,NZ} N^{\Kbox}\big|_{S^{(K)}_{NZ}}$ and $N^{\Kbox}\big|_{S^{(K)}_{NZ}} = \unbf$. The equation~\eqref{eq:puterman_total} can then be simply refined in the one claimed by the lemma.
\end{proof}

The following proposition is a particular case resulting of operator-based techniques from~\cite{akian2011collatz,akian2019solving}, that shall be very useful in several of our proofs.
\begin{proposition}\label{prop:fixed_point_if_contraction}
  Let $n$ in $\N$ be a positive integer, $\mathcal{E}$ a finite set. Let $(C^{(k)})_{k\in\mathcal{E}}$ be nonnegative matrices of $\R^{n\times n}$ and $(d^{(k)})_{k\in\mathcal{E}}$ be vectors of $\R^n$.%
  
  We introduce the following two piecewise-affine selfs map of $\R^n$:
  \[ \textforall x \textin \R^n, \quad C^{\max}(x) \coloneqq \max_{k\in\mathcal{E}}\big\{ C^{(k)} x + d^{(k)}\big\}\,,\]
  \[ \textforall x \textin \R^n, \quad C^{\min}(x) \coloneqq \min_{k\in\mathcal{E}}\big\{ C^{(k)} x + d^{(k)}\big\}\,,\]
  where the minimum is taken componentwise.
  For $(k_1,\dots,k_n)$ in $\mathcal{E}^n$, we denote by $C^{(k_1\cdots k_n)}$ the $n\times n$-matrix whose $i$-th row is given by the $i$-th row of $C^{(k_i)}$.

    If all matrices $(C^{(k_1\cdots k_n)})_{(k_1,\dots,k_n)\in\mathcal{E}^n}$ have a spectral radius strictly less than one, then
    there exists $\lambda$ in $[0,1)$, and $u$ in $\R^n_{>0}$ such that $C^{\max}$ and $C^{\min}$ are $\lambda$-contracting in weight sup-norm $\Vert\cdot\Vert_u$, and thus
    both $C^{\max}$ and $C^{\min}$ admit a (unique) fixed-point in $\R^n$.
\end{proposition}
\begin{proof}
  We suppose that for all $(k_1,\cdots,k_n)$ in $\mathcal{E}^n$, we have $\rho\big(C^{(k_1\cdots k_n)}\big)<1$ and we want to show that $C^{\max}$ and $C^{\min}$ admit a fixed-point.

  We first introduce the homogeneous part $\check{C}^{\max}$ (resp.\ $\check{C}^{\min}$) of operator $C^{\max}$ (resp.\ $C^{\min}$) defined by $\check{C}^{\max}(x)\coloneqq \max_{k\in\mathcal{E}}\big\{C^{(k)} x\big\}$ (resp.\ $\check{C}^{\min}(x)\coloneqq \min_{k\in\mathcal{E}}\big\{C^{(k)} x\big\}$) for $x$ in $\R^n$. We want to apply~\cite[Theorem 1]{akian2019solving}, which claims that for $\lambda\geq 0$ and $u\in\R^n_{>0}$, the order-preserving operator $C^{\max}$ (resp.\ $C^{\min}$) of $\R^n$ is $\lambda$-contracting in weighted sup-norm $\Vert\cdot\Vert_u$ if and only if $\check{C}^{\max}(u)\leq\lambda u$ (resp.\ $\check{C}^{\min}(u)\leq\lambda u)$. Actually, since for all $x$ in $\R^n_{\geq 0}$, we have $\zerobf\leq \check{C}^{\min}(x)\leq\check{C}^{\max}(x)$, we shall just focus on $\check{C}^{\max}$ in what follows.

  Although it is not linear (as a supremum of linear maps), we may define a notion of spectral radius for $\check{C}^{\max}$ by 
  \[\rho(\check{C}^{\max}) \coloneqq \sup\big\{\lambda \geq 0\;\big|\; \exists x \in \R^n_{\geq 0}\setminus\{0\},\;\, \check{C}^{\max}(x) = \lambda x\big\}\]
  
  Since all the maps $x\mapsto C^{(k)}x$ are order-preserving on $\R^n_{\geq 0}$, we obtain according to~\cite[Proposition 8.1]{akian2011collatz} that 
  \[ \rho(\check{C}^{\max}) = \max_{(k_1,\dots,k_n)\in\mathcal{E}^n}\,\rho\big(C^{(k_1\cdots k_n)}\big)\,.\]
  (indeed, to satisfy the selection property needed in~\cite[Proposition 8.1]{akian2011collatz}, we here must take into account all the matrices with rows shuffled).

  Therefore, our initial assumption entails that $\rho(\check{C}^{\max})<1$. However, from the generalized Collatz-Wielandt theorem (see~\cite[Theorem 1.1]{akian2011collatz}), it also holds that
  \[ \rho(\check{C}^{\max}) = \inf \big\{ \lambda > 0\;\big|\; \exists x\in \R^n_{>0},\;\, \check{C}^{\max}(x)\leq\lambda x\big\}\,,\]
  so by taking a specific $\lambda$ in $[\rho(\check{C}^{\max}),1)$, we know that there exists a positive vector $u$ such that $\zerobf\leq \check{C}^{\max}(u)\leq \lambda u$. From our previous bound on $\check{C}^{\min}$ and by applying~\cite[Theorem 1]{akian2019solving}, we conclude that the (no longer homogeneous) operators $C^{\max}$ and $C^{\min}$ are $\lambda$-contracting in the weighted sup-norm relative to $u$ with $0\leq \lambda < 1$.
  It immediately follows that both of them admit a (unique) positive fixed-point in $\R^n$.%
  \end{proof}

\begin{proposition}\label{prop:technical1}
The following are equivalent:
\begin{enumerate}[label=(\arabic*)]
\item\label{prop:technical1:item1} the operator $B$ defined in~\Cref{lemma:N_is_a_fixed_point_of_B} admits a (unique) finite positive fixed-point,
\item\label{prop:technical1:item2} for all $\widetilde{\sigma}$ in $\widetilde{\SSigma}$, $\rho\big(P^{\widetilde{\sigma}}\big|_{(\widetilde{S}\setminus \{\bot\})\times (\widetilde{S}\setminus \{\bot\})}\big) < 1$
\item\label{prop:technical1:item3} for all $\sigma$ in $\SSigma$, $\rho\big(P^{\sigma}\big|_{S_Z\times S_Z}\big) < 1$
\item\label{prop:technical1:item4} the conditions of Assumption~\ref{ass:nonzeno} are fulfilled.
\end{enumerate}
\end{proposition}
\begin{proof}

  \udensdot{The implication~\ref{prop:technical1:item1}\;$\Longrightarrow$\;\ref{prop:technical1:item2}} is an application of~\Cref{prop:fixed_point_if_contraction} to the particular case of the operator $B$ defined in~\Cref{lemma:N_is_a_fixed_point_of_B}, observing that working with policies causes the collection $(P^{\widetilde{\sigma}})_{\widetilde{\sigma}\in\widetilde{\SSigma}}$ to already contain all the matrices that can be constructed by the row-selection procedure. The unique fixed-point of $B$ must lie in $\R^n_{>0}$ because the ``affine part'' of operator $B$ (the terms $(d^{(k)})_{k\in\mathcal{E}}$ of~\Cref{prop:fixed_point_if_contraction}) is positive.

  \udensdot{We prove~\ref{prop:technical1:item2}\;$\Longrightarrow$\;\ref{prop:technical1:item1}} by supposing that $B$ admits a finite positive fixed-point $x^{\sbullet}$ and assuming by contradiction that there is a particular $\widetilde{\sigma}$ in $\widetilde{\SSigma}$ such that $\rho\big(P^{\widetilde{\sigma}}\big|_{(\widetilde{S}\setminus \{\bot\})\times (\widetilde{S}\setminus \{\bot\})}\big) = 1$. By~\Cref{lemma:contains_stochastic_matrix}, the latter matrix contains a stochastic square submatrix, which then admits an invariant measure (i.e. a nonnegative left eigenvector for the eigenvalue $1$). Up to completing this invariant measure by with zeroes, we can construct a nonnegative left eigenvector $\mu\in\R_{\geq 0}^{|\widetilde{S}\setminus\{\bot\}|}$ for the value $1$ for $P^{\widetilde{\sigma}}\big|_{(\widetilde{S}\setminus \{\bot\})\times (\widetilde{S}\setminus \{\bot\})}$. Moreover, the fixed-point property guarantees that 
  \[ x^{\sbullet} \geq P^{\widetilde{\sigma}}\big|_{(\widetilde{S}\setminus \{\bot\})\times (\widetilde{S}\setminus \{\bot\})} x^{\sbullet} + \unbf\,.\]
  Taking the dot-product with the nonnegative vector $\mu$ on both sides leads to $\zerobf \geq \mu^T \unbf$, however $\mu^T \unbf > 0$ (most commonly $\mu$ is chosen normalized so that $\mu^T \unbf = 1$), hence the contradiction.

\udensdot{To obtain~\ref{prop:technical1:item2}$\;\Longleftrightarrow$\;\ref{prop:technical1:item3}}, we remark for example by computing the characteristic polynomial of the big triangular matrix written in the proof of~\Cref{lemma:N_is_a_fixed_point_of_B} that for $\widetilde{\sigma}\in\widetilde{\SSigma}$, we have by denoting $(\sigma_1,\dots,\sigma_K)\in\SSigma^K$ the restrictions of $\widetilde{\sigma}$ to layers $(S^{(k)})_{1\leq k\leq K}$ that  
\[ \rho\big(P^{\widetilde{\sigma}}\big|_{(\widetilde{S}\setminus \{\bot\})\times (\widetilde{S}\setminus \{\bot\})}\big) = \max_{1\leq k\leq K} \rho\big(P^{(k)}_{Z,Z}\big) = \max_{1\leq k\leq K} \rho\big(P^{\sigma_k}\big|_{S_Z\times S_Z}\big)\,,\]
and as a result
\[ \max_{\widetilde{\sigma}\in\widetilde{\SSigma}}\rho\big(P^{\widetilde{\sigma}}\big|_{(\widetilde{S}\setminus \{\bot\})\times (\widetilde{S}\setminus \{\bot\})}\big) = \max_{\sigma\in\SSigma} \rho\big(P^{\sigma}\big|_{S_Z\times S_Z}\big)\,.\]

\udensdot{The equivalence~\ref{prop:technical1:item3}$\;\Longleftrightarrow\;$\ref{prop:technical1:item4}} follows immediately from~\Cref{lemma:contains_stochastic_matrix} since by negation its claim for all policy we have
\[\forall \sigma\in\SSigma,\;\, \rho\big(P^{\sigma}\big|_{S_Z\times S_Z}\big) < 1 \;\Longleftrightarrow\; \forall\sigma\in\SSigma,\;\,\forall F \in \mathcal{F}(\sigma),\;\, F \not\subset S_Z\,,\]
the last assertion corresponding exactly to any final class of any policy having at least one positive sojourn time, i.e. Assumption~\ref{ass:nonzeno}. 
\end{proof}

The previous results allow us to quickly demonstrate~\Cref{thm:value_well_defined}
\begin{proof}[Proof of~\Cref{thm:value_well_defined}]
\udensdot{Suppose that Assumption~\ref{ass:nonzeno} holds.}
Then, by~\Cref{prop:technical1}, the operator $B$ defined in~\Cref{lemma:N_is_a_fixed_point_of_B} admits a unique positive fixed-point, which is finite. However,~\Cref{lemma:N_is_a_fixed_point_of_B} tells us that this fixed-point coincides with $N^{\Kbox}\big|_{\widetilde{S}\setminus\{\bot\}}$. It follows that $N^{\Kbox}$ is finite for all positive $K$ in $\N$. The~\Cref{lemma:sandwich_N} therefore guarantees that for all $i$ in $S$ and $t\geq 0$ that \udensdot{$\sup_{f\in\Fbb}\mathbb{E}^f_i\big(\widehat{N}_t\big) < \infty$}.

\udensdot{Conversely, if Assumption~\ref{ass:nonzeno} is not verified}, we reuse~\Cref{prop:technical1} to conclude that $N^{\scalebox{.72}{$[1]$}}$ must have non-finite value on some coordinate $i$ in $S$. \Cref{lemma:sandwich_N} entails that for the same $i$ and some $t\geq \tmin$ we have \udensdot{$\sup_{f\in\Fbb}\mathbb{E}^f_i\big(\widehat{N}_t\big) = \infty$}.
\end{proof}

\begin{figure}
  \def\tkzscl{.5}
  \begin{tikzpicture}[auto,node distance=8mm,>=latex,scale=\tkzscl]
  \tikzset{arrowPetri/.style={>=latex,rounded corners=5pt}}

    \tikzstyle{round}=[thick,draw=black,circle,inner sep=2pt]

    \def\p{5}
    \def\q{1.5}
    \def\r{1}

    \node[round] (s1) at (0,0)                {\scriptsize$1$};
    \node[round] (s2) at ($(s1)+(\p,-\r)$)    {\scriptsize$2$};
    \node[round] (s3) at ($(s1)+(\r,-1*\p)$)   {\scriptsize$3$};

   \coordinate (a11) at ($(s1)+(\q,0)$) {};  
   \coordinate (a12) at ($(s1)+(0,-\q)$) {};  
   \coordinate (a21) at ($(s2)+(-\q,0)$) {};  
   \coordinate (a22) at ($(s2)+(0,-\q)$) {};  
   \coordinate (a31) at ($(s3)+(0,\q)$) {};  
   \coordinate (a32) at ($(s3)+(\q,0)$) {};  

    \node (lbl_a11) at ($(a11)+(-.1*\r,.7*\r)$) {\tiny$a_{11}:(*,0)$};
    \node (lbl_a12) at ($(a12)+(-.7\r,-.3)$) [rotate=90]{\tiny$a_{12}:(*,0)$};
    \node (lbl_a21) at ($(a21)+(.2*\r,.6*\r)$) {\tiny$a_{21}:(*,3)$};
    \node (lbl_a22) at ($(a22)+(-1.2*\r,0)$) {\tiny$a_{22}:(*,0)$};
    \node (lbl_a31) at ($(a31)+( 1.3*\r,0)$) {\tiny$a_{31}:(*,1)$};
    \node (lbl_a32) at ($(a32)+( 1*\r,.5*\r)$) {\tiny$a_{32}:(*,2)$};   
   
    \node (prob1) at ($(a11)+( 1.2*\r,1.6*\r)$) [rotate=45]{\tiny$1/3$};   
    \node (prob1) at ($(s2)+( 0,1.3*\r)$) [rotate=45]{\tiny$2/3$};   
    \node (prob1) at ($(a22)+( .5*\r,-1.1*\r)$) [rotate=45]{\tiny$1/2$};   
    \node (prob1) at ($(a22)+( -.5*\r,-1.1*\r)$) [rotate=45]{\tiny$1/2$};   
    \node (prob1) at ($(a32)+( 1.3*\r,-1.2\r)$) [rotate=45]{\tiny$3/5$};   
    \node (prob1) at ($(a32)+( 3.3*\r,-.4\r)$) [rotate=45]{\tiny$2/5$};   
   
     \draw[-{Latex[scale=1]}, rounded corners=5pt, line width=.12mm] (a11) -| (s2);
     \draw[-{Latex[scale=1]}, rounded corners=5pt, line width=.12mm] (a11) -| ($(s1)+(2.5*\r,1.5*\r)$) -| (s1);
     \draw[-{Latex[scale=1]}, rounded corners=5pt, line width=.12mm] (a12) |- (s3);
     \draw[-{Latex[scale=1]}, rounded corners=5pt, line width=.12mm] (a31) |- ($(s2)+(-\r,-\r)$) -- (s2);
     \draw[-{Latex[scale=1]}, rounded corners=5pt, line width=.12mm] (a22) |- ($(s3)+( \r, \r)$) -- (s3);
     \draw[-{Latex[scale=1]}, rounded corners=5pt, line width=.12mm] (a22) |- ($(a22)+(\r, -1.5*\r)$) |- (s2);
     \draw[-{Latex[scale=1]}, rounded corners=5pt, line width=.12mm] (a21) -- ($(s1)+( \r,-\r)$) -- (s1);
     \draw[-{Latex[scale=1]}, rounded corners=5pt, line width=.12mm] (a32) -| ($(s2)+( 2*\r,0)$) -- (s2);
     \draw[-{Latex[scale=1]}, rounded corners=5pt, line width=.12mm] (a32) -| ($(s3)+( 2.5*\r,-1.5*\r)$) -| (s3);

   \draw[-{Square[scale=1.5, open, fill=white]}, thick] (s1) -- (a11);
   \draw[-{Square[scale=1.5, open, fill=white]}, thick] (s1) -- (a12);
   \draw[-{Square[scale=1.5, open, fill=white]}, thick] (s2) -- (a21);
   \draw[-{Square[scale=1.5, open, fill=white]}, thick] (s2) -- (a22);
   \draw[-{Square[scale=1.5, open, fill=white]}, thick] (s3) -- (a31);
   \draw[-{Square[scale=1.5, open, fill=white]}, thick] (s3) -- (a32);
\end{tikzpicture}
\end{figure} 
\begin{figure}
  \def\tkzscl{.5}
  \input{./TikZ/zeno_layered.tex} 
\end{figure}

\begin{remark}
  The layered MDP constructed in this subsection is a Stochastic Shortest Path configuration (or should we say a ``Stochastic Longest Path'' since we seek to maximize the number of moves before hitting $\bot$). The assertion~\ref{prop:technical1:item2} of~\Cref{prop:technical1} essentially requires that all the policies of this SSP are proper, hence providing contraction properties used in the proof as observed in~\cite[Proposition 1]{bertsekas1991analysis}.
\end{remark}

\subsection{Dynamics of the value function}\label{appendix:dynamics}

\begin{proof}[Proof of~\Cref{prop:uniquely_determined}]
  We prove the result by showing correctness of the following inductive scheme: ``let $t_0\geq 0$, if $v$ is known on $[-\tmax, t_0)$, then $v$ is uniquely determined on $[t_0,t_0+\tmin)$'', since it suffices to apply this reasoning step by step for $t_0 \in \tmin \N$ to prove the proposition.

  Suppose then that $t_0\geq 0$ is fixed and $v$ is known on $[-\tmax, t_0)$. Let $t\in[t_0,t_0+\tmin)$ be a fixed instant as well. We reuse the partition of the state space $S=S_Z\uplus S_{NZ}$ introduced in the previous subsection. We shall also denote $A_{i,Z}\coloneqq \{ a\in A_i\;|\; t^a = 0\}\neq\emptyset$ if $i$ is in $S_Z$.

  If $i\in S_{NZ}$, $v(i,t)$ is explicitly determined by the~\eqref{eq:SMDP_finite_horizon} equation 
  \begin{equation}\label{eq:vNZ_explicit}
    v(i,t) = \min_{a\in A_i} \bigg\{ c^{a} + \sum_{j\in S} p^{a}_{j}\,v(j, t-t^{a}) \bigg\} \, ,
  \end{equation}
  for by definition, all the $(t^a)_{a\in A_i}$ are positive and therefore the $(t-t^a)_{a\in A_i}$ lie in $[-\tmax+t_0,t_0)$. However, the restriction $v_Z(t) \coloneqq (v(i,t))_{i\in S_Z}$ of $v$ to states of $S_Z$ is solution of the implicit equation     
  \[ \quad v_Z(t) = \varphi_t( v_Z(t))\coloneqq \min\big( \, C_{Z,t}(v_Z(t))\,,\, w_t \,\big) \,,\]
  where we have
  \[ \text{for all}\;\,i\in S_Z,\quad \big[w_t\big]_i\coloneqq \min_{a\in A_i\setminus A_{i,Z}} \bigg\{ c^{a} + \sum_{j\in S} p^{a}_{j}\,v(j, t-t^{a}) \bigg\}\,,\]
  \[ \text{for all}\;\,a\in \biguplus_{i\in S_Z}A_{i,Z},\quad c^a_t \coloneqq c^{a} + \sum_{j\in S_{NZ}} p^{a}_{j}\,v(j, t)\,,\,\textrm{and}\]
  \[ \text{for all}\;\,i\in S_Z,\quad \big[C_{Z,t}(x)\big]_i \coloneqq \min_{a\in A_{i,Z}} \bigg\{ c^{a}_t + \sum_{j\in S_Z} p^{a}_{j}\,x_j \bigg\}\,.\]
  
  Observe that $w_t$ and the $(c^a_t)$ actually depend on $t$ only via positive time-delays, and they are explicitly determined by the knowledge of $v$ on $[-\tmax,t_0)$ as well. We keep the index ``$t$'' nevertheless to emphasize that we shall determine $v_Z$ pointwise on $[t_0,t_0+\tmin)$.
  
  As a self-map of $\R^{|S_Z|}$, observe that $\varphi_t$ is continuous. In addition under Assumption~\ref{ass:nonzeno}, we obtain using~\Cref{prop:technical1} and~\Cref{prop:fixed_point_if_contraction} the existence of $\lambda\in[0,1)$ and $u\in\R^{|S_Z|}_{>0}$ such that $\varphi_t$ is $\lambda$-contracting in weighted sup-norm $\Vert\cdot\Vert_u$ (taking the minimum of the operator $\check{C}^{\min}$ defined there with the constant application $w_t$ doesn't affect its contraction rate). As a result, $v_Z(t)$ arises as the unique fixed-point of $\varphi_t$, not necessarily positive but still finite. This shows that $v(t)$ is uniquely determined for all $t$ in $[t_0,t_0+\tmin)$, hence achieving our inductive scheme.

  Now that we have proved the uniqueness of $v(t)$ and in particular $v_Z(t)$, we can present for the latter a more convenient form. Indeed, we know that $v_Z(t)$ satisfies the following implicit relation:
  \begin{equation}\label{eq:vZ_implicit}
     v_Z(t) = \min\bigg(\min_{\sigma\in\SSigma_Z}\Big\{c^{\sigma}_t\big|_{S_Z}+P^{\sigma}\big|_{S_Z\times S_Z}v_Z(t)\Big\}\,,\,w_t\bigg)\,,
  \end{equation}
  where $\SSigma_Z\coloneqq \prod_{i\in S_Z}A_{i,Z}$ and for $\sigma$ in $\SSigma_Z$, $c^{\sigma}_t|_{S_Z}$ (resp.\ $P^{\sigma}|_{S_Z\times S_Z}$) denotes the vector $(c^{\sigma(i)}_t)_{i\in S_Z}$ (resp.\ the matrix $(p^{\sigma(i)}_j)_{i,j\in S_Z}$) -- the use of the restriction bars is actually not correct but we keep them to help visualize the dimensions of the objects at stake.

  Equation~\eqref{eq:vZ_implicit} holds if and only if for all $\sigma$ in $\SSigma_Z$, we have $v_Z(t) \leq c^{\sigma}_t\big|_{S_Z} + P^{\sigma}\big|_{S_Z\times S_Z}$ and $v_Z(t)\leq w_t$, with an equality achieved for each coordinate. Since by Assumption~\ref{ass:nonzeno}, we have $\rho(P^{\sigma}\big|_{S_Z\times S_Z})<1$ for all $\sigma$ of $\SSigma_Z$, the first set of inequalities can be replaced by $v_Z(t) \leq (I-P^{\sigma}|_{S_Z\times S_Z})^{-1}c^{\sigma}_t\big|_{S_Z}$ for all $\sigma$ in $\SSigma_Z$ (indeed for any stochastic matrix $P$ with $\rho(P) < 1$, the matrix $(I-P)^{-1}$ is nonnegative and as a result order-preserving as it can be seen for instance by a series expansion). We therefore obtain the following explicit expression:
  \begin{equation}\label{eq:vZ_explicit}
    v_Z(t) = \min\bigg(\min_{\sigma\in\SSigma_Z}\Big\{\Big(I-P^{\sigma}\big|_{S_Z\times S_Z}\Big)^{-1}c^{\sigma}_t\big|_{S_Z}\Big\}\,,\,w_t\bigg)
  \end{equation}

  If $v$ is the function determined by $v^0\in\Vcal$ and~\eqref{eq:SMDP_finite_horizon}, we easily show that if $v^0$ is càdlàg (resp.\ piecewise-affine, piecewise-constant), then is so is $v$, by supposing that $v$ is indeed càdlàg (resp.\ piecewise-affine, piecewise-constant) on $[-\tmax, t_0)$ for $t_0\geq 0$ and extending this property on $[t_0, t_0+\tmin)$ using equations~\eqref{eq:vNZ_explicit} and~\eqref{eq:vZ_explicit} which preserve the càdlàg character (resp.\ piecewise-affine, piecewise-constant) and the fact that $t\mapsto (c^a_t)_{a\in A_Z}$ and $t\mapsto w_t$ are càdlàg (resp.\ piecewise-affine, piecewise-constant) mappings of $[t_0, t_0+\tmin)$ as well.

  The remark subsequent to the theorem that claims that the discontinuity points of $v$ must be contained in $\mathcal{D}(v^0)+\sum_{a\in A}t^a\N$, (where $\mathcal{D}(v^0)$ is the countable sets of discontinuities of $v^0$) is obtained by a similar induction scheme.
\end{proof}

\begin{proof}[Proof of~\Cref{thm:value_is_S0}]
Following~\Cref{prop:uniquely_determined}, we denote by $w$ the unique function defined on $[-\tmax,\infty)$ such that $w$ is null over $[-\tmax,0)$ and $w$ satifies~\eqref{eq:SMDP_finite_horizon}. \Cref{prop:uniquely_determined} again guarantees that $w$ is piecewise constant and càdlàg, in particular we have a piecewise constant and càdlàg ``choice of argmin in~\eqref{eq:SMDP_finite_horizon}'' mapping $\sigma:[-\tmax,\infty)\to \SSigma$, such that for all $i$ in $S$ and $t\geq 0$ 
\[ w_i(t) = c^{\sigma_i(t)} + \sum_{j\in S}p^{\sigma_i(t)}_j w_j(t-t^{\sigma_i(t)})\,,\]

When $f$ is a strategy in $\Fbb$, we define:
\[ v^{f}\!(i,t) \coloneqq \mathbb{E}_{(i,t)}^{f}\bigg(\sum_{k=0}^{\widehat{N_t}}\widehat{c}_k\bigg) =  \mathbb{E}_{(i,t)}^{f}\bigg(\sum_{k=0}^{\infty}\widehat{c}_k\unbm_{\widehat{s}_k\leq t}\bigg)\]

Our goal is to prove that for all $i$ in $S$ and $t\geq 0$, we have $w_i(t) = \inf_{f\in\Fbb}v^f(i,t)$.
We now introduce a particular strategy $f^*=(f^*_0,f^*_1,f^*_2,\dots)$ in $\Fbb$, by letting 
\[ \textforall k \textin \N, \;\, \textforall h_k \textin H_k,\;\, f^*_k(h_k) \coloneqq \sigma_{i_k}(\max(t-s_k,0))\]

The strategy $f^*$ is pure but not strictly speaking Markovian: it does depend on the history only through the current state and the current total elapsed time since $0$, or equivalently the time remaining to live before hitting the planning horizon $t$. Although $f^*$ is also not stationary, observe it still has \emph{some form} of stationarity~: the choice of actions to play at different epochs varies but does not depend on the epoch number itself, the same horizon-dependent rule being applied at all epochs. Huang and Guo call these strategies ``horizon-relevant stationary'' (see~\cite{huang2011finite}). This fact shall be central to proving that $v^{f^*}$ verifies~\eqref{eq:SMDP_finite_horizon}. Indeed, let $i$ in $S$ and $t\geq 0$. We have  
\[\begin{aligned}
  v^{f^*}\!(i,t)  &= \mathbb{E}_{(i,t)}^{f^*}(\widehat{c}_0\unbm_{0\leq t}) + \mathbb{E}_{(i,t)}^{f^*}\Big(\sum_{k=1}^{\infty}\widehat{c}_k\unbm_{\widehat{s}_k\leq t}\Big) \\
  &=c^{\sigma_i(t)} + \mathbb{E}_{(i,t)}^{f^*}\bigg(\mathbb{E}_{(i,t)}^{f^*}\Big(\sum_{k=1}^{\infty}\widehat{c}_k\unbm_{\widehat{s}_k\leq t}\;\Big|\; \widehat{s}_0, \widehat{i}_0, \widehat{a}_0, \widehat{s}_1, \widehat{i}_1\Big)\bigg)\\
  &=c^{\sigma_i(t)} + \sum_{j\in S}p^{\sigma_i(t)}_j \mathbb{E}_{(i,t)}^{f^*}\Big(\sum_{k=1}^{\infty}\widehat{c}_k\unbm_{\widehat{s}_k\leq t}\;\Big|\; \widehat{a}_0 = \sigma_i(t), \widehat{i}_1 = j \Big)\\
\end{aligned}\]
where the second equality follows from the law of total expectation and the third one from the fact that $\widehat{s}_0$, $\widehat{i}_0$ and $\widehat{s}_1$ are deterministic under the choice of $\widehat{i}_0=i$ and $\widehat{a}_0$.

Consider an admissible random process $(\widehat{s}_0, \widehat{i}_0, \widehat{a}_0, \widehat{s}_1, \widehat{i}_1, \widehat{a}_1, \dots)$  (called ``process A'') generated by the initial state $i\in S$ and the strategy $f^*_t$ and such that $\widehat{i}_1 = j$ for $j\in S$ fixed. Then, the random process $(\widehat{s}_1-t^{\sigma_i(t)}, \widehat{i}_1, \widehat{a}_1, \widehat{s}_2-t^{\sigma_i(t)}, \widehat{i}_2, \widehat{a_2}, \dots) = (\widehat{s}{}'_0, \widehat{i}{}'_0, \widehat{a}{}'_0, \widehat{s}{}'_1, \widehat{i}{}'_1, \widehat{a}{}'_1, \dots)$ (that we may call ``process B'') has the same law as the random process generated by initial state $j$ and the strategy $f^*_{t-t^{\sigma_i(t)}}$. Indeed the $k$-th chosen action in process B is given by $\sigma_{\widehat{i}{}'_k}(t-t^{\sigma_i(t)}-\widehat{s}{}'_k)$, coinciding with $\sigma_{\widehat{i}_{k+1}}(t-\widehat{s}_{k+1})$, the $(k+1)$-th chosen action in process A. Hence

\[\begin{aligned}
  v^{f^*}\!(i,t) &=c^{\sigma_i(t)} + \sum_{j\in S}p^{\sigma_i(t)}_j \mathbb{E}_{(j,t-t^{\sigma_i(t)})}^{f^*}\Big(\sum_{k=1}^{\infty}\widehat{c}{}'_{k-1}\unbm_{t^{\sigma_i(t)}+\widehat{s}{}'_{k-1}\leq t}\Big)\\
  &=c^{\sigma_i(t)} + \sum_{j\in S}p^{\sigma_i(t)}_j \mathbb{E}_{(j,t-t^{\sigma_i(t)})}^{f^*}\Big(\sum_{k=0}^{\infty}\widehat{c}{}'_{k}\unbm_{\widehat{s}{}'_{k}\leq t-t^{\sigma_i(t)}}\Big)\\
  &=c^{\sigma_i(t)} + \sum_{j\in S}p^{\sigma_i(t)}_j v^{f^*}\!(j,t-t^{\sigma_i(t)})\\
\end{aligned}\]

This identity allows us to prove by induction (consider intervals of the form $[t_0,t_0+\tmin)$ again) that $v^{f^*}$ and $w$ that we already know to coincide on $[-\tmax,0)$ actually coincide on $[-\tmax,\infty)$. In particular, $v^{f^*}$, satisfies equations~\eqref{eq:SMDP_finite_horizon}, and $w\geq v^*$.

We now want to conclude that $w = v^{f^*} \leq v^* = \inf_{f\in\Fbb} v^f$.
Using computations similar to those used for showing the recursive equation verified by $v^{f^*}$, we can demonstrate that if $f = (f_0,f_1,f_2,\dots)$ is a randomized and history-dependent strategy, we have for all $i$ in $S$ and $t\geq 0$
\begin{equation}\label{eq:HR_SMDP}
  v^f(i,t) = \sum_{a\in A_i}f_0(a)\bigg(c^a + \sum_{j\in S}p^a_j v^{f'_a}(j,t-t^a)\bigg)\,,\end{equation}
where for $i$ in $S$ and $a$ in $A_i$, $f'_a$ denotes the randomized and history-dependent strategy $(f'_{a,0},f'_{a,1},\dots)$ with for all $k$ in $\N$ and $h_k$ in $H_k$, $f'_{a,k}(h_k) \coloneqq f_{k+1}((i,a,h_k))$. This equation becomes simpler if $f$ is a Markovian strategy, i.e. when $f_k(h_k) = f_k(i_k)$, because in this case all the $(f'_a)_{a\in A}$ coincide with the strategy $f'\coloneqq(f_1,f_2,\dots)$. We then take advantage of the fact that for any randomized history-dependent strategy $f$ in $\Fbb$, there is a randomized Markovian strategy $f_{\mathrm{MR}}$ in $\Fbb_{\mathrm{MR}}$ so that the random processes generated by $f$ and $f_{\mathrm{MR}}$ have the same law (see~\cite[Theorem 5.5.1]{puterman2014markov}), and as a result $\inf_{f\in \Fbb} v^f = \inf_{f \in \Fbb_{\mathrm{MR}}} v^f$.

We now show that for all $i$ in $S$ and $t$ in $[t_0,t_0+\tmin)$, we have $\inf_{f\in\Fbb_{\textrm{MR}}} v^f(i,t)\geq v^{f^*}\!(i,t)$, by supposing that this inequality holds for all $t$ in $[-\tmax, t_0)$. In a similar way of the proof of~\Cref{prop:uniquely_determined},  we know that this heredity step is trivial on states $i$ in $S_{NZ}$, and thus we focus on the restriction $v^f_Z(t)$ of $v^f$ to states of $S_Z$. We reuse the notation $C_{Z,t_0}$ and $w_{t_0}$ of~\Cref{prop:uniquely_determined}, but add the ``$*$'' symbol to indicate that their definition use the particular function $v^{f^*}$ instead of a general $v$. The fact that $f$ is Markovian and~\eqref{eq:HR_SMDP} gives us 
\[  v^f_Z(t) \geq \min\big( \, C^*_{Z,t}(v^{f'}_Z(t))\,,\, w^*_{t}\,\big)\coloneqq \varphi^*_{t}(v^{f'}_Z(t))\]
Denoting by $f^{(n)}$ the strategy $(f_n,f_{n+1},f_{n+2},\dots)$ if $n$ is in $\N$, we immediately obtain for $\varphi^*_t$ is order-preserving that for all $n$
\[ v^f_Z(t) \geq \big(\varphi^*_t\big)^n\big(v^{f^{(n)}}_Z(t)\big)\]

Now denote for short $x_n\coloneqq v^{f^{(n)}}_Z(t)$ and $x^*\coloneqq v^{f^*}_Z(t)$ and $\varphi\coloneqq \varphi^*_t$ , we claim that $\lim_{n\to\infty} \varphi^n(x_n) = x^*$. Indeed as seen in~\Cref{prop:uniquely_determined}, we have $u>0$ such that $\varphi$ is $\lambda$-contracting in weighted sup-norm $\Vert\cdot\Vert_{u}$ and admits $x^*$ as a fixed-point (for it verifies~\eqref{eq:SMDP_finite_horizon} at $t$). Also observe that following~\Cref{lemma:sandwich_N}, we have $\Vert x_n \Vert_{u} \leq M\coloneqq \max(u)\max_{i\in S_Z}\{N^{\scalebox{.72}{[$\lceil t/\tmin\rceil$]}}(i^{(1)})\}\max_{a\in A}\{|c^a|\}$. So we have for $n$ in $\N$:
\[\begin{aligned}
\varepsilon_n &\coloneqq \Vert \varphi^n(x_n)-x^*\Vert_u \\
& \leq \lambda \Vert \varphi^{n-1}(x_n) - x^*\Vert_u \\
& \leq \lambda\Vert \varphi^{n-1}(x_{n-1})-x^*\Vert_u + \lambda\Vert \varphi^{n-1}(x_n)-\varphi^{n-1}(x_{n-1})\Vert_u \\
& \leq \lambda \varepsilon_{n-1} + 2\lambda^{n-1} M\\
& \leq \lambda^n\varepsilon_0 + 2n\lambda^{n-1}M\qquad \text{by immediate induction}
\end{aligned}\]
Since $\lim_{n\to\infty}\varepsilon_n = 0$, we deduce that for all $f$ in $\Fbb_{\mathrm{MR}}$, $v^f_Z(t)\geq v^{f^*}_Z(t)$, which shows correctness our inductive scheme.

We can finally conclude by the inequalities 
\[ v^{f^*}\!(i,t)\leq \inf_{f\in\Fbb_{\mathrm{MR}}} v^f(i,t)= \inf_{f\in\Fbb} v^f(i,t) = v^*(i,t)\leq v^{f^*}\!(i,t)\]
that the value function $v^*$ is given by the unique solution of the equations~\eqref{eq:SMDP_finite_horizon} with null initial condition.
\end{proof}

\begin{proof}[Proof of the fact that $(\Sg_t)_{t\geq 0}$ is a semigroup]
  First notice that we naturally have $\Sg_{0}$ is the identity mapping of $\Vcal$. Then, take nonnegative numbers $t_1$ and $t_2$, a test function $v^0$ in $\Vcal$ and $s$ in $[-\tmax,0)$. The value $(\Sg_{t_2}\circ\Sg_{t_1})(v^0)(s)$ is equal to $v'(t_2+s)$, where $v'$ is the solution of the dynamics equations determined by the initial condition $\Sg_{t_1}(v^0)$, but the latter is given by $r\mapsto v(t_1+r)$ for $r$ in $[-\tmax,0)$ where $v$ is the solution of the dynamics equations determined by $v^0$. Since $v'$ and $r\mapsto v(t_1+r)$ coincide on $[-\tmax,0)$, they are also equal on $[-\tmax,\infty)$ and we obtain $v'(t_2+s)=v(t_1+t_2+s)$. As a result, $\Sg_{t_2}\circ\Sg_{t_1}=\Sg_{t_1+t_2}$.
\end{proof}  

\begin{proof}[Proof of~\Cref{prop:properties_semigroup}]
To show that the evolution semi-group is additively-homogeneous and order-preserving,  we reuse the inductive scheme of the proof of~\Cref{prop:uniquely_determined} and some of the notation defined there. 

\udensdot{We first prove the additively-homogeneous character}. Suppose $v^0$ and $v^0{}'$ are in $\Vcal$, denote $v$ and $v'$ the functions they uniquely determine on $[-\tmax,\infty)$. In the particular case where $v^0{}' = v^0 + \alpha\tilde{\unbf}$ with $\alpha$ in $\R$, if we suppose that $v'(t) = v(t) + \alpha\unbf$ for all $t$ in $[-\tmax, t_0)$, then we easily see that this property still holds on $[t_0, t_0+\tmin)$ for non-Zeno states using explicit expression~\eqref{eq:vNZ_explicit}. On Zeno states, we make use of the explicit expression~\eqref{eq:vZ_explicit}. Following from the non-Zeno part, it is clear that $w_t' = w_t + \alpha\unbf$, and for all $\sigma$ in $\SSigma$, we also have $c^{\sigma}_t{}' = c^{\sigma}_t + \alpha P^{\sigma}|_{S_{NZ}\times S_Z}\unbf$. However for $P^{\sigma}|_{S_{NZ}\times S_Z}\unbf+P^{\sigma}|_{S_{Z}\times S_Z}\unbf = \unbf$ (fundamental property of stochastic matrices), we also have $c^{\sigma}_t{}' = c^{\sigma}_t + \alpha(I-P^{\sigma}|_{S_Z\times S_Z})^{-1}\unbf$, 
 and it follows readily from~\eqref{eq:vZ_explicit} that $v_Z'(t) = v_Z(t) + \alpha\unbf$, which completes the proof of the additively homogenous character.

\udensdot{For the monotonicity}, take $v^0{}' \geq v^0$ and let us suppose that $v'\geq v$ on $[-\tmax, t_0)$. Using explicit expression~\eqref{eq:vNZ_explicit}, we still obtain that this carries along on $[t_0,t_0+\tmin)$ 
for non-Zeno states. On states of $S_Z$, observe that for a fixed $t$ in $[t_0,t_0+\tmin)$ we have $w'_t\geq w_t$ and $c^a_t{}'\geq c^a_t$ for all $a$, hence explicit expression~\eqref{eq:vZ_explicit} provides $v'_Z(t)\geq v_Z(t)$ by order-preserving character of the employed inverses.

\udensdot{The nonexpansiveness} itself is a consequence of both the additively homogeneous and order-preserving characters, as it was shown by Crandall in Tartar in~\cite{crandall}.

It is clear that \udensdot{$\Sg_t$ is continuous for the uniform topology} for all $t\geq 0$ from the nonexpansiveness for the latter amounts to the $1$-Lipschitz property. The \udensdot{continuity with respect to the pointwise} \udensdot{convergence} is a little more subtle. We build on the fact (for instance following from the proof of~\Cref{prop:uniquely_determined} by induction once again) that for all $t\geq 0$ and $s\in[-\tmax,0)$, the number $\Sg_t[v^0](s)$ continuously depend on finitely many values taken by $v^0\in\Vcal$. Indeed, we can write 
\[ \Sg_t[v^0](s) = \psi_{t+s}\Big(\Big\{v^0(\tau)\;\Big|\;\tau\in \Big(t+s-\sum_{a\in A}t^a\N\Big)\cap [-\tmax, 0)\Big\}\Big)\,,\]
the mapping $\psi_{t+s}$ is continuous because it is obtained by taking minimums and linear combinations (including inversions of Cramer systems, see equations~\eqref{eq:vNZ_explicit} and~\eqref{eq:vZ_explicit}), and the number of values of $v^0$ used is finite as the intersection of a discrete countable set (whose points are all isolated) and a bounded interval. We then use the sequential characterization of the continuity to deduce that if $(v^{0}_{(k)})_{k\in\N}$ converge pointwise towards $v^0_{\infty}$ as functions of $\Vcal$, then for all $t\geq 0$ and for all $s$ in $[-\tmax,0)$, the sequence $(\Sg_t[v^0_{(k)}](s))_{k\in\N}$ converges towards $\Sg_t[v^0_{\infty}](s)$. 
\end{proof}

\subsection{Proofs of the results on SSP}

\begin{proof}[Proof of~\Cref{thm:normalization}]
The result follows from reusing the~\eqref{eq:SMDP_finite_horizon} equations and injecting the definition of $\Delta v^*$:
\[\Delta v^*(i,t)  = \min_{a\in A_i} \bigg\{ c^{a} + \sum_{j\in S} p^{a}_{j}\,\Big(\Delta v^*(j, t-t^{a})+\lambda(t-t^a)\Big) \bigg\}-\lambda t\]
Since for all $a$ in $A$, $\sum_{j\in S}p^a_j=1$, the terms $\lambda t$ cancel out and the reduced cost $c^a-\lambda t^a$ appears.
\end{proof}

\begin{proof}[Proof of~\Cref{prop:ass_SSP_is_fluid_regime}]
  We first prove $\ref{prop:ass_SSP_is_fluid_regime:item1}\,\Longleftrightarrow\,\ref{ass:fluid_regime:item1}$. It is easy to see that the accessibility of $j$ from $i$ also amounts to the existence of a particular policy $\sigma$ in $\SSigma$ and an integer $n$ in $\N$ such that $(P^{\sigma})^n_{ij} > 0$. In this sense, the implication $\ref{prop:ass_SSP_is_fluid_regime:item1}\,\implies\,\ref{ass:fluid_regime:item1}$ is direct. Conversely, suppose that $\ref{ass:fluid_regime:item1}$ holds, and let $\mathcal{T}$ be a directed spanning tree with root $0$ of the graph $\mathcal{G}$, whose nodes set is $S$ and whose arcs set contains $(i,j)\in\mathcal{G}^2$ iff there exists $a\in A_i$ such that $p^a_j>0$. Such a spanning tree indeed exists since every node has access to $0$ and it can be constructed using depth-first search. We now construct a proper policy state-by-state: for any $i$ in $S$, there is a path $i\to i'\to\dots\to 0$ in $\mathcal{T}$, we then choose for $\sigma(i)$ any action in $A_i$ such that $p^a_{i'} > 0$. The accessibility property and the fact that state $0$ is absorbing guarantee that $(P^{\sigma}|_{(\SminusZ)\times(\SminusZ)})^{|S|}$ is strictly substochastic, hence $\rho(P^{\sigma}|_{(\SminusZ)\times(\SminusZ)})<1$ and we have $\lim_{n\to\infty}(P^{\sigma}|_{(\SminusZ)\times(\SminusZ)})^n = 0$, so that $\sigma$ is proper, hence the implication $\ref{ass:fluid_regime:item1}\,\implies\,\ref{prop:ass_SSP_is_fluid_regime:item1}$.

  We continue by showing $\ref{ass:fluid_regime:item2}\,\implies\,\ref{prop:ass_SSP_is_fluid_regime:item2}$. 
  Assume $\ref{ass:fluid_regime:item2}$ holds, so that $\underline{\chi}>0$, and consider an improper policy $\sigma$ in $\SSigma$ of the complete Semi-Markov SSP. For $\sigma$ is improper, it admits a final class $F$ different than $\{0\}$. Denoting by $v^{\sigma}(i,t)$ the expected-cost incurred up to time $t\geq 0$ starting in state $i\in S$ by applying the stationary strategy associated with $\sigma$, we have by~\Cref{prop:affine1} that for all $i$ in $F$, $v^{\sigma}(i,t)\sim \langle \mu^{\sigma}_F, c^{\sigma}\rangle/\langle \mu^{\sigma}_F, t^{\sigma}\rangle t$. Since this slope is positive by~\eqref{eq:chi_lower} and Assumption~\ref{ass:fluid_SSP_configuration}, we deduce that applying the policy $\sigma$ yields a positive infinite total cost starting in states of $F$.
  
  To conclude the proof, we show that $\ref{prop:ass_SSP_is_fluid_regime:item1}\wedge\neg\ref{ass:fluid_regime:item2}\,\implies\,\neg\ref{prop:ass_SSP_is_fluid_regime:item2}$, since it brings $\ref{prop:ass_SSP_is_fluid_regime:item1}\wedge\ref{prop:ass_SSP_is_fluid_regime:item2}\,\implies\,\ref{prop:ass_SSP_is_fluid_regime:item1}\wedge\ref{ass:fluid_regime:item2}$.
  If~\ref{ass:fluid_regime:item2} does not hold, it means by~\eqref{eq:chi_lower} that there exists a policy $\sigma_{\scalebox{.6}{$\mathrm{I}$}}$ of $\SSigma$ and a final class $F$ of $\sigma_{\scalebox{.6}{$\mathrm{I}$}}$ such that $\langle \mu^{\sigma}_F, c^{\sigma_{\scalebox{.6}{$\mathrm{I}$}}}\rangle/\langle \mu^{\sigma}_F, t^{\sigma_{\scalebox{.6}{$\mathrm{I}$}}}\rangle \leq 0$. However, if~\ref{prop:ass_SSP_is_fluid_regime:item1} also holds, it means that we have a proper policy $\sigma_{\scalebox{.6}{$\mathrm{P}$}}$ in $\SSigma$. We can therefore introduce the new policy $\sigma$ in $\SSigma$ such that for all $i$ in $F$, $\sigma(i)\coloneqq \sigma_{\scalebox{.6}{$\mathrm{I}$}}(i)$ and for all $i$ in $S\setminus\{0\}$, $\sigma(i)\coloneqq \sigma_{\scalebox{.6}{$\mathrm{P}$}}(i)$. 
  Let us denote $S'\coloneqq S\setminus(\{0\}\cup F)$, then since $P^{\sigma}|_{S'\times S'}$ is extracted from $P^{\sigma_{\scalebox{.6}{$\mathrm{P}$}}}|_{\SminusZ\times\SminusZ}$, it cannot contain any stochastic submatrix, which means that the policy $\sigma$ admits only $\{0\}$ and $F$ as final classes. The fact that $F\neq\{0\}$ is a final class of policy $\sigma$ naturally means that $\sigma$ is improper. Eventually, by~\Cref{prop:affine1}, we know that for all $i$ in $S$, $v^{\sigma}(i,t) {=}_{t\to\infty} \chi_i t + O(1)$, where $\chi_i$ is a nonnegative linear combination of $0$ and $\langle \mu^{\sigma}_F, c^{\sigma_{\scalebox{.6}{$\mathrm{I}$}}}\rangle/\langle \mu^{\sigma}_F, t^{\sigma_{\scalebox{.6}{$\mathrm{I}$}}}\rangle$. Hence, we have exhibited an improper policy $\sigma$ which does not yield a total infinite total cost, starting from any state, so~\ref{prop:ass_SSP_is_fluid_regime:item2} is false.

\end{proof}

We denote by $v^{\sbullet}$ the constant function of $\Vcal_0$ equal to $u^*$.

\begin{lemma}\label{lemma:v_fixedpoint_of_Sg}
  If $v^0$ in $\Vcal_0$ is a fixed-point of $\Sg_t$ for all $t\geq 0$, then $v^0=v^{\sbullet}$. Conversely, $v^{\sbullet}$ is a fixed-point of $\Sg_t$ for all $t\geq 0$.\end{lemma}
\begin{proof}We first prove the first assertion. Suppose we have such a $v^0$, so that 
\[ \forall t\geq 0, \;\; \forall s \in [-\tmax, 0), \quad \Sg_t[v^0](s) = v^0(s)\,.\]
Letting $t = -s$ yields $v^0(s) = v(0)$ for all $s\in[-\tmax,0)$ where $v$ is the function uniquely determined by $v^0$ and~\eqref{eq:SMDP_finite_horizon} (by~\Cref{def:semigroup}). As a result $v^0$ is constant. The same~\eqref{eq:SMDP_finite_horizon} relationship provides its  value since $v(0)$ verifies that for all $i\in S$, 
\[\begin{aligned}v(i,0) &= \min_{a\in A_{i,Z}} \bigg\{ c^{a} + \sum_{j\in S} p^{a}_{j}\,v(j, 0) \bigg\}\wedge \min_{a\not\in A_{i,Z}} \bigg\{ c^{a} + \sum_{j\in S} p^{a}_{j}\,v^0(j,-t^a) \bigg\}\\
  &=\min_{a\in A_i} \bigg\{ c^{a} + \sum_{j\in S} p^{a}_{j}\,v(j, 0) \bigg\}\,,\end{aligned}\]
  where we have reused the fact that $v^0(s)=v(0)$ for $s$ in$[-\tmax, 0)$. Thus $v(0)=T(v(0))$, and since $v$ is null on state $0$, we have $v(0) = u^*$, hence $v^0 = v^{\sbullet}$.

  Conversely, we can reapply the inductive scheme of the proof of~\Cref{prop:uniquely_determined} to show that $\Sg_t[v^{\sbullet}] = v^{\sbullet}$ for all $t\geq 0$. Let us call $v$ the function uniquely determined by $v^{\sbullet}$ and~\eqref{eq:SMDP_finite_horizon}, we suppose that there is $t_0\geq 0$ such that $v(s)=u^*$ for all $s$ in $[-\tmax, t_0)$. If $t\in[t_0,t_0+\tmin)$, we have $v(i,t) = u^*(i)$ for all non-Zeno state $i$ in $S_{NZ}$. By replacing these information in the equations given by~\eqref{eq:SMDP_finite_horizon} for the possibly Zeno states, we see that $v_Z(t)$ is solution of the dynamics if and only $u\coloneqq (u^*|_{NZ},v_Z(t))$ satisfies $u=T(u)$, hence $v_Z(t) = u^*|_Z$, so the induction is shown.
\end{proof}

\begin{lemma}\label{lemma:sandwich_v}
  If there is $v^0\in\Vcal_0$ such that for all $t\geq 0$, we have $\vsb \leq \Sg_t[v^0] \leq v^0$, then $\Sg_t[v^0]$ converges pointwise towards $\vsb$ (and the result also holds if conversely, we have $v^0\leq \Sg_t[v^0] \leq \vsb$).
  \end{lemma}
\begin{proof}Indeed, by monotonicity of the semi-group (\Cref{prop:properties_semigroup}) and by the fixed-point property of $v^{\sbullet}$ (\Cref{lemma:v_fixedpoint_of_Sg}), observe that for all $t,t'\geq 0$, we have $\vsb \leq \Sg_{t+t'}[v^0] \leq \Sg_t[v^0]$. If follows that for all $s\in\tInterval$, $t\mapsto \Sg_t[v^0](s)$ is a non-increasing and bounded from below function. Therefore, $\Sg_t[v^0]$ converges pointwise towards a function $\tilde{v}\in\Vcal_0$. Taking the limit when $t'\to\infty$ in $\Sg_{t+t'}[v^0] = \Sg_t[\Sg_{t'}[v^0]] \leq \Sg_{t'}[v^0]$ yields by pointwise continuity of the semigroup (\Cref{prop:properties_semigroup}) $\Sg_t[\tilde{v}]\leq\tilde{v}$ for all $t\geq 0$. On the other hand, for $\tilde{v}$ is a non-increasing limit we have for all $t,t'\geq 0$ that $\tilde{v}\leq \Sg_{t+t'}[v^0] = \Sg_t[\Sg_{t'}[v^0]]$. Letting $t'\to\infty$ entails $\tilde{v}\leq\Sg_t[\tilde{v}]$ for all $t\geq 0$. As a result, $\tilde{v}$ is a fixed-point of $\Sg_t$ in $\Vcal_0$ for all $t\geq 0$, so we must have $\tilde{v}=\vsb$ by~\Cref{lemma:v_fixedpoint_of_Sg}.\end{proof}

\begin{proof}[Proof of~\Cref{thm:SSSP}]
  Take $v^0$ in $\Vcal_0$.
  We let $\delta\coloneqq \Vert v^0 - \vsb\Vert_{\infty}$ and denote by $\Delta\in\RS$ the vector such that $\Delta_0 = 0$ and $\Delta_i = \delta$ if $i\in \SminusZ$. We denote by $\tilde{\Delta}$ the constant function $s\mapsto \Delta$ of $\Vcal_0$. By definition we have $-\tilde{\Delta}\leq v^0 - v^{\sbullet} \leq \tilde{\Delta}$ (finer than a double inequality with $\delta\tilde{\unbf}$ since $v^0$ and $v^{\sbullet}$ are both null on coordinate $0$).
  
  We introduce the two functions $\underline{v}$ and $\overline{v}$ in $\Vcal_0$ defined by $\underline{v} \coloneqq v^{\sbullet}-\tilde{\Delta}$ and $\overline{v} \coloneqq v^{\sbullet} + \tilde{\Delta}$. It is clear that we have 
  \[ \underline{v} \leq v^0 \leq \overline{v}\qquad\text{and}\qquad \underline{v} \leq v^{\sbullet} \leq \overline{v}\]

  The monotonicity of the semi-group for all $t$ (\Cref{prop:properties_semigroup}) and the fixed-point property of $v^{\sbullet}$ (\Cref{lemma:v_fixedpoint_of_Sg}) lead to
  \[ \Sg_t[\underline{v}] \leq \Sg_t[v^0] \leq \Sg_t[\overline{v}]\qquad\text{and}\qquad \Sg_t[\underline{v}] \leq v^{\sbullet} \leq \Sg_t[\overline{v}]\]

  By non-expansiveness of the evolution semi-group (\Cref{prop:properties_semigroup}), we have 
  \[ \Vert \Sg_t[v^{\sbullet}]-\Sg_t[\underline{v}]\Vert_{\infty} \leq \Vert v^{\sbullet} - \underline{v}\Vert_{\infty}\]
  which we can rewrite as follows using the monotonicity of the semi-group (we recall that $\Vcal_0$ is stable under $\Sg_t$ for all $t$): 
  \[ \tilde{\zerobf}\leq \Sg_t[v^{\sbullet}]-\Sg_t[\underline{v}] \leq \tilde{\Delta}\]
  By the fixed-point property once again, the right inequality above provides $\underline{v}=v^{\sbullet}-\tilde{\Delta}\leq \Sg_t[\underline{v}]$.

  Conversely, we also have by similar arguments that
  \[ \Vert \Sg_t[v^{\sbullet}]-\Sg_t[\overline{v}]\Vert_{\infty} \leq \Vert v^{\sbullet} - \overline{v}\Vert_{\infty}\]
  which leads to the double inequality
  \[ -\tilde{\Delta}\leq \Sg_t[v^{\sbullet}]-\Sg_t[\overline{v}] \leq \tilde{\zerobf}\]
  where we retain the left inequality to obtain $\Sg_t[\overline{v}]\leq v^{\sbullet}+\tilde{\Delta} = \overline{v}$.

  We have satisfactorily shown that for all $t\geq 0$, we have 
  \[ \underline{v}\leq \Sg_t[\underline{v}] \leq v^{\sbullet} \leq \Sg_t[\overline{v}]\leq\overline{v}\,\]
  we can then invoke~\Cref{lemma:sandwich_v} for both $\underline{v}$ and $\overline{v}$, and we obtain that the limits of $\Sg_t[\underline{v}]$ and $\Sg_t[\overline{v}]$ when $t$ tends to $\infty$ do exist and are equal to $v^{\sbullet}$. Using the bounds stated on $\Sg_t[v^0]$, we conclude that $\lim_{t\to\infty} \Sg_t[v^0] = v^{\sbullet}$.
  \end{proof}
  
  \begin{proof}[Proof of~\Cref{prop:uniquely_determined_heaviside}]
    The first part of the proposition follows the same lines as~\Cref{prop:uniquely_determined}; the key ingredient that carries over is that the solution $v$ of~\eqref{eq:SMDP_finite_horizon} at $t$ is uniquely determined by the values of $v$ on $[-\tmax, t-\tmin]$.

    The show the second part, we denote by $v$ the complete trajectory of the $\lambda$-sink SMDP defined by~\Cref{prop:uniquely_determined_heaviside}. Supposing as in the beginning of~\Cref{ssec:SSP} that actions giving access to state $0$ are equipped with zero sojourn time, we have for $t\leq \underline{t}$ and $s$ in $[-\tmax,0)$ that $v(t+s) = w(t+s)$ (the affine stationary regime pursues along). For $t > \underline{t}$ and $s$ in $[-\tmax,0)$ such that $t+s \geq \underline{t}$ however, we have
    \[ \begin{aligned} v(t+s) & = \Sg_{t-\underline{t}}\big[r\mapsto w(\underline{t}+r)+M\unbm_{\{0\}}\big](s)\\ 
      & = \Sg_{t-\underline{t}}\big[ r\mapsto w(r) + M\unbm_{\{0\}}\big](s) + \lambda \underline{t}\unbf\end{aligned}\,,\]
    where we have used the property $w(r+\underline{t}) = \lambda(r+\underline{t})\unbf + u^* = w(r) +\lambda \underline{t}\unbf$ and the additively homogeneous character of the evolution semigroup (see~\Cref{prop:properties_semigroup}). The latter property can be used to convert the ``positive step'' arising on state $0$ at $t=\underline{t}$ to ``negative steps'' on states of $\SminusZ$:
    \[ v(t+s) = \Sg_{t-\underline{t}}\big[ r\mapsto w(r) - M\unbm_{\SminusZ}\big](s) + (\lambda \underline{t}+M)\unbf\,.\]
    The initial condition function in the last equation is in the space $\Vcal_{\lambda}$. We can verify using~\Cref{coro:value_cv} that we indeed have
    \[v(t+s)\underset{t\to\infty}{=}\lambda(t-\underline{t})\unbf + u^* + (\lambda \underline{t}+M)\unbf + o(1) = \lambda t \unbf + u^* + M\unbf + o(1)\,\]
    where we retrieve that the whole trajectory follows the new input up to a delay vector.
  \end{proof}

  \begin{proof}[Proof of~\Cref{thm:reduction_SSP_2}]
    The proof of~\Cref{prop:uniquely_determined_heaviside} shows that studying the solutions of the dynamics of a $\lambda$-sink SMDP initialized from the function $s\mapsto w(s)-M\unbm_{\SminusZ}$ of $\Vcal_{\lambda}$ allows one to characterize the catch-up of a step on the input; equivalently according to~\Cref{thm:normalization}, it also amounts to studying the dynamics of the reduced-costs SMDP in SSP configuration, associated with evolution semigroup $(\Sg^{\Delta})_{t\geq 0})$, initialized from the function $s\mapsto u^* - M\unbm_{\SminusZ}$ of $\Vcal_0$.
  \end{proof}

  \subsection{Proofs of~\Cref{sec:hierarchical}}

  \begin{proposition}\label{prop:monotone_evolution_from_udot}
    Suppose Assumptions~\ref{ass:nonzeno} and~\ref{ass:fluid_regime} hold. 
    Let $r\in\R$ and $v^0$ be the particular function of $\Vcal_0$ defined by 
    \[ v^0(s)=u^*+r\unbm_{\SminusZ} \quad\textforall s \textin [-\tmax,0)\,.\]
    We denote by $v$ the unique function of $[-\tmax,\infty)\to \R^{|S|}$ determined by $v^0$ and the dynamic programming equations~\eqref{eq:SMDP_finite_horizon} of the reduced-costs SMDP in SSP configuration. Then:
    \begin{enumerate}[label=(\roman*)]
      \item\label{prop:monotone_evolution_from_udot:item1} if $r>0$, then $v$ is non-increasing and $v(t)\geq u^*$ for all $t\geq 0$,
      \item\label{prop:monotone_evolution_from_udot:item2} if $r<0$, then $v$ is non-decreasing and $v(t)\leq u^*$ for all $t\geq 0$.
    \end{enumerate}
    \end{proposition}
  \begin{proof}[Proof of~\Cref{prop:monotone_evolution_from_udot}]
    This result actually follows from the proof of~\Cref{thm:SSSP} and~\Cref{lemma:sandwich_v}, with $\delta=|r|$ and item~\ref{prop:monotone_evolution_from_udot:item1} (resp.\ item~\ref{prop:monotone_evolution_from_udot:item2}) corresponding to the study of $\overline{v}$ (resp.\ $\underline{v}$).
  \end{proof}

  \begin{proof}[Proof of~\Cref{thm:finite_time_theoretical}]
    We sketch the proof idea in the MDP case; this entails no loss of generality since in large horizons and thus when taking the sojourn times into account in an SMDP become less and less relevant, it is optimal to play in the SMDP the same actions that in the corresponding unit-time MDP.
Let us start by the special case in which
$u^*=0$, $A_i=A_i^*$ and $c^a=0$ for all $i$ and $a\in A_i^*$. 
Then, the implication \ref{thm:finite_time_theoretical:item1} $\implies$ \ref{thm:finite_time_theoretical:item3} above becomes a consequence of the following lemma.

Let $T:\R^S\to \R^S$ be such that $T_i(x)=\min_{a\in A_i^*}p^a x$.
  If there exists some $k$ such that $T^{k}(-\unbf) = 0$, then,
  there exists a partial ordering $(\leq)$ of $S$ such that
  for all $i\in[n]$ and $a\in A_i^*$, $p^a_j>0 \implies j<i$.

  To see this, consider the matrix $P$ whose $i$th row is given by
$P_i=\sum_{a\in A_i^*}p^a$,
let $N:=\max_{i\in S}|A_i^*|$, and observe that for all $x\leq 0$,
$T(x) \leq N^{-1} Px\leq 0$. Then
we deduce that $0=T^k(-\unbf) \leq -N^{-k}P^k\unbf =0$,
and since $P$ is nonnegative, this entails that $P^k=0$.
Then, the relation $\leq$ defined by $j\leq i$ if $P_{ij}>0$
must be acyclic, and so, its transitive closure yields a partial order,
satisfying the requirements of the lemma.
Next,~\cite[Lemma~6.4]{spectral} shows that if $v^0$ lies in a sufficiently
small neighborhood of $u^*$, then only proper optimal actions
are active in the dynamics~\eqref{eq:SMDP_finite_horizon},
allowing us to reduce to the special case of the previous lemma.
Finally, the equivalence between
\ref{thm:finite_time_theoretical:item2} and \ref{thm:finite_time_theoretical:item3} is deduced from a general property of non-linear spectral radii~\cite[Prop.~8.1]{akian2011collatz}, whereas the implication \ref{thm:finite_time_theoretical:item3} $\Rightarrow$ \ref{thm:finite_time_theoretical:item1} is elementary.
\end{proof}

  \begin{proof}[Proof of~\Cref{lemma:accessible_from_i}]
    Remember that relatively to the graph $\mathcal{G}_S$, the state $0$ is accessible from any other state, according to the assumption~\ref{ass:hierarchy1}, and in addition, an accessibility walk can be achieved using only actions in $\bigcup_{i\in S}A_i^-$. 

    The subgraph $\mathcal{G}^{(k)}_S$ represents only improper policies by assumption~\ref{ass:hierarchy2}. Thus, there is a state $j$ that has not access to $0$. This entails that any path connecting $j$ to $0$ in the original graph $\mathcal{G}_S$ (possibly obtained by only descending actions) had to go through state $k$ (the only state from which actions have been ruled out), otherwise these paths would still be available. It follows that $k$ itself has not access to $0$, otherwise $j$ could reach $0$ again, which proves~\ref{lemma:accessible_from_i:i}. Property~\ref{lemma:accessible_from_i:ii} follows since if $j$ is a state accessible from $k$ in $\mathcal{G}^{(k)}_S$, then $j$ has no longer access to $0$ in $\mathcal{G}^{(k)}_S$, and the above argument shows that $j$ has access to $k$.
  \end{proof}

  \begin{proof}[Proof of~\Cref{lemma:reduced_game}]
    Let $\sigma$ be a policy of \[\SSigma_i\coloneqq A_i^+\times\prod_{j\in \CFC{i}\setminus\{i\}}A_j\]
    that realizes the minimal average-cost vector $\chi(\widetilde{\Sg}^{(i)})$ of $\widetilde{\Sg}^{(i)}$, and let $F$ be a final class associated with $\sigma$ included in $\CFC{i}$ such that $\langle \mu^{\sigma}_{F}, r^{\sigma}\rangle/\langle \mu^{\sigma}_F, t^{\sigma}\rangle=\chi^{(i)}\coloneqq\min_{j\in\CFC{i}}\chi(\widetilde{\Sg}^{(i)})_j$. %
    
    Observe that $\chi^{(i)} \geq \underline{\chi}$ by definition of the latter in~\eqref{eq:chi_lower}. If for all $j\in\CFC{i}$, we have $\phi^{\sigma}_{F,j}=1$ (so that $F$ is the only final class of $\sigma$ in $\CFC{i}$), we directly obtain $\chi(\widetilde{\Sg}^{(i)}) = \chi^{(i)}\unbf$. Otherwise, since $\CFC{i}$ is strongly connected, we may construct a policy $\sigma'\in\SSigma_i$ (identical to $\sigma$ on $F$) such that $\sigma'$ admits $F$ as its unique final class and as a consequence $\chi^{(i)}\unbf$ as optimal cost vector. By minimality of $\chi(\widetilde{\Sg}^{(i)})$, we obtain $\chi(\widetilde{\Sg}^{(i)}) \leq \chi^{(i)}\unbf$ and it follows that $\chi(\widetilde{\Sg}^{(i)}) = \chi^{(i)}\unbf$.
  \end{proof}
  
  \begin{lemma}\label{lemma:theta_hierarchy}Let $i$ in $S$ such that $A_i^+\neq\emptyset$. Then, for all $j$ in $\CFC{i}$, we have $\theta'_j\geq\theta'_i$.
  \end{lemma}
  
  \begin{proof}[Proof of~\Cref{lemma:theta_hierarchy}]
    Indeed for such a $j$, \eqref{eq:SMDP_finite_horizon} implies:
    \[ \min_{a\in A_j^-}\bigg\{ c^a + \sum_{k\in S} p^a_k\, v(k, \theta_j-t^a)\bigg\} \geq v(j,\theta'_j) = u^*(j) \]
    However by monotonicity of $v$ which is non-decreasing and tends towards $u^*$, we have
    \[\begin{aligned} \min_{a\in A_j^-}\bigg\{ c^a + \sum_{k\in S} p^a_k\, v(k, \theta'_j-t^a)\bigg\} &\leq \min_{a\in A_j^-}\bigg\{ c^a + \sum_{k\in S} p^a_k\, u^*(k)\bigg\} \\ &\leq u^*(j)\end{aligned} \]
    again using the fixed-point property in~\eqref{eq:fixed_point_of_T} verified by $u^*$.
    As a result we obtain the equality 
    \[ \min_{a\in A_j^-}\bigg\{ c^a + \sum_{k\in S} p^a_k\, v(k, \theta'_j-t^a)\bigg\} = \min_{a\in A_j^-}\bigg\{ c^a + \sum_{k\in S} p^a_k\, u^*(k)\bigg\} \]
    Invoking monotonicity still, it is easily seen that there exists a particular $a$ in $A_j^-$ in both argmins above, i.e. such that 
    \[ c^a + \sum_{k\in S} p^a_k\, v(k, \theta'_j-t^a) = c^a + \sum_{k\in S} p^a_k\, u^*(k) = u^*(j)\]
    (otherwise we could find an action in the right-handside argmin and not in the left-handside argmin, which would contradict non-decreasingness of $v$). The bounds on $v$ are satured, hence
    \[\forall k \in \supp(a),\quad v(k,\theta'_j-t^a) = u^*(k)\]
    We deduce that for all $k$ in $\supp(a)$, we have $\theta'_k\leq \theta'_j-t^a\leq \theta'_j$. Recall in addition that since $a$ is in $A_j^-$, $\supp(a)$ in included in $\{k\in S\;|\; k < j\}$, and if $j\neq i$, then $\supp(a)$ is also included in $\CFC{i}$ (indeed according to~\Cref{lemma:accessible_from_i}, going through state $i$ \emph{via} a descending action is the only manner of leaving $\CFC{i}$). Therefore, if $i$ is in $\supp(a)$, it is clear that $\theta'_i\leq \theta'_j$, otherwise it suffices to reapply this reasoning (finitely many times) for all $k$ in $\supp(a)$ until eventually hitting $i$.
  \end{proof}
  
  \begin{proof}[Proof of~\Cref{thm:bound_SSP}]
    We prove the result by showing by induction that the property $\mathcal{P}(i)$: ``the bound on $\theta'_i$ in the theorem is valid%
    '' holds for all state $i$ in $S$. To this purpose, it suffices to show that $\mathcal{P}(0)$ is true (initialization), and that if $\mathcal{P}(j)$ holds for all $j< i$, then $\mathcal{P}(i)$ holds as well (heredity). Indeed, although not all states are comparable relatively to the partial ordering $(\leq)$, the state $0$ is a minimal element and any state can reach $0$ using only descending actions (following from Assumption~\ref{ass:hierarchy1}); finally  all states are comparable within such a path. 
    
    It is immediate to obtain that \udensdot{$\mathcal{P}(0)$ is true}, i.e. $\theta'_0=0$, since for all $t\geq 0$, we have $v(0,t) = 0$ (recall that $\Vcal_0$ is stable under evolution equations associated with SSP configurations). %

    In what follows, we let $i\in \SminusZ$ and suppose that $\mathcal{P}(j)$ is true for all $j\in S$ such that $j<i$. Cases are distinguished based on the existence or absence of ascending actions from $i$.

    \udensdot{Let $i\in \SminusZ$ such that $A_i^+ = \emptyset$}, so that by~\eqref{eq:SMDP_finite_horizon}, for all $t\geq 0$,
    \[ v(i,t) = \min_{a\in A_i^-}\bigg\{c^a +\sum_{j\in S} p^a_j\,  v(j,t-t^a) \bigg\}.\] 

    Let us take $t=\max_{a\in A_i^-}\big\{ t^a+\max_{j \in \supp(a)} \theta'_j\big\}$.
    For all $a$ in $A_i^-$, we have for all $j$ in $\supp(a)$ that $v(j,t-t^a)=u^*(j)$, by monotonicity of $v$. As a result, we have 
    \[ v(i,t) = \min_{a\in A_i^-}\bigg\{c^a + \sum_{j\in S}p^a_j\,u^*(j)\bigg\}\,\]
    where we recognize the righthand side of~\eqref{eq:fixed_point_of_T}, which provides $v(i,t)=u^*(i)$. As a result we must have by minimality $\theta'_i \leq \max_{a\in A_i^-}\big\{ t^a+\max_{j \in \supp(a)} \theta'_j\big\}$.

    \udensdot{Let us now take $i\in \SminusZ$ such that $A_i^+ \neq \emptyset$}.
    Denote for short 
    \begin{equation}\label{eq:def_Theta_i}\Theta_i\coloneqq \max_{a\in A_i^-}\big\{ t^a+\max_{j \in \supp(a)} \theta'_j\big\}\,.\end{equation}
    Since the fixed-point in~\eqref{eq:fixed_point_of_T} is achieved by some proper policy which by Assumption~\ref{ass:hierarchy2} makes use of only descending actions, we have (by an argument similar as the previous case) for all $t\geq\Theta_i$:
    \[ v(i,t) = \min\bigg( u^*(i)\,,\, \min_{a\in A_i^+}\bigg\{ c^a + \sum_{j\in S} p^a_j\, v(j, t-t^a)\bigg\}\bigg)\,.\]
    In addition, still by monotonicity of $v$, we know that the minimum in the above equation must be attained by the rightmost term if $t < \theta'_i$. Suppose that $\theta'_i > \Theta_i$ (the opposite leads to the same upper bound $\theta'_i\leq \Theta_i$ as before).
    
    For all $t$ such that $\Theta_i\leq t < \theta'_i$, only actions of $A_i^+$ are excerced from state $i$ to determine $v(t)$ based on the values of $v$ at prior times, and as a result the evolution of $v$ satisfies %
    \[ \forall s \in [-\tmax,0), \quad v(t+s) = \Sg^{(i)}_{t-\Theta_i}\big[v(\cdot + \Theta_i)\big](s)\]
    As shown in Lemma~\ref{lemma:reduced_game}, it is necessary and sufficient to focus on states of $\CFC{i}$ to analyze the evolution of $v(i,t)$, since this class is fully autonomous under the action of $(\Sg^{(i)}_t)_{t\geq 0}$. In particular, reusing the notation of this lemma, we have 
    \begin{equation}\label{eq:dyn_i_plus} \forall s \in [-\tmax,0), \quad v(t+s)\big|_{\CFC{i}} = \widetilde{\Sg}^{(i)}_{t-\Theta_i}\big[v^0\big](s)\,.\end{equation}
    where $v^0$ is the function of $[-\tmax,0)$ to $\R^{|\CFC{i}|}$ such that 
    \[ \textforall s \textin [-\tmax,0),\quad v^0(s)\coloneqq v(s+\Theta_i)\big|_{\CFC{i}}\,.\]  

    Invoking monotonicity of $v$, we can bound $v^0$ by below using our initial condition:
    \[ \forall s\in[-\tmax,0),\quad v^0(s) \geq -\minMi\unbf + u^*\big|_{\CFC{i}} \]
    Using the fact given by~\Cref{lemma:reduced_game} and Assumption~\ref{ass:fluid_SSP_configuration} that $\chi^{(i)}>0$, we can also bound $v^0$ by below by the Kohlberg affine invariant function of $\widetilde{S}^{(i)}$, up to shifting it by a constant. We can write for all $s$ in $[-\tmax,0)$: 
    \[ v^0(s) \geq s\chi^{(i)}\unbf + h^{(i)} -\underbrace{\left(\minMi+\max_{k\in\CFC{i}}\big\{h^{(i)}_k-u^*(k)\big\}\right)}_{\alpha}\unbf \]
    Unravelling the dynamics by applying the operator $\widetilde{\Sg}^{(i)}_{t-\Theta_i}$ as done in~\eqref{eq:dyn_i_plus} has just the effect of a shift in time of magnitude $t-\Theta_i$ for this last affine function, owing to~\Cref{prop:affine1} and to the additively homogeneous character of the evolution semi-group (\Cref{prop:properties_semigroup}). Building on the order-preserving character of the latter (\Cref{prop:properties_semigroup} as well), we obtain from~\eqref{eq:dyn_i_plus} that for all $s$ in $[-\tmax, 0)$:
    \begin{equation}\label{eq:dyn_i_plus_2}
      v(t+s)\big|_{\CFC{i}} \geq (t+s-\Theta_i)\chi^{(i)}\unbf + h^{(i)} -\alpha\unbf\,.
    \end{equation}
    Now suppose by contradiction that $\theta'_i > \Theta_i+M/\chi^{(i)}$ 
    so that for $\varepsilon>0$ sufficiently small, we can let $t=\Theta_i+M/\chi^{(i)}+\varepsilon$ in~\eqref{eq:dyn_i_plus} and then evaluate the result in $s=-\varepsilon$, to get \textit{via} \eqref{eq:dyn_i_plus_2}
    \begin{equation}\label{eq:dyn_i_plus_3} v\big(\Theta_i+M/\chi^{(i)}\big)\big|_{\CFC{i}} \geq h^{(i)} - \max_{k\in\CFC{i}}\big\{h^{(i)}_k-u^*(k)\big\}\unbf\,.\end{equation}
    But by denoting $j$ a state of $\CFC{i}$ such that $\max_{k\in\CFC{i}}\big\{h^{(i)}_k-u^*(k)\big\}=h^{(i)}_j-u^*(j)$, we are left with 
    \[ v\big(j,\Theta_i+M/\chi^{(i)}\big) \geq u^*(j)\,,\]
    from which we conclude that $\theta'_j\leq \Theta_i+M/\chi^{(i)}$, however we know from~\Cref{lemma:theta_hierarchy} that $\theta'_i\leq\theta'_j$, which contradicts our hypothesis $\theta'_i> \Theta_i+M/\chi^{(i)}$. We have therefore proved that $\theta'_i \leq \Theta_i+M/\chi^{(i)}$, which is the bound announced by the theorem 
    which achieves to show that $\mathcal{P}(i)$ is true and completes the induction.
    \end{proof}

\end{document}